\documentclass[11pt, a4paper, twoside]{article}

\usepackage[left=1.3in, right=1.3in, top=1.3in, bottom=1in, includefoot, headheight=13.6pt]{geometry}

\usepackage{fancyhdr}

\usepackage{sectsty}
\sectionfont{\centering\Large\mdseries\scshape}

\usepackage[main=english, ngerman]{babel}
\usepackage[T1]{fontenc}
\usepackage{microtype}

\usepackage{graphicx}
\usepackage{subcaption}
\usepackage{adjustbox}
\graphicspath{
    {figures/example_decomposition/}
    {figures/example_properties/}
    {figures/stable_curves_g2/}
    {figures/stable_curves_g3/}
}

\usepackage{array}
\usepackage{booktabs}
\usepackage{cellspace}
\usepackage{mathtools, amssymb, amsfonts}
\usepackage[mathscr]{euscript}
\usepackage{needspace}

\usepackage{enumitem}

\usepackage{xcolor}
\definecolor{myblue}{RGB}{38, 115, 227} 
\definecolor{myred}{RGB}{186, 0, 38}

\usepackage[colorlinks,breaklinks=true]{hyperref}
\hypersetup{
    bookmarksnumbered,
    pdfstartview={FitH},
    citecolor={blue},
    linkcolor={black},
    urlcolor={black},
    pdfpagemode={UseOutlines}
}

\usepackage[amsmath,hyperref,thmmarks]{ntheorem}
\usepackage{etoolbox}

\newtheorem{Xbase}{X}[section]

\theoremstyle{plain}
\theorembodyfont{\slshape}

\newtheorem{theorem}[Xbase]{Theorem}
\newtheorem*{theoremstar}{Theorem}
\newtheorem{corollary}[Xbase]{Corollary}
\newtheorem{lemma}[Xbase]{Lemma}
\newtheorem{proposition}[Xbase]{Proposition}

\theorembodyfont{\upshape}

\newtheorem{definition}[Xbase]{Definition}
\newtheorem{example}[Xbase]{Example}
\newtheorem{remark}[Xbase]{Remark}
\newtheorem{problem}[Xbase]{Problem}

\newtheorem{setup}[Xbase]{Setup}

\newtheorem*{claimstar}{Claim}

\newtheorem{summary}[Xbase]{Summary}

\newtheorem{notation}[Xbase]{Notation}

\theoremstyle{nonumberplain}
\theoremheaderfont{\itshape}
\theorembodyfont{\normalfont}
\theoremseparator{.}
\theoremsymbol{\ensuremath{\Box}}

\newtheorem{proof}{Proof}

\theoremsymbol{}
\theoremstyle{plain}
\theoremheaderfont{\scshape}
\theorembodyfont{\upshape}

\newtheorem{case}{Case}
\newtheorem{subcase}{Subcase}[case]

\AfterEndEnvironment{proof}{\setcounter{case}{0}\setcounter{step}{0}}
\AfterEndEnvironment{sproof}{\setcounter{case}{0}\setcounter{step}{0}}

\newcommand{\EquivListLabel}{(\roman*)}
\newcommand{\StatementListLabel}{(\alph*)}
\newcommand{\DefListLabel}{(\arabic*)}

\newlist{prooflist}{enumerate}{1}
\setlist[prooflist]{
    label=\StatementListLabel,
    wide,
    labelindent=0pt
}

\newcommand{\ZZ}{\mathbb{Z}}

\newcommand{\CC}{\mathbb{C}}

\newcommand{\PP}{\mathbb{P}}

\newcommand{\I}{\mathcal{I}}
\newcommand{\LL}{\mathcal{L}}
\newcommand{\OO}{\mathcal{O}}

\newcommand{\X}{\mathcal{X}}
\newcommand{\Y}{\mathcal{Y}}

\newcommand{\p}{\mathfrak{p}}

\newcommand{\mf}[1]{\mathfrak{#1}}

\newcommand{\Fs}{\mathscr{F}}

\DeclareMathOperator{\Aut}{Aut}
\DeclareMathOperator{\CaCl}{CaCl}
\DeclareMathOperator{\Char}{char}
\DeclareMathOperator{\Cl}{Cl}
\DeclareMathOperator{\coker}{coker}
\DeclareMathOperator{\Frac}{Frac}
\DeclareMathOperator{\Hom}{Hom}

\DeclareMathOperator{\SheafExt}{\mathcal{E}\!xt}
\DeclareMathOperator{\id}{id}
\DeclareMathOperator{\im}{im}
\DeclareMathOperator{\Pic}{Pic}
\DeclareMathOperator{\Proj}{Proj}
\DeclareMathOperator{\Spec}{Spec}
\DeclareMathOperator{\Supp}{Supp}
\DeclareMathOperator{\Trace}{Tr}
\DeclareMathOperator{\Res}{Res}
\DeclareMathOperator{\tr}{tr}
\DeclareMathOperator{\Ann}{ann}
\DeclareMathOperator{\Div}{Div}
\DeclareMathOperator{\Graph}{Graph}

\newcommand{\oo}[1]{\overline{#1}}
\newcommand{\ssep}{\mid}

\usepackage{tikz}
\usetikzlibrary{hobby, intersections, calc}
\input{curve_library.tex}

\title{\Huge Semistable Reduction of Plane Quartics}
\author{\Large \scshape Max Schwegele\thanks{This article is the author's master's thesis, completed at Universität Ulm under the supervision of Stefan Wewers, which led to the articles \cite{hyperelliptic} and \cite{quarticpaper}.}}
\date{} 

\begin{document}

\maketitle

\begin{abstract}
    The Stable Reduction Theorem guarantees that any smooth, projective, geometrically irreducible curve of genus $g \geq 2$ over a discretely valued field admits a unique stable model after a finite field extension. Computing this model is a central problem in arithmetic geometry. For non-hyperelliptic genus $3$ curves, which are canonically embedded as plane quartics, methods like admissible reduction become challenging in small residue characteristics. This thesis establishes a precise connection between the abstractly defined stable model and computationally accessible GIT-stable plane models. We prove that a GIT-stable plane model of a smooth plane quartic exists if and only if its stable reduction is non-hyperelliptic. When this condition holds, we show that the stable model is the unique minimal semistable model that dominates the GIT-stable model. The corresponding domination morphism is geometrically explicit: it contracts the $1$-tails of the stable reduction to cusps on the special fiber of the GIT-stable model and is an immersion elsewhere. This result provides a geometric framework for computing the stable model by first finding a GIT-stable model and then resolving its cuspidal singularities.
\end{abstract}

\tableofcontents

\section*{Introduction}
\markboth{Introduction}{Introduction}
\addcontentsline{toc}{section}{Introduction}

The moduli space of smooth projective curves of genus $g \ge 2$ is a central object of study, but it lacks the crucial property of being compact. A groundbreaking achievement by Deligne and Mumford was to construct its canonical compactification by enlarging the space to include certain singular curves, which they termed ``stable'' \cite{deligne-mumford}. The allowed singularities are the mildest possible: ordinary double points (nodes). To ensure the finiteness of automorphisms, an additional condition is imposed: any smooth rational component must intersect the rest of the curve at three or more points.

A cornerstone of this theory is the Stable Reduction Theorem. In our context, it states that if $K$ is a field complete with respect to a discrete valuation $v_K$ (such as a $p$-adic number field), with ring of integers $\OO_K$ and residue field $k$, then any smooth, absolutely irreducible curve $X/K$ of genus $g \ge 2$ admits a model $\X$ over a finite extension $L/K$ whose special fiber $\X_s \eqqcolon \oo{X}$ is a stable curve. This stable model is unique, and its special fiber $\oo{X}$ is called \emph{the} stable reduction of $X$. The geometry of this reduction encodes deep arithmetic information about the original curve and its Jacobian, such as its conductor (see, for example, \cite{bouw2017computing}). Consequently, developing effective methods for computing the stable model is a highly motivating challenge in arithmetic geometry.

A powerful technique for computing the stable reduction is the theory of \emph{admissible reduction} (for background, see \cite{liu1999models}, \cite{liu2006stable} and \cite[p. 12]{wewers2024semistable}). This method computes the stable model of a curve $X$ by realizing it as a cover $\varphi\colon X \to \PP^1_K$ of degree $n$. The theory works if the residue characteristic $p$ does not divide the order of the monodromy group of the cover; this holds, for instance, if $p$ is larger than the degree. This has been successfully applied to superelliptic curves \cite{bouw2017computing} and the important sub-loci of hyperelliptic curves (covers of degree $2$) and genus $3$ Picard curves (covers of degree $3$) \cite{bouw2020conductor}. The method can also be generalized to curves with a non-trivial automorphism group; for example, a complete answer for the stable reduction of Ciani curves of genus $3$, which admit a faithful action of the Klein group $C_2 \times C_2$, is given in \cite{bouw2021reduction}. If $p = n$, one can also use the methods developed in \cite{ossen2024semistable, ossen2025semistable}. However, there are still critical cases. One is tempted to choose $n$ as small as possible, i.e., equal to the gonality of the curve. The smallest case where this method is not sufficient is for non-hyperelliptic curves of genus $3$ when $p=2$. This provides the primary motivation for this thesis, which focuses on the stable reduction of non-hyperelliptic curves of genus $3$.

To tackle this specific class of curves, we first recall their concrete geometric realization. A non-hyperelliptic curve of genus $3$ is canonically embedded as a smooth plane quartic in $\PP^2_K$. Such a curve $X$ can be realized as a degree-4 cover of $\PP^1_K$ in many ways, for instance by projecting from a point not on the curve. If $X$ has a $K$-rational point, we can even realize it as a degree-3 cover. For $p > 3$, we can therefore apply the method of admissible reduction, which is likely the most efficient way to compute the stable reduction of $X$. For $p=3$, one could use the results of \cite{ossen2024semistable} after passing to a suitable finite extension, but this may become very inefficient if the required extension is large. This motivates finding a method that works uniformly for all characteristics.

Our approach is to connect the geometric notion of Deligne-Mumford stability with the algebraic framework of Geometric Invariant Theory (GIT), developed by Mumford with the aim of explicitly constructing moduli spaces as quotients of parameterizing schemes \cite{mumford1994geometric}. In our context, we consider the space of ternary quartics under the action of $\mathrm{SL}_3$. A quartic over the residue field is GIT-stable if its only singularities are ordinary nodes ($A_1$) and cusps ($A_2$); the notion of GIT-semistability is slightly more relaxed (see Definition \ref{def:git-stability}). We apply these notions to models of our curve $X$: a plane model $\X_0 \subset \PP^2_{\OO_K}$ is called GIT-(semi)stable if its special fiber has this property. It is known that after a suitable field extension, $X$ always admits a GIT-semistable plane model; such a field extension and models are computable for all primes $p$ using the results of \cite{stern2025models}. A GIT-stable model, however, does not have to exist, but if one exists, it is the unique GIT-semistable model up to isomorphism.

The connection between GIT-semistable models and the stable reduction has been explored before. The most comprehensive partial answer is given in \cite[Theorem 5.2]{van2025reduction}, which determines the combinatorial type of the special fiber from the reduction type of one GIT-semistable model. However, this result is only proved for $p > 7$. Moreover, the proof relies on methods from Invariant Theory that are not easily made explicit for computation and does not provide a direct algorithm to reconstruct the stable model $\X$ from a GIT-semistable model.

The central contribution of this thesis is to establish a precise bridge between these two notions of stability, formally stated in the following main theorem.

\begin{theoremstar}[Main Theorem]
    Let $X$ be a smooth plane quartic over $K$ whose stable model $\X$ exists. The following conditions are equivalent:
    \begin{enumerate}[label=\EquivListLabel]
        \item The stable reduction $\oo{X}$ is non-hyperelliptic.
        \item A GIT-stable plane model $\X_0$ of $X$ exists.
    \end{enumerate}
    Furthermore, when these equivalent conditions hold, the stable model $\X$ is the unique minimal semistable model dominating the GIT-stable model $\X_0$. The corresponding domination morphism
    \[
        \phi\colon \X \to \X_0 \subset \PP_{\OO_K}^2
    \]
    induces a morphism on the special fiber that contracts the $1$-tails\footnote{An irreducible component of arithmetic genus one attached to the rest of the curve at a single node.} of $\oo{X}$ to cusps on the special fiber $\X_{0,s}$ and is an immersion elsewhere.
\end{theoremstar}

This work is structured as follows. Chapter \ref{chap:1} covers the essential background for this thesis. We begin by reviewing models of curves and the stable reduction theorem. We then explain the connection between non-hyperelliptic genus $3$ curves and plane quartics. Building on this, we formalize the concept of a GIT-stable model for a smooth plane quartic, concluding with the formal statement of our main theorem. Chapter \ref{chap:2} is dedicated to a detailed exploration of what it means for a singular, stable curve to be hyperelliptic, establishing the precise combinatorial and geometric conditions required for the proof. Finally, Chapter \ref{chap:3} presents the proof of the main theorem, which proceeds in two parts. First, assuming the stable reduction is non-hyperelliptic, we construct a GIT-stable plane model by defining a contraction map from the abstract stable model. This map is built by twisting the dualizing sheaf, effectively contracting the so-called ``$1$-tails'' of the stable reduction to cusps. Conversely, assuming a GIT-stable model exists, we prove that the stable reduction must be non-hyperelliptic by realizing the stable model as the minimal semistable model, or ``stable hull,'' that dominates the given plane model.

\paragraph{A Note on the Text}
I would like to acknowledge the use of Google's Gemini AI for its assistance in refining the language in this thesis.

\paragraph{Conventions}

Throughout this thesis, we adhere to the following conventions. Concerning the geometry of curves, the unqualified term \emph{genus} invariably refers to the arithmetic genus, which we denote by $g$ (rather than the alternative notation $p_a$); the geometric genus will be explicitly specified when needed. Furthermore, when discussing components of stable curves, the term ``rational component'' implicitly denotes a ``smooth rational component'' (i.e., isomorphic to $\PP^1_k$ when working over an algebraically closed field $k$).

Regarding sheaves and divisors, the terms \emph{line bundle} and \emph{invertible sheaf} are used interchangeably. We employ the modern language of linear series, following the conventions of \cite[\href{https://stacks.math.columbia.edu/tag/0CCM}{Section 0CCM}]{stacks-project}: a \emph{linear series} is defined as a pair $(\LL, V)$, consisting of an invertible sheaf $\LL$ and a finite-dimensional vector subspace $V \subseteq H^0(C, \LL)$ of its global sections. The \emph{complete linear system} associated with $\LL$ is denoted by $|\LL|$, representing the pair $(\LL, H^0(C, \LL))$.
\section{Statement of Main Theorem} \label{chap:1}

Throughout this thesis, we assume the reader is familiar with the basics of algebraic geometry, as presented in \cite{hartshorne}, and also with the concepts of a more arithmetic nature as in \cite{liu}.

In this chapter, we begin by recalling the notion of a model and the semistable reduction theorem. We then explain the connection between non-hyperelliptic genus $3$ curves and plane quartics, introduce the concept of GIT-semistable models, and conclude by stating our main theorem.

\subsection{Models of Curves and Semistable Reduction}

We establish the standard setup and notation used in this thesis.

\begin{setup} \label{setup:field}
    Let $(K, v_K)$ be a complete discretely valued field with valuation ring $\OO_K$, a fixed uniformizer $\pi$, and an algebraically closed residue field $k$. Let $X$ be a smooth, projective, and geometrically irreducible curve of genus $g$ over $K$.
\end{setup}

To study the reduction of the curve $X$ modulo $\pi$, we must extend it to a scheme over $\OO_K$. Such an extension is called a \emph{model}, and the resulting reduction depends heavily on this choice. We define this concept formally as follows.

\begin{definition}
    An $\OO_K$-model (or simply a model) of $X$ is a flat, proper $\OO_K$-scheme $\X$ equipped with an isomorphism $\X_\eta := \X \otimes_{\OO_K} K \simeq X$. The special fiber of the model, $\X_s := \X \otimes_{\OO_K} k$, is called the reduction of $X$ (with respect to $\X$).
\end{definition}

\begin{remark}
    For simplicity, the isomorphism $\X_\eta \simeq X$ between the generic fiber and the curve $X$ is often treated as an equality. We will frequently write $\X_\eta = X$ and view $X$ as an open subscheme of its model $\X$.
\end{remark}

The flatness of the model $\X$ over $\OO_K$ implies that $\X$ is an integral scheme, since its generic fiber $X$ is integral. Furthermore, properness and flatness ensure that the special fiber, $\X_s$, is a projective curve over $k$ with the same arithmetic genus as $X$. However, the special fiber can have a more complicated structure than the generic fiber; it might be reducible or even non-reduced. We do not require models to be normal, although we note that in some of the cited literature, this is a required property for a model. The goal is often to find a model whose special fiber has desirable properties, such as being smooth or having only mild singularities. To compare different models, we introduce the notion of a morphism between them.

\begin{definition}
    Let $\X$ and $\X'$ be two models of $X$. A morphism $\X \to \X'$ between the models is a morphism of $\OO_K$-schemes that is compatible with the isomorphism of the generic fiber to $X$. Since the generic fiber is dense, at most one such morphism can exist. In this case, we say that the model $\X'$ is dominated by $\X$. We call $\X$ and $\X'$ isomorphic if each model is dominated by the other.
\end{definition}

It is straightforward to check that domination induces a partial order on the isomorphism classes of models of $X$.

Finding a model with a \emph{smooth} reduction is not always possible, even after a finite extension of the base field $K$. However, a fundamental result of Deligne and Mumford (see \cite[Corollary 2.7]{deligne-mumford}) is that one can always achieve a reduction with a well-behaved structure: a \emph{semistable curve}. A semistable curve is a reduced curve whose only singularities are ordinary double points (nodes)\footnote{Some authors additionally require that every smooth rational component meets the rest of the curve in at least two points. We adopt the more general definition where this is not required.}.

\begin{theorem}[Semistable Reduction Theorem]
    There exists a finite field extension $L/K$ such that the curve $X_L := X \otimes_K L$ has a semistable model, i.e., a model whose special fiber is a semistable curve. Furthermore, if the genus is $g \geq 2$, there exists a unique minimal semistable model (with respect to domination), called the \emph{stable model} of $X$.
\end{theorem}

From now on, we assume $g \ge 2$ and that $K$ has been replaced by a finite extension over which a semistable model of $X$ exists. The Semistable Reduction Theorem then guarantees a unique minimal semistable model, which we call the \emph{stable model} and denote by $\X$. The uniqueness of this model justifies calling its special fiber, $\oo{X} := \X_s$, \emph{the} stable reduction of $X$. As its name suggests, $\oo{X}$ is a stable curve: a semistable curve whose automorphism group is finite (see Chapter \ref{chap:2} for details). Geometrically, this means that every smooth rational component of $\oo{X}$ must intersect the other components in at least three points.

An important property of the stable model $\X$ is that it is always a \emph{normal} scheme. This follows from its special fiber being reduced (see \cite[Lemma 1.1]{liu1999models}). While normal, the stable model is not necessarily regular; its singularities are described precisely by the following proposition.

\begin{proposition}[{{\cite[Corollary 10.3.22]{liu}}}] \label{prop:thickness}
    The singularities of the stable model $\X$ are located exclusively at the nodes of the special fiber $\oo{X}$. Locally at such a node $p \in \oo{X}$, the completed local ring of $\X$ is given by
    \[
        \widehat{\OO}_{\X, p} \simeq \OO_K[\![u,v]\!]/(uv - \pi^d)
    \]
    for a unique integer $d \ge 1$, called the \emph{thickness} of the node. The model $\X$ is regular at $p$ if and only if this thickness is $d = 1$.
\end{proposition}

The normality of the stable model is a crucial property, as it ensures a well-defined theory of Weil divisors on $\X$, a tool that will be essential in Chapter \ref{chap:3}.

If the goal is to compute the stable reduction $\oo{X}$, a natural first step is to determine its \emph{combinatorial type}\footnote{See Definition \ref{def:associated_graph_comb_type} for a formal definition.}, also known as the \emph{reduction type} of $X$. This describes the ``skeleton'' of the curve: its irreducible components, their geometric genera, and how they intersect.

\begin{example} \label{ex:comb_types_g=2}
    For genus $g=2$, there are seven distinct combinatorial types of stable curves, as shown in Figure \ref{fig:stable_genus_2}. The diagrams follow the style used in \cite{bommelcayley} and were drawn using the Hobby-Editor \cite{hobbyeditor}.
    \begin{figure}[htbp]
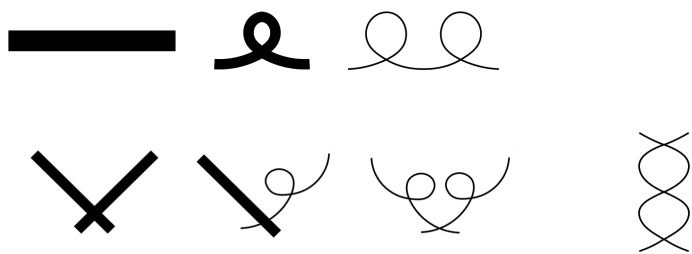

        \centering
        \begin{tabular}{cccc}
            \adjustbox{valign=c}{\scalebox{1.75}{\HobbyCurve{2}}} & 
            \adjustbox{valign=c}{\scalebox{1.75}{\HobbyCurve{1n}}} & 
            \adjustbox{valign=c}{\scalebox{1.75}{\HobbyCurve{0nn}}} & 
            \\
            [1.5em]
            \adjustbox{valign=c}{\scalebox{1.75}{\HobbyCurve{ee}}} & 
            \adjustbox{valign=c}{\scalebox{1.75}{\HobbyCurve{me}}} & 
            \adjustbox{valign=c}{\scalebox{1.75}{\HobbyCurve{mm}}} & 
            \adjustbox{valign=c}{\scalebox{1.75}{\HobbyCurve{0---0}}}
        \end{tabular}
        \caption{The $7$ combinatorial types of stable curves of genus $2$.}
        \label{fig:stable_genus_2}
    \end{figure}
    The seven configurations can be grouped into three irreducible types, three types formed by two curves joined at a single node, and a final type, a binary curve, in which two smooth rational curves meet at three distinct points. In the figure, the geometric genus of each component ($0$, $1$, or $2$) is indicated by line thickness, from thinnest to thickest. This classification can be verified using the formula relating the arithmetic genus of a curve to the geometric genera of its components (see \cite[Proposition 7.5.4]{liu}).
\end{example}

\begin{remark} \label{rem:comb_types}
    The number of combinatorial types of stable curves grows rapidly with the genus. The sequence for small genera is known ($g \leq 8$):
    \[
        7, 42, 379, 4555, 69808, 1281678, 27297406
    \]
    This is sequence \href{https://oeis.org/A174224}{A174224} in the On-Line Encyclopedia of Integer Sequences (OEIS) and may be computed with the algorithm in \cite{maggiolo2011generating}.
\end{remark}

\subsection{\texorpdfstring{Non-Hyperelliptic Genus $3$ Curves as Plane Quartics}{Non-Hyperelliptic Genus 3 Curves as Plane Quartics}}

Having established the general framework of stable reduction, we now focus on the specific case of interest for this thesis: \emph{non-hyperelliptic curves of genus $g=3$}. A fundamental classification for such curves (which, unless stated otherwise, we assume to be smooth, projective, and geometrically irreducible as in Setup~\ref{setup:field}) is the distinction between hyperelliptic and non-hyperelliptic, which we now recall.

\begin{definition}[{{\cite[Definition 7.4.7]{liu}}}]
    A smooth, projective, geometrically irreducible curve $X$ of genus $g \ge 2$ is called \emph{hyperelliptic} if it admits a finite, separable morphism of degree $2$ to a smooth projective conic, $T$.\footnote{Some authors also admit the case $g = 1$, but there is a reason to exclude the case or only allow pointed genus $1$ curves.}
\end{definition}

\begin{remark}
    We have $T \simeq \PP^1_K$ if and only if $T$ admits a $K$-rational point (see \cite[Proposition 7.4.1]{liu}). Since we may pass to a finite extension of $K$, we may assume for simplicity that $T = \PP^1_K$.
\end{remark}

This morphism $X \to \PP^1_K$ is unique up to automorphisms of $\PP^1_K$. The geometric distinction between hyperelliptic and non-hyperelliptic curves is captured by the canonical map $\Phi_X\colon X \to \PP_K^{g-1}$, induced by the canonical sheaf $\omega_{X/K} = \Omega_{X/K}^1$. The behavior of this map differs fundamentally depending on whether the curve is hyperelliptic. For a hyperelliptic curve, the canonical map is never an embedding; it factors through the 2-to-1 map to $\PP^1_K$ (see \cite[Proposition 7.4.34]{liu}). For non-hyperelliptic curves, the situation is completely different, as described by the following theorem.

\begin{theorem} \label{thm:canonical_embedding}
    Let $X$ be a smooth, projective, geometrically irreducible curve with genus $g \geq 2$. The canonical map $\Phi_X\colon X \to \PP_K^{g-1}$ is a closed immersion if and only if $X$ is not hyperelliptic.
\end{theorem}

\begin{proof}
    See \cite[Proposition 7.4.12]{liu}.
\end{proof}

As a consequence of this result, any genus $2$ curve is hyperelliptic (see also \cite[Proposition 7.4.9]{liu}). Hence, the case of our interest, genus $3$, is the smallest genus for which non-hyperelliptic curves appear.

Let $X$ be a non-hyperelliptic curve of genus $g=3$. By Theorem~\ref{thm:canonical_embedding}, its canonical map is a closed immersion
\[
    \Phi_X\colon X \hookrightarrow \PP_K^{g-1} = \PP_K^2.
\]
Since $\omega_{X/K}$ has degree $2g-2 = 4$, this map realizes $X$ as a smooth plane quartic. Conversely, any smooth plane quartic has genus $g = \frac{(4-1)(4-2)}{2} = 3$ (\cite[Example 7.3.22]{liu}), and the adjunction formula (see \cite[Theorem 9.1.37]{liu})
\[
    \omega_{X/K} \simeq \omega_{\PP_K^2/K}|_X \otimes_{\OO_X} \OO_{\PP^2_K}(X)|_X \simeq \OO_{\PP^2_K}(-3)|_X \otimes_{\OO_X} \OO_{\PP^2_K}(4)|_X \simeq \OO_{\PP^2_K}(1)|_X
\]
shows that its embedding is given by the canonical map; this implies, in particular, that the curve is non-hyperelliptic.

This correspondence establishes a fundamental equivalence that is central to our topic:
\begin{center}
    Studying non-hyperelliptic genus $3$ curves \\ is the same as studying smooth plane quartics.
\end{center}

Hence, we may assume that our curve $X$ is explicitly given by a homogeneous degree-4 equation in $\PP^2_K$:
\[
    X: F(x,y,z) = 0.
\]

The Semistable Reduction Theorem ensures the existence of a unique stable model $\X$ for the curve $X$. This fact motivates the central problem of this work.

\begin{problem}
    Given a smooth plane quartic $X$, compute its stable reduction $\oo{X}$, or at least determine its combinatorial type.
\end{problem}

This problem can be approached in several ways, such as using analytic methods on the Berkovich space of the curve, performing explicit resolutions of singularities, or applying the theory of admissible reductions. The approach taken in this thesis is more algebraic and relies on what are known as plane models.

Since $X$ is a plane quartic over $K$, the most direct way to construct a model is to find a suitable equation over the ring of integers $\OO_K$. This leads to the concept of a plane model.

\begin{definition}
    A \emph{plane model} $\X_0$ of $X$ is a closed subscheme of the projective plane $\PP^2_{\OO_K}$ that is flat over $\OO_K$ and whose generic fiber is isomorphic to $X$.
    Explicitly, such a model is defined by a primitive, homogeneous, degree-4 polynomial $G \in \OO_K[x,y,z]$:
    \[
        \X_0 : G(x,y,z) = 0.
    \]
\end{definition}

The key insight, explored in recent work such as \cite{van2025reduction}, is that there is a direct connection between the stable model $\X$ and plane models $\X_0$ that are (semi)stable in the sense of \emph{Geometric Invariant Theory} (GIT). GIT provides a powerful framework for classifying algebraic objects under group actions. For our purposes, a detailed exposition of GIT is not necessary; instead, we will use its classification of plane quartics as a working definition.

\begin{definition}[see {\cite[Lemma 1.4]{artebani2009compactification}\cite[p.~80]{mumford1994geometric}}]
\label{def:git-stability}
    Let $C \subset \PP_k^2$ be a plane quartic over the residue field $k$.
    \begin{enumerate}[label=\DefListLabel]
        \item $C$ is called \emph{GIT-semistable} if it has no triple points and is not the union of a cubic with an inflectional tangent.
        \item $C$ is called \emph{GIT-stable} if its singularities are at most of type $A_1$ (nodes) and $A_2$ (cusps).
        \item A plane model $\X_0$ of $X$ is called \emph{GIT-(semi)stable} if its special fiber, $\X_{0,s}$, is GIT-(semi)stable in the sense above.
    \end{enumerate}
\end{definition}

If $C$ has only isolated singularities, then having no triple point is equivalent to saying the curve has only $A_n$ singularities (analytically isomorphic to $y^2 = x^{n+1}$\,\footnote{In $\Char k = 2$, the normal forms for $A_n$-singularities are more complicated; see \cite{greuel1990simple}. The equation $xy = 0$ for a node ($A_1$) and the equation $y^2=x^3$ for a cusp ($A_2$) hold in any characteristic.}). Otherwise $C$ is not reduced, and having no triple points can only happen if $C$ is a double smooth conic. A classification of the possible singularity types over $\CC$, with normal forms for each case, was originally given by Hui \cite{hui}. This classification and the corresponding normal forms also hold in characteristic $p > 7$ (see \cite[Remark 2.3]{van2025reduction}). For an overview, see \cite[Table 2.1, Table 2.2]{van2025reduction}, where the GIT-semistable types are also indicated.

The goal of this thesis is to explore this connection using a more geometric approach than that of \cite{van2025reduction}, inspired by the proof of \cite[Proposition A.1]{van2025reduction}. It is known that after a suitable field extension, $X$ always has a GIT-semistable plane model, which can be computed using methods from \cite{stern2025models}. A GIT-stable model may not exist, but if it does, it is the unique GIT-semistable model up to isomorphism. In contrast, if there does not exist a GIT-stable model, there may be several non-isomorphic GIT-semistable models.

This raises the central question that motivates this work:
\begin{center}
    What is the precise relationship between the abstractly defined stable model $\X$ and a GIT-stable plane model $\X_0$?
\end{center}
By analyzing a GIT-(semi)stable model, we aim to gain direct, computable insight into the geometry of the stable reduction $\oo{X}$.

We now arrive at the central result of this thesis, which establishes a precise connection between the geometry of the stable reduction and the existence of a GIT-stable plane model.

\begin{theorem}[Main Theorem] \label{thm:main_theorem}
    The following conditions are equivalent:
    \begin{enumerate}[label=\EquivListLabel]
        \item \label{thm:main_theorem_i} The stable reduction $\oo{X}$ is non-hyperelliptic.
        \item \label{thm:main_theorem_ii} A GIT-stable plane model $\X_0$ of $X$ exists.
    \end{enumerate}
    Furthermore, when these equivalent conditions hold, the stable model $\X$ is the unique minimal semistable model dominating the GIT-stable model $\X_0$. The corresponding domination morphism
    \[
        \phi\colon \X \to \X_0 \subset \PP_{\OO_K}^2
    \]
    induces a morphism $\oo{\phi}$ on the special fiber that contracts the $1$-tails of $\oo{X}$ to cusps on $\X_{0,s}$ and is an immersion elsewhere.
\end{theorem}

To understand this theorem, we must first clarify what it means for a possibly singular, stable curve such as $\oo{X}$ to be non-hyperelliptic. This concept will be defined and explored in detail in the next chapter.

\begin{remark}
    The assumption that the field $k$ is algebraically closed is made for convenience; the weaker assumption that $k$ is perfect would suffice. All proofs can be reduced to the algebraically closed case by applying the same arguments as in \cite[Chapter 10]{liu}. Including these arguments in every proof would be repetitive and would clutter the presentation; hence, we make this simplifying assumption.
\end{remark}

\section{Hyperelliptic Curves} \label{chap:2}

In this chapter, $k$ is a fixed algebraically closed field. For our main theorem, it is necessary to explain what it means for a stable curve $C$ over $k$ to be hyperelliptic. We begin with a brief overview of the existing literature on this topic.\footnote{The material in this chapter led to the article \cite{hyperelliptic}.}

Informally, a stable curve is considered hyperelliptic if it is the ``limit'' of smooth hyperelliptic curves. While this idea can be made precise using the moduli space of stable curves, such a definition is not intrinsic, making it difficult to determine whether a given stable curve is hyperelliptic from the definition alone. In the smooth case, hyperellipticity is defined by the existence of a degree-two map $C \to \PP_k^1$. When allowing $C$ to degenerate into a stable curve, one must also allow its target $\PP_k^1$ to degenerate into a semistable curve $T$ of genus $0$, also known as a rational tree. By defining a specific class of allowed maps $C \to T$ --- the so-called (2-sheeted) admissible covers --- one can characterize these ``limits of smooth hyperelliptic curves'' (see \cite[Theorem 3.160]{harris1998moduli}). It is important to note that the cited authors work over $k = \CC$. Their proofs, being complex-analytic, do not translate directly to positive characteristic. It is likely, however, that the theory can be adapted for fields of characteristic $p \neq 2$, while the case $p=2$ is known to involve true complications.

Another approach, which we will pursue, uses a generalization of the hyperelliptic involution to define the notion of hyperellipticity intrinsically, as briefly explained in \cite[Chapter X]{ACGII}. This book shows that, again for $k = \CC$, this method also precisely characterizes ``the limits of smooth hyperelliptic curves'' (see \cite[Chapter XI, Lemma 6.14 and 6.15]{ACGII}). We generalize this definition to stable pointed curves without making any assumptions on the characteristic of the field $k$. This requires a careful approach; in characteristic $2$, the naive generalization is too strict and does not capture the correct class of stable curves. The definition can be adjusted slightly to define the correct class of curves in this case. Note that the arguments in \cite{ACGII} cannot be generalized to positive characteristic, as they rely on analytic techniques inapplicable in that setting. The proof that this intrinsic definition correctly identifies the limits of smooth hyperelliptic curves in positive characteristic is given by Maugeais \cite[Th\'eor\`eme 5.4]{maugeais2003relevement}, whose work uses a generalization of admissible covers in positive characteristic. For further exposition, see also \cite[Appendix A.2]{yamaki2004cornalba}.

A third, older, and generally non-equivalent approach to defining hyperellipticity focuses on the canonical map. If a projective, connected curve $C$ has only Gorenstein singularities, it possesses an invertible dualizing sheaf $\omega_{C/k}$, which is a generalization of the canonical sheaf for smooth curves. This sheaf defines a canonical (though generally rational) map $C \to \PP_k^{g-1}$. This approach was pioneered by Rosenlicht in 1952. Working with a specific class of irreducible curves, he defined the property of being ``quasi-hyperelliptic'' for a curve if its canonical map is not birational \cite{rosenlicht1952equivalence}. He showed that for these curves, this property is equivalent to the existence of a degree-two map $C \to \PP^1_k$, a characterization that is often used as a definition today (compare with \cite{kleiman2009canonical} and \cite{hartshorne1986generalized}). A more detailed investigation of the canonical map and the so-called pluricanonical maps (given by $\omega_{C/k}^{\otimes n}$ for $n \geq 1$) for reducible Gorenstein curves was undertaken by Catanese \cite{catanese1982pluricanonical} in 1982. His work was motivated by the fact that $\omega_{C/k}^{\otimes n}$ is very ample for $n \geq 3$ on stable curves, an essential result used by Deligne and Mumford \cite{deligne-mumford} in their 1969 construction of the moduli space of stable curves. Catanese gave a generalization of Theorem \ref{thm:canonical_embedding}, which characterizes when the canonical map is an immersion. We will utilize his results in Chapter \ref{chap:3}. However, it is crucial to note that his definition of hyperellipticity \cite[Definition 3.9]{catanese1982pluricanonical}, based on the non-birationality of the canonical map, does not in general coincide with our definition for stable curves. To avoid confusion, and motivated by Rosenlicht's terminology, we will refer to curves that are hyperelliptic in the sense of Catanese as ``quasi-hyperelliptic''. As we will see, for certain classes of stable curves, the two notions do agree.

We begin this chapter by defining stable pointed curves and then present our definition of when such curves are called hyperelliptic. This will involve a generalization of the hyperelliptic involution, for which we will prove uniqueness, similar to the smooth case. Subsequently, we will focus on the case of genus $3$ curves, classifying all stable types and characterizing the hyperelliptic property for them in concrete terms. To conclude the chapter, we will discuss the related notion of ``quasi-hyperelliptic'' curves, clarifying its connection and distinction from our primary definition. With this information, we will be equipped to prove the main theorem in the next chapter.

\subsection{Stable Pointed Curves}

We begin by defining the main objects of study.

\begin{definition}
    An algebraic curve $C$ over $k$ is called semistable if it is connected, projective, reduced, and its only singularities are ordinary double points (nodes). 
    
    A semistable curve $C$ is called stable if it satisfies the following equivalent conditions:
    \begin{enumerate}[label=\EquivListLabel]
        \item The curve has arithmetic genus $g \geq 2$ and every smooth rational component ($Z \simeq \PP_k^1$) meets the other components in at least three points.
        \item The automorphism group $\Aut(C)$ is finite.
    \end{enumerate}
\end{definition}

\begin{remark}
    In his book \cite{liu}, Liu does not require semistable curves to be connected and projective. In this thesis, however, we will require them to have these properties for convenience. Such curves are also called ``nodal curves'' in the literature. Different contexts place different demands on semistability; for example, in GIT, one often requires that every smooth rational component meets the rest of the curve in at least two points, a condition equivalent to the automorphism group being reductive.
\end{remark}

It is convenient to generalize the notion of stability to pointed curves.

\begin{definition} \label{def:stable_pointed}
    An $n$-pointed semistable curve is a datum $(C, P)$ of a semistable curve $C$ together with a set $P = \{p_1, \dots, p_n\}$ of $n$ distinct smooth points of $C$. For each irreducible component of the curve, let the set of special points be the union of its marked points and the nodes connecting it to other components.

    Such a curve is called a stable $n$-pointed curve if it satisfies the following equivalent conditions:
    \begin{enumerate}[label=\EquivListLabel]
        \item \label{def:stable_pointed_i} Every genus $0$ component has at least three special points, and every genus $1$ component has at least one special point.
        \item The subgroup of automorphisms of $C$ that fix each point in $P$, denoted $\Aut(C, P)$, is finite.
    \end{enumerate}
\end{definition}

\begin{remark}
    Note that a stable $0$-pointed curve is simply a stable curve as previously defined.
\end{remark}

The notion of stability is strongly connected to the ampleness of the canonical sheaf, also called the dualizing sheaf (see, for instance, \cite[Section 6.4]{liu}). For a more explicit description of the dualizing sheaf, we refer to Chapter \ref{chap:3}. To a semistable pointed curve $(C,P)$, we associate the so-called log-canonical sheaf $\omega_{C/k}(P)$. Here, we identify the set of points $P = \{p_1, \dots, p_n\}$ with the divisor $p_1 + \cdots + p_n$.

\begin{proposition} \label{prop:stable_via_canonical}
    Let $(C,P)$ be an $n$-pointed semistable curve. Then the following are equivalent:
    \begin{enumerate}[label=\EquivListLabel]
        \item \label{prop:stable_via_canonical_i} The curve $(C,P)$ is stable.
        \item \label{prop:stable_via_canonical_ii} The log-canonical sheaf $\omega_{C/k}(P)$ is ample.
        \item \label{prop:stable_via_canonical_iii} For every irreducible component $Z$ of $C$, we have
        \[
            \deg \omega_{C/k}(P)|_Z > 0.
        \]
        \item \label{prop:stable_via_canonical_iv} For every irreducible component $Z$ of $C$ with $s_Z$ special points, we have
        \[
            2 g(Z) - 2 + s_Z > 0.
        \]
    \end{enumerate}
    If these equivalent conditions are satisfied, then
    \[
        \deg \omega_{C/k}(P) = 2g - 2 + n \geq m > 0,
    \]
    where $g = g(C)$ is the arithmetic genus of $C$ and $m$ is the number of irreducible components of $C$.
\end{proposition}

\begin{proof}
    The equivalence of \ref{prop:stable_via_canonical_ii} and \ref{prop:stable_via_canonical_iii} follows from \cite[Proposition 7.5.5]{liu}. The equivalence of \ref{prop:stable_via_canonical_i} and \ref{prop:stable_via_canonical_iv} follows directly from the definition of a stable pointed curve. Indeed, let $Z$ be an irreducible component of $C$ with $s_Z$ special points. The condition $2 g(Z) - 2 + s_Z > 0$ is automatically satisfied if $g(Z) \geq 2$; for $g(Z) = 1$, it is equivalent to $s_Z \geq 1$; and for $g(Z) = 0$, it is equivalent to $s_Z \geq 3$. These are precisely the conditions given in Definition \ref{def:stable_pointed} \ref{def:stable_pointed_i}.

    The equivalence of \ref{prop:stable_via_canonical_iii} and \ref{prop:stable_via_canonical_iv} is also clear: For an irreducible component $Z$ of $C$, let $S_Z$ be the set of its special points (the marked points on $Z$ and the nodes connecting it to other components), so that $|S_Z| = s_Z$. We have the isomorphism $\omega_{C/k}(P)|_Z \simeq \omega_{Z/k}(S_Z)$ (cf. \cite[Lemma 10.3.12]{liu}). Since $\deg(\omega_{Z/k}) = 2g(Z) - 2$ (see \cite[Corollary 7.3.31]{liu}), it follows that $\deg \omega_{C/k}(P)|_Z = 2g(Z)-2+s_Z$, establishing the equivalence.

    For the last assertion, let $Z_1, \dots, Z_m$ be the irreducible components of $C$. If the curve is stable, then condition \ref{prop:stable_via_canonical_iii} holds, so $\deg \omega_{C/k}(P)|_{Z_i} \geq 1$ for each component. Using \cite[Proposition 7.5.7]{liu}, we can sum these degrees:
    \[
        2 g - 2 + n = \deg \omega_{C/k}(P) = \sum_i \deg \omega_{C/k}(P)|_{Z_i} \geq m > 0.
    \]
\end{proof}

\begin{definition}
    In the situation of Proposition \ref{prop:stable_via_canonical}, an irreducible component $Z$ of $C$ is called an \emph{unstable component} if it fails to satisfy condition \ref{prop:stable_via_canonical_iii}. In other words, $Z$ is unstable if it is a genus $0$ component with fewer than three special points, or if it is a genus $1$ component with no special points. The latter case can only occur if the component $Z$ is the entire curve $C$, which must then be of genus $1$ with no marked points.
\end{definition}

\begin{remark} \label{rem:stable_curves_low_genus}
    Proposition \ref{prop:stable_via_canonical} implies that a stable $n$-pointed curve of genus $g$ can exist only if one of the following holds: $g = 0$ with $n \geq 3$; $g = 1$ with $n \geq 1$; or $g \geq 2$. Moreover, the proposition shows that the number of irreducible components is bounded by $2g-2+n$, which implies that for a given $(g,n)$, the number of possible combinatorial types (see Definition \ref{def:associated_graph_comb_type}) of $(C,P)$ is finite.
\end{remark}

For a more local analysis, it is convenient to decompose a semistable curve $C$ at a set of nodes $S$. Formally, we use the partial normalization $\nu\colon \tilde{C} \to C$ of $C$ at $S$. The preimage $\tilde{C}$ is a disjoint union of connected components, $\tilde{C} = C_1 \sqcup \cdots \sqcup C_m$, and we call $\{C_1, \dots, C_m\}$ the decomposition of $C$ induced by $S$. For notational simplicity, we adopt the following convention: if a point has a unique preimage under $\nu$, we denote it by the same letter. If a node $p \in S$ has two preimages lying on distinct components $C_i$ and $C_j$, we also denote both preimages by $p$; context will make it clear whether we mean the point on $C_i$ or on $C_j$. We write $S_i := \nu^{-1}(S) \cap C_i$ for the set of preimages of the nodes in $S$ that lie on the component $C_i$. If no two preimages of a node in $S$ lie on the same component $C_i$, then the restriction $\nu|_{C_i} \colon C_i \to C$ is an isomorphism onto its image. In this case, we may identify $C_i$ with the subcurve $\nu(C_i) \subset C$, and under this identification, $S_i$ is a subset of $S$. For example, if $S$ is the set of all nodes connecting distinct irreducible components of $C$, then the curves $C_i$ in the decomposition are precisely those irreducible components.

The components of a decomposed curve should themselves be viewed as pointed curves. Let $(C,P)$ be a stable $n$-pointed curve, and let $S \subset C$ be a set of nodes with the induced decomposition $\{C_1, \dots, C_m\}$. Since the points in $P$ are smooth, they are unaffected by the partial normalization at $S$, so we may identify them with their preimages on the components $C_i$. We then equip each component $C_i$ with a set of marked points $P_i := S_i \cup (P \cap C_i)$. This set consists of the preimages of the nodes in $S$ (where the curve was ``cut'') and the original marked points from $P$ that lie on $C_i$. We call the resulting collection of pointed curves $\{(C_1, P_1), \dots, (C_m, P_m)\}$ the (pointed) decomposition of $(C,P)$ induced by $S$.

\begin{lemma} \label{lem:stable_decomposition}
    Let $(C, P)$ be a semistable curve, and let $S \subset C$ be a set of nodes. Then
    \[
        (C, P) \text{ is stable } \iff (C_i, P_i) \text{ is stable for all $i = 1, \dots, m$},
    \]
    where $\{(C_1, P_1), \dots, (C_m, P_m)\}$ is the decomposition induced by $S$.
\end{lemma}

\begin{proof}
    Let $\nu \colon \tilde{C} = \bigsqcup C_i \to C$ be the partial normalization at the nodes in $S$. The map $\nu$ induces a bijection between the irreducible components of $\tilde{C}$ and $C$. Let $Z$ be an irreducible component of $C$, and let $\tilde{Z}$ be the corresponding irreducible component in the decomposition, which we can assume lies in $C_i$. By Proposition \ref{prop:stable_via_canonical} \ref{prop:stable_via_canonical_iv}, it suffices to verify that
    \begin{equation} \label{eq:stability_normalization}
        2g(Z) - 2 + s_Z = 2g(\tilde{Z}) - 2 + s_{\tilde{Z}},
    \end{equation}
    where $s_Z$ and $s_{\tilde{Z}}$ are the number of special points on $Z$ (in $(C,P)$) and on $\tilde{Z}$ (in $(C_i,P_i)$), respectively.

    The restricted morphism $\nu|_{\tilde{Z}} \colon \tilde{Z} \to Z$ is the normalization of $Z$ at the set of nodes $N_Z \subset S$ that are self-intersections of $Z$. Let $N_{\tilde{Z}}$ be the set of preimages of $N_Z$ on $\tilde{Z}$; note that $|N_{\tilde{Z}}| = 2|N_Z|$. The genus formula for normalization gives the relation $g(Z) = g(\tilde{Z}) + |N_Z|$. By comparing the definitions of special points, one finds that their numbers are related by $s_{\tilde{Z}} = s_Z + |N_{\tilde{Z}}|$. Substituting these two relations confirms the equality \eqref{eq:stability_normalization}.
\end{proof}

It is standard to associate a graph with a semistable curve $C$, called the \emph{dual graph}. Graph-theoretical terminology is often more convenient for certain arguments and allows us to formalize the so-called \emph{combinatorial type} of $C$. Our presentation is inspired by \cite[Chapter X]{ACGII}. We begin by defining what is meant by a graph in this context:

\begin{definition}
    A graph $\Gamma$ is given by the following data:
    \begin{description}[leftmargin=!,labelwidth=\widthof{\textit{Vertex assignment:}}]
        \item[Vertices:] A finite nonempty set $V = V(\Gamma)$.
        \item[Half-edges:] A finite set $L = L(\Gamma)$.
        \item[Edge-pairing:] An involution $\iota\colon L \to L$.
        \item[Vertex assignment:] A partition of $L$ indexed by $V$, i.e., an assignment to each $v \in V$ of a (possibly empty) subset $L_v \subset L$ such that $L = \bigsqcup_v L_v$.
        \item[Genus assignment:] An assignment of a nonnegative integer weight $g_v$ to each vertex $v \in V$.
    \end{description}
    A pair of distinct elements of $L$ interchanged by the involution is called an \emph{edge} of the graph, and the set of edges is denoted by $E = E(\Gamma)$. A fixed point of the involution is called a \emph{leg} of the graph, and the set of legs is denoted by $P = P(\Gamma)$. The \emph{degree} of a vertex $v \in V$, denoted $\deg(v)$, is the number of half-edges attached to it: $\deg(v) := |L_v|$. The genus of $\Gamma$ is defined by
    \[
        g(\Gamma) := \sum_{v \in V} g_v + 1 - \chi(\Gamma),
    \]
    where $\chi(\Gamma) := |V| - |E|$ is the Euler characteristic of $\Gamma$.
\end{definition}

The structure resulting from this definition is a type of weighted \emph{multigraph}. It allows for \emph{loops} (edges connecting a vertex to itself), \emph{multiple edges} between the same pair of vertices, and \emph{legs} (half-edges not paired into a full edge).

An \emph{isomorphism} between two such graphs is a pair of bijections on the sets of vertices and half-edges that preserves the entire defining structure: the edge-pairing, vertex assignment, and genus assignment.

\begin{definition} \label{def:associated_graph_comb_type}
    Let $\{(C_1, P_1), \dots, (C_m, P_m)\}$ be the pointed decomposition of a semi\-stable curve $(C, P)$ induced by a set of nodes $S$. The associated graph, denoted $\Gamma = \Graph^S(C,P)$, is defined by the following data:
    \begin{description}
        \item[Vertices:] The set of vertices is $V := \{C_1, \dots, C_m\}$.

        \item[Half-edges:] The set of half-edges is the disjoint union of the sets of marked points from the pointed decomposition: $L := \bigsqcup_{i=1}^m P_i$.

        \item[Edge-pairing:] The involution $\iota\colon L \to L$ is defined by its action on a half-edge $h \in P_i \subset L$:
            \begin{itemize}
                \item If $h \in S_i$, it is one of the two preimages of a node $s \in S$ under the partial normalization map $\nu$. The fiber $\nu^{-1}(s)$ consists of the pair $\{h, h'\}$, and we define $\iota(h) := h'$.
                \item If $h \in P \cap C_i$, it corresponds to one of the original marked points from $P$. We define $\iota(h) := h$, making it a fixed point.
            \end{itemize}

        \item[Vertex assignment:] For each vertex $v = C_i \in V$, the set of attached half-edges is $L_v := P_i$.

        \item[Genus assignment:] Each vertex $v = C_i \in V$ is assigned its arithmetic genus as a weight: $g_v := g(C_i)$.
    \end{description}
    In the special case where $S$ consists of all nodes of $C$, the associated graph $\Gamma = \Graph(C,P)$ is called the \emph{dual graph} of $(C,P)$. The \emph{combinatorial type} of $(C,P)$ is then defined as the isomorphism class of its dual graph.
\end{definition}

\begin{remark}
    \begin{enumerate}[label=\StatementListLabel]
        \item The genus of the associated graph $\Gamma = \Graph^S(C, P)$ coincides with the arithmetic genus of the curve $C$ (see \cite[Chapter X, p. 88, formula (2.8)]{ACGII}). Since $C$ is connected, its associated graph $\Gamma$ is also connected, and the genus formula may be written as
        \[
            g(\Gamma) = \sum_{v \in V} g_v + h^1(\Gamma),
        \]
        where $h^1(\Gamma)$ is the first Betti number of $\Gamma$.
        \item For the dual graph, the weight $g_v$ of a vertex is precisely the \emph{geometric genus} of the corresponding irreducible component. The combinatorial type, therefore, informally captures the arrangement of the irreducible components, their geometric genera, and how they intersect. Throughout this thesis, we will represent these types visually using diagrams, as in Example \ref{ex:comb_types_g=2}, where the thickness of a line indicates the geometric genus of the corresponding component.
    \end{enumerate}
\end{remark}

With this new graph-theoretic language, we may express the stability condition as follows.

\begin{lemma}
    Let $(C, P)$ be a semistable curve, and let $S \subset C$ be a set of nodes. If $(C,P)$ is stable, then
    \[
        2 g_v - 2 + \deg(v) > 0
    \]
    for all vertices $v$ in the graph $\Graph^S(C, P)$. The converse is also true, provided that $S$ contains all nodes connecting distinct irreducible components of $C$; this holds, in particular, if $S$ contains all nodes of $C$.
\end{lemma}

\begin{proof}
    Let $\{(C_1, P_1), \dots, (C_m, P_m)\}$ be the decomposition induced by $S$. For the vertex $v$ corresponding to a component $C_i$, we have $g_v = g(C_i)$ and $\deg(v) = |P_i|$ by definition. The statement of the lemma follows directly from applying Lemma \ref{lem:stable_decomposition} and Proposition \ref{prop:stable_via_canonical} to each component of the decomposition. The condition on $S$ for the converse is precisely what ensures that each component $C_i$ is irreducible, allowing the implication in Proposition \ref{prop:stable_via_canonical} to be reversed.
\end{proof}

\begin{remark} \label{rem:stable_components}
    Two choices for the set of nodes $S$ are particularly useful.
    As we have seen, taking $S$ to be all nodes of $C$ yields the \emph{dual graph}. Its vertices correspond to the normalizations of the irreducible components of $C$, so their assigned weights are the \emph{geometric} genera.
    Alternatively, taking $S$ to be only the nodes connecting distinct irreducible components results in a decomposition where the vertices are the irreducible components themselves (with their self-intersections); their weights are thus the \emph{arithmetic} genera.
    The associated graph in this second case, which by construction has no loops, is sometimes called the ``component graph,'' but this term is not standard.
    This makes the component graph a coarser representation, as it does not distinguish between curves whose components have the same arithmetic genus but different geometric genera.
    The dual graph, by contrast, captures this additional detail by also translating self-intersection nodes into loops.
\end{remark}

\subsection{Hyperelliptic Stable Curves}

To define hyperellipticity for stable curves, we first recall the smooth case. A smooth curve $C$ of genus $g \ge 2$ is hyperelliptic if it admits a separable\footnote{Note that ``separable'' is automatic since otherwise $C$ would have genus $0$; see proof of \cite[Lemma 7.4.8]{liu}.} morphism to $\PP^1_k$ of degree $2$. This map is induced by a unique order-two automorphism $\sigma$, called the hyperelliptic involution. The generalization to stable curves involves replacing the target $\PP^1_k$ with a semistable curve of genus $0$. Such a curve has irreducible components that are all rational and a dual graph that is a tree; for these reasons, it is called a \emph{rational tree}.

This leads to the standard definition used in the literature, which we present here for pointed curves (cf. \cite[Chapter X, Section 3]{ACGII}).

\begin{definition} \label{def:hyperelliptic}
    A stable pointed curve $(C, P)$ over $k$ is called \emph{hyperelliptic} if it admits an involution $\sigma\colon C \to C$ that fixes every point in $P$ and whose quotient curve $T := C/\langle\sigma\rangle$ is a rational tree. When $\Char k \neq 2$, it is additionally required that $\sigma$ has only isolated fixed points. Such an automorphism $\sigma$ is called a \emph{hyperelliptic involution}.
\end{definition}

\begin{remark}[Associated Degree-2 Covers] \label{rem:degree2_map}
    In characteristic different from 2, the requirement that the fixed points of $\sigma$ be isolated is equivalent to the condition that the restriction of $\sigma$ to any irreducible component is not the identity. Consequently, the quotient morphism $\pi\colon C \to T$ is a finite map of degree two over every component. It is, however, not in general an \emph{admissible cover} in the classical sense (see \cite[Section 3.G]{harris1998moduli} and \cite[Remark 3.4]{hyperelliptic}).

    In characteristic $2$, this requirement is dropped, and $\sigma$ may be the identity on some components of $C$. This occurs precisely on certain rational components $Z \simeq \PP^1_k$; the exact conditions are given in Proposition~\ref{prop:hyperelliptic_genus2}. The quotient map $\pi$ is then of degree one over the images $T_Z := \pi(Z)$.
    
    To naturally associate a finite morphism of degree two to $\sigma$ everywhere, one must compose $\pi$ with a ``partial Frobenius'' map $F\colon T \to T'$. Here, $T'$ is another rational tree, constructed from $T$ by replacing each component $T_Z$ (where $\sigma$ was trivial) with its Frobenius twist $T_Z^{(2)}$. The map $F$ is defined component-wise as the relative Frobenius morphism $F_{T_Z/k}\colon T_Z \to T_Z^{(2)}$ on these components and the identity elsewhere; the gluing of $T'$ is defined such that $F$ is a well-defined morphism. The resulting map, $\pi' := F \circ \pi \colon C \to T'$, is a finite morphism of degree two, factoring into the separable morphism $\pi$ and the purely inseparable morphism $F$. The curve $T'$ remains a rational tree of the same combinatorial type as $T$; its components are still rational (as the Frobenius twist preserves geometric genus, $T_Z^{(2)}$ remains a smooth genus $0$ curve, and since $k$ is algebraically closed we must have $T_Z^{(2)} \simeq \PP^1_k$), and the map $F$, being purely inseparable, is a homeomorphism that preserves the combinatorial structure of $T$.
    
    Since $k$ is algebraically closed (and thus perfect), $T$ and $T'$ are isomorphic as abstract schemes. Crucially, however, they are generally not isomorphic as $k$-schemes. A $k$-isomorphism must preserve the multiset of $k$-isomorphism classes of the pointed irreducible components (components equipped with their special points). Consider a component $(T_Z, S_Z)$ where $F$ acts non-trivially. In $T'$, it is replaced by the twisted component $(T_Z^{(2)}, F(S_Z))$. The relative Frobenius generally alters the $k$-isomorphism class. For instance, if $|S_Z| = 4$, the isomorphism class is determined by the cross-ratio $\lambda$ of the four points. The corresponding points on $T_Z^{(2)}$ then have a cross-ratio of $\lambda^2$. These configurations are projectively equivalent if and only if $\lambda^2$ is in the $S_3$-orbit of $\lambda$, which in characteristic $2$ means $\lambda \in \mathbb{F}_4$. If the modulus $\lambda$ is generic ($\lambda \notin \mathbb{F}_4$), the isomorphism class of the component changes. Provided the structure of $T$ is sufficiently asymmetric, this change in component moduli implies that the multisets of invariants for $T$ and $T'$ differ. Therefore, $T$ and $T'$ are generally not isomorphic as $k$-schemes.
\end{remark}

\begin{remark}[The Moduli-Theoretic Definition]
    Consider the unpointed case. One can then define hyperelliptic stable curves from the perspective of moduli spaces. Let $\mathcal{H}_{g} \subset \mathcal{M}_{g}$ be the locus of smooth hyperelliptic curves inside the moduli space of smooth curves of genus $g$. The space of hyperelliptic stable curves is defined as the closure of this locus in the Deligne-Mumford compactification, denoted $\overline{\mathcal{H}}_{g} \subset \overline{\mathcal{M}}_{g}$.

    A non-trivial theorem establishes that this moduli-theoretic definition coincides with the geometric one presented in Definition~\ref{def:hyperelliptic}, which is based on the existence of an involution (see \cite[Theorem 5.5]{hyperelliptic}; see also \cite{ACGII} for $k = \mathbb{C}$, and \cite{maugeais2003relevement}, \cite[Appendix A.2]{yamaki2004cornalba} for general fields, including $\Char k = 2$). In other words, a stable curve admits a hyperelliptic involution if and only if it is a limit of smooth hyperelliptic curves.
\end{remark}

As in the smooth case, the hyperelliptic involution is unique if it exists (cf. \cite[Chapter X, Lemma 3.5]{ACGII}). We provide a proof of this uniqueness, as the underlying decomposition method is essential for our subsequent classification of genus $3$ curves and their hyperellipticity. To do so, we first need to introduce some combinatorial terminology, inspired by \cite{ran2014canonical}.

\begin{definition}[Combinatorial Terminology] \label{def:comb_term}
    Let $C$ be a semistable curve over $k$.
    \begin{enumerate}[label=\DefListLabel]
        \item A node $p$ is called a \emph{separating node} if the partial normalization of $C$ at $p$ disconnects the curve. (In the dual graph, the corresponding edge is a bridge.)
        \item A pair of distinct nodes $\{p, q\}$ is called a \emph{separating pair} if the partial normalization at $\{p, q\}$ disconnects the curve, while normalizing at either $p$ or $q$ individually does not. (In the dual graph, the corresponding edges form a minimal 2-edge-cut.)
        \item The curve $C$ is called \emph{inseparable}\footnote{We follow the terminology of \cite{ran2014canonical}. This geometric notion of ``inseparable'' should not be confused with the standard algebraic concept of an inseparable morphism or field extension. The same property is referred to as ``2-connected'' by Catanese \cite{catanese1982pluricanonical}.} if it has no separating nodes, and \emph{separable} otherwise. It is called \emph{2-inseparable} if it is inseparable and has no separating pairs, and \emph{2-separable} otherwise. (In terms of the dual graph, these properties correspond to the graph being 2-edge-connected and 3-edge-connected, respectively.)
        \item The curve $C$ is of \emph{semicompact type} if every node $p$ is part of at most one separating pair.
    \end{enumerate}
\end{definition}

\begin{figure}[htbp]
    \centering
    \begin{subfigure}[b]{0.24\textwidth}
        \centering
        \scalebox{1.75}{%
            \begin{tikzpicture}[scale=0.5]
                \useasboundingbox (-0.0168,-0.7627) rectangle (2.0051,1.2627);
                \clip (-0.0168,-0.7627) rectangle (2.0051,1.2627);
                
                \draw[line width = 2, name path=path1] (0, 0) to[curve through={(1.2, 0.2) (1.4, 0.4) (1.2, 0.8) (1, 0.4) (1.2, 0.2)}] (2, 0);
                \draw[line width = 2, name path=path2] (0.4, 1) to (0.4, -0.5);
                
                \fill[myblue, name intersections={of=path1 and path2, by=S}] (S) circle (0.12);
            \end{tikzpicture}%
        }
        \caption{\textcolor{myblue}{Separating node}.}
        \label{fig:1ne}
    \end{subfigure}
    \hfill
    \begin{subfigure}[b]{0.24\textwidth}
        \centering
        \scalebox{1.75}{%
            \begin{tikzpicture}[scale=0.15]
                \useasboundingbox (-2.3698,0.1242) rectangle (4.3698,6.8755);
                \clip (-2.3698,0.1242) rectangle (4.3698,6.8755);
                
                \draw[name path=path1] (-0.5, 0.5) to[curve through={(2, 3)}] (-0.5, 6.5);
                \draw[name path=path2] (2.5, 0.5) to[curve through={(0, 3)}] (2.5, 6.5);
                
                \fill[myred, name intersections={of=path1 and path2, name=S}] 
                    (S-1) circle (0.25) 
                    (S-2) circle (0.25);
            \end{tikzpicture}%
        }
        \caption{\textcolor{myred}{Separating pair}.}
        \label{fig:Z}
    \end{subfigure}
    \hfill
    \begin{subfigure}[b]{0.24\textwidth}
        \centering
        \scalebox{1.75}{%
            \begin{tikzpicture}[scale=0.2]
                \useasboundingbox (-1.5273,0.4665) rectangle (3.5273,5.53);
                \clip (-1.5273,0.4665) rectangle (3.5273,5.53);
                
                \draw (0, 0.5) to[quick curve through={(1, 1) (2, 2) (1, 3) (0, 4) (1, 5)}] (2, 5.5);
                \draw (2, 0.5) to[quick curve through={(1, 1) (0, 2) (1, 3) (2, 4) (1, 5)}] (0, 5.5);
            \end{tikzpicture}%
        }
        \caption{2-inseparable}
        \label{fig:0---0}
    \end{subfigure}
    \hfill
    \begin{subfigure}[b]{0.24\textwidth}
        \centering
        \scalebox{1.75}{%
            \begin{tikzpicture}[scale=0.5]
                \useasboundingbox (-0.0109,-0.6627) rectangle (2.0109,1.3627);
                \clip (-0.0109,-0.6627) rectangle (2.0109,1.3627);
                
                \draw[line width = 2] (0, 0) to (2, 0);
                \draw[line width = 2] (1.3, 1) to (0.1, -0.3);
                \draw[line width = 2] (0.7, 1) to (1.9, -0.3);
            \end{tikzpicture}%
        }
        \caption{Not semicompact}
        \label{fig:not_semicompact}
    \end{subfigure}

    \caption{Examples of the properties in Definition \ref{def:comb_term}.}
    \label{fig:curve_properties}
\end{figure}
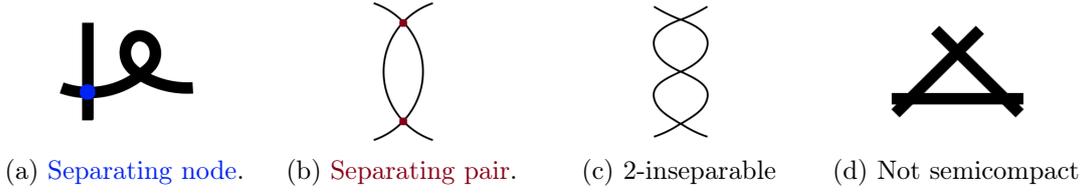

A separating node or a separating pair naturally induces a decomposition of the semistable curve $C$ into two connected subcurves, $C_1$ and $C_2$, such that $C = C_1 \cup C_2$. The intersection $C_1 \cap C_2$ consists of either the single separating node or the two nodes of the separating pair. In the terminology introduced earlier, this corresponds to the decomposition induced by the set of nodes $S$, where $S=\{p\}$ or $S=\{p,q\}$, respectively.
Given an involution $\sigma$ on $C$, we say that it \emph{respects} this decomposition if it preserves each component, i.e., $\sigma(C_i) = C_i$ for $i=1,2$.

With these notions in place, let us classify the types of nodes that appear on a hyperelliptic stable curve.

\begin{proposition}[{{cf. \cite[p. 102]{ACGII}}}] \label{prop:node_classification}
    Let $(C, P)$ be a pointed stable hyperelliptic curve with hyperelliptic involution $\sigma$, and let $\pi\colon C \to T$ be the quotient map. Any node $p \in C$ falls into one of the following three mutually exclusive categories:
    \begin{enumerate}[label=\StatementListLabel]
        \item \label{lem:properties_nodes_sep} \textbf{Separating Node.} A node $p$ is a separating node if and only if it is a fixed point of $\sigma$ at which $\sigma$ preserves the local branches. In this case, its image $\pi(p)$ is a node of $T$, and $\sigma$ respects the decomposition of $C$ induced by $p$.
        \item \label{lem:properties_nodes_pair} \textbf{Part of a Separating Pair.} A node $p$ is part of a separating pair if and only if it is not a fixed point of $\sigma$. In this case, its image under the involution, $p' := \sigma(p)$, is the unique node that forms a separating pair $\{p, p'\}$ with $p$. Both points map to a single node $\pi(p) = \pi(p')$ in $T$, and $\sigma$ respects the decomposition of $C$ induced by the pair $\{p, p'\}$.
        \item \label{lem:properties_nodes_non} \textbf{Non-Separating Fixed Node.} A node $p$ is a non-separating fixed node if and only if it is a fixed point of $\sigma$ at which $\sigma$ interchanges the local branches. In this case, $p$ is neither a separating node nor part of a separating pair, its image $\pi(p)$ is a smooth point of $T$, and the irreducible components of $C$ forming the node $p$ are either identical or both rational.
    \end{enumerate}
    Moreover, the curve $C$ is of semicompact type, and the rational tree $T = C/\langle \sigma \rangle$ is irreducible (i.e., $T \simeq \PP^1_k$), if and only if $C$ is 2-inseparable.\footnote{This fact seems to be folklore, and we do not know of an explicit reference; cf. \cite[Remark 2.14]{ran2014canonical}.}
\end{proposition}

\begin{remark} \label{rem:types_of_points_fixpoints}
    \begin{enumerate}[label=\StatementListLabel]
        \item In the unpointed case, the following standard terminology is used: A non-separating fixed node (category~\ref{lem:properties_nodes_non}) is referred to as type~$\eta_0$. A separating node (category~\ref{lem:properties_nodes_sep}) is known as type~$\delta_i$; it decomposes the genus~$g$ curve into two components with genera~$i$ and~$j$ (with $i \le j$) satisfying the relation $i+j=g$. Finally, a separating pair of nodes (category~\ref{lem:properties_nodes_pair}) is of type~$\eta_i$; this pair decomposes the curve into two components with genera~$i$ and~$j$ (with $i \le j$) satisfying $i+j=g-1$.

        From a moduli-theoretic perspective, this classification of nodes provides a geometric description of the irreducible components of the boundary of the hyperelliptic locus, $\partial \mathcal{H}_g = \overline{\mathcal{H}}_g \setminus \mathcal{H}_g$, see \cite[Chapter XIII, Section 8]{ACGII}.
        \item \label{rem:types_of_points_fixpoints_ii} We will see in the proof of this result that the property that the hyperelliptic involution $\sigma$ has isolated fixed points in $\Char k \neq 2$ is not needed. The same holds true for the uniqueness result (Proposition~\ref{prop:uniqueness_hyperelliptic}).
    \end{enumerate}
\end{remark}

The proof of this result requires two lemmas regarding decompositions of the curve. We begin with the following lemma, which describes how hyperelliptic involutions act on the decomposed parts.

\begin{lemma} \label{lem:orbit_decomposition_preserved}
    Let $(C, P)$ be a stable hyperelliptic curve with a hyperelliptic involution $\sigma$. Let $S$ be a set of nodes that is $\sigma$-invariant, i.e., $\sigma(S) = S$. Let $\nu\colon \tilde{C} \to C$ be the partial normalization at $S$, and let $\{(C_1, P_1), \dots, (C_m,P_m)\}$ be the induced decomposition, and $S_i = \nu^{-1}(S) \cap C_i$.
    \begin{enumerate}[label=\StatementListLabel]
        \item \label{lem:orbit_decomposition_preserved_i} The involution $\sigma$ lifts to a unique involution of $\tilde{C}$, still denoted by $\sigma$, which permutes the components $\{C_i\}$. Consequently, $\sigma$ either fixes a component $C_i$ or swaps it with another component $C_j$ (for $i \neq j$).
        \item \label{lem:orbit_decomposition_preserved_ii} If $\sigma$ swaps $C_i$ with $C_j$ ($i \neq j$), then $C_i$ is a rational tree and $|P_i| = |S_i| \geq 3$. In particular, any component $C_i$ with $|S_i| \leq 2$ must be fixed by $\sigma$.
        \item \label{lem:orbit_decomposition_preserved_iii} If $|S| \leq 2$, the involution $\sigma$ must fix every component of the decomposition.
    \end{enumerate}
\end{lemma}

\begin{proof}
    \begin{prooflist}
        \item Since $\sigma\colon C \to C$ is an automorphism stabilizing the normalization locus $S$, by the universal property of normalization, $\sigma$ lifts uniquely to an automorphism $\tilde{\sigma}\colon \tilde{C} \to \tilde{C}$ such that $\nu \circ \tilde{\sigma} = \sigma \circ \nu$. Since $\sigma^2=\id_C$, we must have $\tilde{\sigma}^2=\id_{\tilde{C}}$ by uniqueness. We denote $\tilde{\sigma}$ again by $\sigma$. As an automorphism of $\tilde{C}$, it permutes the connected components $\{C_i\}$.

        \item Now, assume $\sigma(C_i) = C_j$ with $i \neq j$. We may identify $\nu^{-1}(P)$ with $P$. The (lift of the) hyperelliptic involution $\sigma$ must fix the points in $P$, and since $\sigma(C_i) = C_j$ and $C_i, C_j$ are disjoint, we conclude that $P \cap C_i$ is empty. Thus $P_i = S_i$.

        We now show that $C_i$ is a rational tree. Let $\tilde{T} := \tilde{C}/\langle\sigma\rangle$ be the quotient of the partial normalization by the lifted involution, and let $\tilde{\pi}\colon \tilde{C} \to \tilde{T}$ be the quotient map. Since $\nu\colon \tilde{C} \to C$ is $\sigma$-equivariant (as established in \ref{lem:orbit_decomposition_preserved_i}), it descends to a morphism $\tilde{\nu}\colon \tilde{T} \to T$, where $T=C/\langle\sigma\rangle$. The map $\tilde{\nu}$ is the partial normalization of $T$ at the image of the nodes $S$ under the quotient map $\pi\colon C \to T$. By definition of a hyperelliptic curve, $T$ is a rational tree, and hence $\tilde{T}$ is a disjoint union of rational trees. Let $\tilde{T}_i$ be the connected component of $\tilde{T}$ that is the image of $C_i$ (and $C_j$) under $\tilde{\pi}$. The restriction $\tilde\pi|_{C_i}\colon C_i \to \tilde{T}_i$ yields an isomorphism, so $C_i$ (and also $C_j$) is a rational tree.

        Now we use the stability condition. By Lemma~\ref{lem:stable_decomposition}, since $(C, P)$ is stable, the component $(C_i, P_i)$ must also be stable. Since $C_i$ has genus $0$, it follows that $|S_i| = |P_i| \geq 3$ (see Remark~\ref{rem:stable_curves_low_genus}).

        \item Assume for contradiction that $\sigma$ swaps components $C_i$ and $C_j$ for $i \neq j$. By Lemma~\ref{lem:orbit_decomposition_preserved}~\ref{lem:orbit_decomposition_preserved_ii}, this requires $|S_i|, |S_j| \geq 3$. Since the total number of preimages is $\sum_l |S_l| = 2|S|$, the hypothesis $|S| \leq 2$ immediately leads to a contradiction:
        \[
            6 \leq |S_i| + |S_j| \leq \sum_l |S_l| = 2|S| \leq 4.
        \]
    \end{prooflist}
\end{proof}

The proof of Proposition~\ref{prop:node_classification} involves arguments that decompose the curve at specific nodes; this process can quickly lead to confusion about the structure of the resulting decomposition. For this reason, we formulate the outcome more generally in the following lemma. Note that this lemma is essentially a graph-theoretical statement.

\begin{lemma} \label{lem:chains_of_cuts}
    Let $C$ be a semistable curve, and let $S = \{p_1, \dots, p_n\}$ be a set of $n$ distinct nodes. Let $\{(C_1, S_1), \dots, (C_m, S_m)\}$ be the decomposition induced by $S$. In both cases below, the two preimages of any node in $S$ lie on distinct components of the decomposition. This allows each component $C_i$ to be identified with a subcurve of $C$; under this identification, $S_i \subset S$.
    \begin{enumerate}[label=\StatementListLabel]
        \item \label{lem:chains_of_cuts_i} Assume each node in $S = \{p_1, \dots, p_n\}$ is a separating node of $C$. Then the decomposition has $m = n+1$ components arranged in a tree-like structure; that is, there is a unique chain of components and nodes in $S$ connecting any two distinct components.
        \item \label{lem:chains_of_cuts_ii} Assume $\{p_i, p_{i+1}\}$ are separating pairs for all $i = 1, \dots, n-1$ ($n \geq 2$). Then $\{p_i, p_j\}$ is a separating pair for any $i \neq j$. The decomposition has $m = n$ components. After relabeling, the subcurves and nodes form a cycle where $p_i$ connects $C_i$ and $C_{i+1}$ for $i=1, \dots, n-1$, and $p_n$ connects $C_n$ back to $C_1$. Consequently, the set of preimages on each component is $S_i = \{p_{i-1}, p_i\}$ for $i=1, \dots, n$, where indices are read modulo $n$.
    \end{enumerate}
\end{lemma}

\begin{proof}
We translate this problem into the language of graph theory by considering the associated graph $\Gamma := \Graph^S(C,\emptyset)$, where vertices correspond to the components $C_i$ and edges correspond to the nodes in $S$. Since $C$ is connected, $\Gamma$ is also connected. The claim in the lemma's preamble---that the two preimages of any node in $S$ lie on distinct components---is equivalent to the graph-theoretic statement that $\Gamma$ has no loops.
\begin{prooflist}
\item The assumption that each node in $S$ is a separating node means that every edge in the associated graph $\Gamma$ is a bridge. Since $\Gamma$ is connected, a graph in which every edge is a bridge must be a tree (in particular, with no loops). A tree with $n$ edges has $n+1$ vertices, which proves the lemma's assertion that the decomposition has $m=n+1$ components with a tree-like structure.

\item The lemma's assertion is equivalent to the graph-theoretic statement that $\Gamma$ is a cycle of length $n$, where for $n=2$ a cycle is understood to be two vertices connected by two parallel edges. The assumptions on $S$, in graph-theoretic terms, mean that each pair of edges corresponding to $\{p_i, p_{i+1}\}$ is a minimal 2-edge-cut. Moreover, since every edge is part of a minimal 2-edge-cut, it follows that any edge is neither a loop nor a bridge. Since $\Gamma$ is connected, this implies that it is 2-edge-connected.

We now argue that every pair of edges forms a minimal 2-edge-cut. A standard result in graph theory states that if $\{e_1, e_2\}$ and $\{e_2, e_3\}$ are edge-cuts, then so is their symmetric difference $\{e_1, e_3\}$. Since $\Gamma$ is 2-edge-connected, any 2-edge-cut is automatically minimal. Thus, by induction, it follows that any pair of edges corresponding to $\{p_i, p_j\}$ for $i \neq j$ is a minimal 2-edge-cut.

The result will follow if we can show that every vertex in $\Gamma$ has degree $2$. Since $\Gamma$ is 2-edge-connected, every vertex has a degree of at least $2$. Suppose, for contradiction, that there exists a vertex $v$ with $\deg(v) \geq 3$. Let $e_1, e_2, e_3$ be three distinct edges incident to $v$, with other endpoints $u_1, u_2, u_3$. (Since $\Gamma$ has no loops, $u_i \neq v$). Consider the cut $\{e_1, e_2\}$. Since it is a minimal 2-edge-cut, it partitions the vertices $V = V(\Gamma)$ into exactly two connected components, $V_1$ and $V_2$. Assume $v \in V_1$. Since $e_1=(v, u_1)$ and $e_2=(v, u_2)$ are the only edges between $V_1$ and $V_2$, we must have $u_1, u_2 \in V_2$. The endpoint $u_3$ must also be in $V_1$, since otherwise the edge $e_3 = (v,u_3)$ would be a third edge in the cut.

Now consider the pair $\{e_1, e_3\}$, which, as established above, must also be a minimal 2-edge-cut. Let $\Gamma' = \Gamma \setminus \{e_1, e_3\}$. Since this cut contains the edge $e_1=(v,u_1)$, the vertices $v$ and $u_1$ must lie in different connected components of $\Gamma'$. However, a path between them still exists in $\Gamma'$: since the subgraph of $\Gamma$ induced by $V_2$ is connected, there is a path $P$ between $u_1$ and $u_2$ contained entirely in $V_2$. This path $P$ does not use any of the edges $\{e_1, e_2, e_3\}$, since all three edges are incident to the vertex $v \in V_1$. The edge $e_2=(v,u_2)$ also remains in $\Gamma'$. Concatenating $e_2$ with $P$ yields a path from $v$ to $u_1$ in $\Gamma'$. This contradicts the fact that $v$ and $u_1$ must be disconnected. Therefore, the assumption of a vertex with degree $\geq 3$ is false.
\end{prooflist} 
\end{proof}

With the necessary preparations in place, we now present the proof.

\begin{proof}[of Proposition \ref{prop:node_classification}]
    We first address the final statement, assuming the classification holds.
    The property of being semicompact (Definition~\ref{def:comb_term}~(iv)) requires that for any node $p$, there is at most one other node $q$ such that $\{p, q\}$ forms a separating pair. This follows immediately from part~\ref{lem:properties_nodes_pair}, which states that if $p$ is part of a separating pair, its unique partner is $\sigma(p)$.

    The equivalence between $C$ being 2-inseparable and $T$ being irreducible also follows. $C$ is 2-inseparable if and only if it has no separating nodes (type~\ref{lem:properties_nodes_sep}) and no separating pairs (type~\ref{lem:properties_nodes_pair}).
    If $C$ is 2-inseparable, all its nodes must be of type~\ref{lem:properties_nodes_non}. By the characterization of this type, their images under $\pi$ are smooth points. Thus, $T$ has no nodes. As $T$ is a rational tree, it must be irreducible ($T \simeq \PP^1_k$).
    Conversely, if $T$ is irreducible, it has no nodes. By~\ref{lem:properties_nodes_sep} and~\ref{lem:properties_nodes_pair}, $C$ cannot have nodes of these types, as they would map to nodes in $T$. Thus, $C$ must be 2-inseparable.

    We now prove the classification of the nodes. Let $p$ be a node of $C$.

    \begin{prooflist}
        \item First, assume that $p$ is a separating node. Let $p' = \sigma(p)$; since $\sigma$ is an automorphism, $p'$ is also a separating node. We aim to show that $p$ is a fixed point of $\sigma$.
        Consider the $\sigma$-invariant set of nodes $S=\{p,p'\}$. Since $|S| \le 2$, we know from Lemma~\ref{lem:orbit_decomposition_preserved}~\ref{lem:orbit_decomposition_preserved_iii} that $\sigma$ must fix every component of the decomposition induced by $S$.
    
        Now, assume for contradiction that $p \neq p'$, which means $|S|=2$. By Lemma~\ref{lem:chains_of_cuts}~\ref{lem:chains_of_cuts_i}, the decomposition of $C$ by $S$ consists of a chain of three components, which we can label $C_1, C_2, C_3$, such that $p$ connects $C_1$ to $C_2$ and $p'$ connects $C_2$ to $C_3$. The action of $\sigma$ swaps the nodes $p$ and $p'$, and hence $\sigma(C_1) = C_3$, contradicting our finding that $\sigma$ must fix every component. Thus $p=p'$, meaning $p$ is a fixed point of $\sigma$, and the decomposition therefore consists of only two components, which are preserved by~$\sigma$.
    
        Furthermore, since $\sigma$ fixes both the node $p$ and the two components meeting there, it must preserve the branches at $p$. Consequently, the distinct branches at $p$ descend to distinct branches in the quotient, which means the image $\pi(p)$ is a node of $T$.
    
        Conversely, assume that $\sigma$ fixes $p$ and preserves its branches. As just shown, this implies its image $\pi(p)$ is a node in the rational tree $T$. Any node of a tree is a separating point. Since the fiber $\pi^{-1}(\pi(p))$ consists only of the point $p$, it follows that $p$ is a separating node of $C$.

        \item First, assume that $p$ is not a fixed point of $\sigma$. Let $p' = \sigma(p) \neq p$. Since neither $p$ nor $p'$ is a fixed point, by~\ref{lem:properties_nodes_sep} they are not separating nodes. Our aim is to show that $\{p, p'\}$ is a separating pair. The fiber of the quotient map over $\pi(p)$ is the orbit $\{p, p'\}$. Since this fiber consists of two distinct points, the image $\pi(p)$ must be a node of the tree $T$. Any node of a tree is a separating point, which implies that normalizing $C$ at $\{p, p'\}$ disconnects the curve. We thus conclude that $\{p, p'\}$ is a separating pair.

        Conversely, assume $p$ is part of a separating pair $\{p, q\}$; we recall that this implies $p$ and $q$ are distinct, non-separating nodes. We aim to show that $q = \sigma(p)$. This will establish not only that $p$ is not a fixed point, but also that $\sigma(p)$ is the unique node that forms a separating pair with $p$. Let $p' = \sigma(p)$ and $q' = \sigma(q)$, and consider the $\sigma$-invariant set $S = \{p, q, p', q'\}$. A priori, we only know that $2 \le |S| \le 4$. Since $\{p,q\}$ is a separating pair, so is its image $\{p',q'\}$. As established in the first part of this proof, if $p \neq p'$, then $\{p,p'\}$ is also a separating pair (and likewise for $\{q,q'\}$). This implies that the assumptions of Lemma~\ref{lem:chains_of_cuts}~\ref{lem:chains_of_cuts_ii} are satisfied.
    
        Thus, $S$ induces a decomposition of $C$ into $m=|S|$ components, $C_1, \dots, C_m$, arranged in a cycle where each component has exactly two attaching nodes, i.e., $|S_i|=2$. By Lemma~\ref{lem:orbit_decomposition_preserved}~\ref{lem:orbit_decomposition_preserved_ii}, the condition $|S_i|=2$ implies that $\sigma$ must fix every component, $\sigma(C_i)=C_i$. We now argue that the case $m>2$ is impossible. In a cycle with more than two components, the two nodes in any $S_i$ connect $C_i$ to two different components. Since $\sigma$ fixes $C_i$ and the adjacent components, it must fix the connecting nodes in $S_i$ individually. As this holds for all $i$, $\sigma$ must fix every point in $S$. But this would imply $p=p'$ and $q=q'$, so $|S|=2$, contradicting our assumption that $m>2$.
    
        Therefore, we must have $m=2$, which means $S=\{p,q\}$. This leaves two possibilities: either $\sigma$ fixes both $p$ and $q$, or it swaps them.
        If $\sigma$ fixes $p$, it must interchange the branches at $p$, since $p$ is a non-separating fixed node (by~\ref{lem:properties_nodes_sep}). In the decomposition, the two local branches at $p$ correspond to the components $C_1$ and $C_2$. Interchanging them would mean $\sigma(C_1)=C_2$. This contradicts our finding that $\sigma$ must fix each component.
    
        The only remaining possibility is that $p'=q$. Furthermore, since $\sigma$ fixes the components $C_1$ and $C_2$, it respects the decomposition induced by $\{p,q\} = \{p, p'\}$. This completes the proof of part~\ref{lem:properties_nodes_pair}.

        \item A node $p$ falls into this category if it is a fixed point of $\sigma$ that is not a separating node. By the characterization in part~\ref{lem:properties_nodes_sep}, this is equivalent to the condition that $\sigma$ interchanges the branches at~$p$. In this case, the quotient map $\pi$ locally identifies these two branches, meaning the image $\pi(p)$ is a smooth point of the quotient tree~$T$.

        Finally, let $Z_1$ and $Z_2$ be the irreducible components forming the node~$p$. Since $\sigma$ interchanges the branches, it must swap the components, i.e., $\sigma(Z_1)=Z_2$. If these components are distinct, the restriction of the quotient map, $\pi|_{Z_1}\colon Z_1 \to T$, is an isomorphism onto its image. As the image is an irreducible subcurve of the rational tree~$T$, it follows that $Z_1$ (and therefore $Z_2$) must be rational.
    \end{prooflist}
\end{proof}

\begin{remark}
    The stability of $(C,P)$ is a crucial hypothesis used in this proof. One can easily construct examples of semistable curves admitting a hyperelliptic involution for which the analogous characterization of the nodes fails. For instance, such a curve is not necessarily of semicompact type.
\end{remark}

Our next goal is to prove that the hyperelliptic involution on a stable curve is unique. The proof will proceed by induction on the number of separating nodes and separating pairs, thereby reducing the problem to the base case of a 2-inseparable curve. We will first explain how to handle separating nodes, followed by separating pairs, and finally the 2-inseparable base case.

\begin{lemma}[Decomposition at a Separating Node] \label{lem:decomp_sep_node}
    Let $(C, P)$ be a stable \linebreak pointed curve, and let $p \in C$ be a separating node that decomposes $(C,P)$ into two stable pointed curves, $(C_1, P_1)$ and $(C_2, P_2)$. Then $(C, P)$ is hyperelliptic if and only if both $(C_1, P_1)$ and $(C_2, P_2)$ are hyperelliptic.

    Moreover, there is a one-to-one correspondence between hyperelliptic involutions on $(C,P)$ and pairs of hyperelliptic involutions on the components. Specifically, any hyperelliptic involution $\sigma$ on $(C, P)$ restricts to unique hyperelliptic involutions on $(C_1, P_1)$ and $(C_2, P_2)$. Conversely, any pair of such involutions on the components can be uniquely ``glued'' to form a hyperelliptic involution on $(C, P)$.
\end{lemma}

\begin{proof}
    First, assume $(C, P)$ is hyperelliptic and let $\sigma$ be a hyperelliptic involution. Since $p$ is a separating node, by Proposition~\ref{prop:node_classification}~\ref{lem:properties_nodes_sep}, $p$ is a fixed point of $\sigma$ and $\sigma$ respects the decomposition induced by $p$. This means $\sigma$ fixes each component, so the restrictions $\sigma_1 = \sigma|_{C_1}$ and $\sigma_2 = \sigma|_{C_2}$ are well-defined involutions. Each $\sigma_i$ is a hyperelliptic involution on $(C_i, P_i)$, since it fixes the marked points $P_i$ and the quotient $C_i/\langle\sigma_i\rangle$ is a component of the rational tree $C/\langle\sigma\rangle$. Furthermore, if $\Char k \neq 2$, the isolated fixed-point condition for $\sigma$ implies the same for each $\sigma_i$. Thus, both $(C_1, P_1)$ and $(C_2, P_2)$ are hyperelliptic.

    Conversely, let $\sigma_1$ and $\sigma_2$ be hyperelliptic involutions on $(C_1, P_1)$ and $(C_2, P_2)$, respectively. By convention, the preimages of the node $p$ on the components are also denoted by $p$. As a marked point in both $P_1$ and $P_2$, $p$ is fixed by the respective involutions. Because the involutions agree on the point to be identified, they can be uniquely ``glued'' to define a global involution $\sigma$ on $(C, P)$. This resulting map $\sigma$ is a hyperelliptic involution, as its quotient $C/\langle\sigma\rangle$ is formed by joining the two rational trees $C_1/\langle\sigma_1\rangle$ and $C_2/\langle\sigma_2\rangle$ at a point, which is itself a rational tree. If $\Char k \neq 2$, the fixed points of $\sigma$ are the union of the isolated fixed points of $\sigma_1$ and $\sigma_2$, and are thus also isolated. Therefore, $(C, P)$ is hyperelliptic.
\end{proof}

Using the previous lemma repeatedly allows us to reduce the analysis of hyperellipticity to the case of inseparable curves. The next step is to handle separating pairs, reducing the problem further to the 2-inseparable case. This requires a specific modification of the components involved in the decomposition.

We first introduce the necessary construction.

\begin{definition} \label{def:contracted_components}
    Let $(C, P)$ be a stable pointed curve and let $S=\{p, q\}$ be a separating pair of nodes. Let $\{(C_1, P_1), (C_2, P_2)\}$ be the decomposition induced by $S$, and let $S_i$ be the set of preimages of $S$ on each component $C_i$.

    We define the associated \emph{contracted component} $(C_i', P_i')$ as follows: the curve $C_i'$ is obtained from $C_i$ by identifying the two smooth points in $S_i$ to form a new node $r_i$. The curve $C_i$ is thus the partial normalization of $C_i'$ at this new node $r_i$ (cf. \cite[Chapter 3.IV]{serre1988algebraic} for the formal construction). The new set of marked points is $P_i' := P_i \setminus S_i$.
\end{definition}

We first ensure that this construction preserves stability and then analyze the nature of the new node.

\begin{lemma} \label{lem:contracted_component_stable}
    The contracted components $(C_i', P_i')$ as defined above are stable pointed curves, and the node $r_i$ is not a separating node. Moreover, if $(C,P)$ is of semicompact type, the new node $r_i$ is not part of a separating pair on $C_i'$.
\end{lemma}

\begin{proof}
    By Lemma~\ref{lem:stable_decomposition}, $(C_i', P_i')$ is stable since its decomposition at $\{r_i\}$ consists of the single stable component $(C_i, P_i)$. Furthermore, since the normalization of $C_i'$ at $r_i$ is the connected curve $C_i$, the node $r_i$ is by definition not separating.

    Now assume $C$ is of semicompact type. Suppose there exists another node $t_i \in C_i'$ such that $\{r_i, t_i\}$ forms a separating pair on $C_i'$. Let $t \in C_i$ be the node on the original curve corresponding to $t_i$. The condition on the contracted curve implies that on the original curve $C$, the node $t$ forms a separating pair with both $p$ and $q$. This contradicts the assumption that $C$ is of semicompact type. Therefore, $r_i$ cannot be part of a separating pair on $C_i'$.
\end{proof}

\begin{lemma}[Decomposition at a Separating Pair] \label{lem:decomp_sep_pair}
    Let $(C, P)$ be a stable \linebreak pointed curve of semicompact type, and let $\{p,q\}$ be a separating pair. Let $(C_i', P_i')$ be the associated contracted components with new nodes $r_i$ (Definition~\ref{def:contracted_components}). Then $(C, P)$ is hyperelliptic if and only if both $(C_1', P_1')$ and $(C_2', P_2')$ are hyperelliptic.

    Moreover, any hyperelliptic involution $\sigma$ on $(C,P)$ induces unique hyperelliptic involutions on the contracted components $(C_i', P_i')$. Conversely, any pair of hyperelliptic involutions on the contracted components can be uniquely lifted and glued to form a hyperelliptic involution on $(C,P)$.
\end{lemma}

\begin{proof}
    First, assume $(C, P)$ is hyperelliptic with involution $\sigma$. Since $\{p,q\}$ is a separating pair, by Proposition~\ref{prop:node_classification}~\ref{lem:properties_nodes_pair}, we must have $q=\sigma(p)$. The proposition also states that $\sigma$ respects the decomposition, meaning it fixes each component $C_i$. The restriction $\sigma_i = \sigma|_{C_i}$ is therefore an involution on $(C_i, P_i)$. Since $\sigma$ swaps $p$ and $q$, the restriction $\sigma_i$ must swap the two points in the set $S_i$.

    This involution $\sigma_i$ on $C_i$ descends to a well-defined involution $\sigma_i'$ on the contracted component $C_i'$. Because $\sigma_i$ swaps the two points of $S_i$ that are identified to form the new node $r_i$, the descended involution $\sigma_i'$ must fix $r_i$. The quotient $C_i'/\langle\sigma_i'\rangle$ is isomorphic to $C_i/\langle\sigma_i\rangle$, which is a component of the rational tree $C/\langle\sigma\rangle$ and thus is itself a rational tree. Furthermore, if $\Char k \neq 2$, the isolated fixed-point condition for $\sigma$ is inherited by the descended involution $\sigma_i'$. Therefore, $(C_i', P_i')$ is hyperelliptic.

    Conversely, assume both $(C_1', P_1')$ and $(C_2', P_2')$ are hyperelliptic with involutions $\sigma_1'$ and $\sigma_2'$, respectively. By Lemma~\ref{lem:contracted_component_stable}, the new nodes $r_1$ and $r_2$ are non-separating and not part of any separating pair on their respective curves (since we assumed $C$ is of semicompact type). By Proposition~\ref{prop:node_classification}, this implies that each $r_i$ must be a fixed point of $\sigma_i'$ and that $\sigma_i'$ must interchange the branches at $r_i$.

    Lifting the involution $\sigma_i'$ from $C_i'$ to its normalization $C_i$ gives an involution $\sigma_i$ on $(C_i, P_i)$. The fact that $\sigma_i'$ interchanges the branches at $r_i$ means that the lifted involution $\sigma_i$ must swap the two points in $S_i$. We can now ``glue'' these involutions to define a global involution $\sigma$ on $(C,P)$. Since $\sigma_1$ swaps the preimages of $\{p,q\}$ on $C_1$ and $\sigma_2$ does the same on $C_2$, the global involution $\sigma$ is well-defined and swaps the nodes $p$ and $q$.

    The quotient $C/\langle\sigma\rangle$ is formed by gluing the two rational trees $C_1'/\langle\sigma_1'\rangle$ and $C_2'/\langle\sigma_2'\rangle$ along the images of the nodes $r_1$ and $r_2$. Since the involutions interchange the branches at these nodes, their images are smooth points in the quotients. Gluing two rational trees at a pair of smooth points results in a rational tree. If $\Char k \neq 2$, the fixed points of $\sigma$ are the union of the isolated fixed points of $\sigma_1'$ and $\sigma_2'$, and are thus also isolated. Therefore, $(C,P)$ is hyperelliptic.
\end{proof}

\begin{remark}
    The assumption that the curve is of semicompact type is crucial in the preceding lemma. To see why, consider the non-semicompact curve from Figure~\ref{fig:not_semicompact}. If this curve is decomposed at a separating pair, its contracted components can be shown to be hyperelliptic (see Lemma~\ref{lem:always_hyperelliptic}). Nevertheless, the original curve itself cannot be hyperelliptic, because Proposition~\ref{prop:node_classification} establishes that any stable hyperelliptic curve must be of semicompact type.
\end{remark}

We now address the 2-inseparable case and show that such a curve admits at most one hyperelliptic involution.

\begin{lemma}\textbf{(cf. \cite[Chapter X, Lemma 3.5]{ACGII}, \cite[Proposition 2.12]{ran2014canonical})} \label{lem:unique_2_inse}
    Let $(C, P)$ be an n-pointed stable, 2-inseparable, hyperelliptic curve. Then its structure must be one of the following two types:
    \begin{enumerate}[label=\StatementListLabel]
        \item $C$ is irreducible.
        \item $C$ is a reducible curve called a \emph{binary curve}\footnote{cf.\ \cite[Proposition 7]{coelho2023gonality} for the hyperelliptic condition.}: the union of two smooth rational components connected at $g(C)+1 \ge 3$ nodes. In this case, there are no marked points ($n=0$).
    \end{enumerate}
    In both cases, the hyperelliptic involution on $(C, P)$ is unique.
\end{lemma}

\begin{proof}
Let $\sigma$ and $\tau$ be two hyperelliptic involutions on $(C, P)$. We aim to show they are identical.

Since $(C, P)$ is a 2-inseparable stable curve, we can apply Proposition~\ref{prop:node_classification}. The quotient curve $T := C/\langle\sigma\rangle$ must be irreducible, which for a rational tree means $T \simeq \PP^1_k$. The same holds for the quotient by $\tau$. Furthermore, 2-inseparability implies $C$ has no separating nodes (type~\ref{lem:properties_nodes_sep}) and no separating pairs (type~\ref{lem:properties_nodes_pair}). Consequently, every node of $C$ must be of type~\ref{lem:properties_nodes_non}: it is a fixed point of the involution, and the involution interchanges the branches at the node. This holds for both $\sigma$ and $\tau$.

Let $\nu\colon \tilde{C} \to C$ be the normalization of $C$, and let $\tilde{\sigma}$ and $\tilde{\tau}$ be the unique lifted automorphisms on $\tilde{C}$. Both lifts must fix the preimages of the marked points $P$ and swap the two preimages over each node of $C$.

By the universal property of normalization, the map $\nu$ is $\sigma$-equivariant, so it descends to a proper morphism on the quotients, $\bar{\nu}\colon \tilde{C}/\langle\tilde{\sigma}\rangle \to C/\langle\sigma\rangle$. Because $\tilde{C}$ is a smooth curve, its quotient $\tilde{C}/\langle\tilde{\sigma}\rangle$ is also a smooth proper curve. The target $C/\langle\sigma\rangle \simeq \PP^1_k$ is smooth as well. Since $\bar{\nu}$ is a birational morphism between smooth proper curves, it must be an isomorphism. Thus, we have established that $\tilde{C}/\langle\tilde{\sigma}\rangle \simeq \PP^1_k$ and, by the exact same argument, $\tilde{C}/\langle\tilde{\tau}\rangle \simeq \PP^1_k$.

\begin{case}
    $C$ is irreducible.
    The normalization $\tilde{C}$ is a single smooth, connected curve. We analyze the possibilities for its genus, $g(\tilde{C})$.
    \begin{description}
        \item[$g(\tilde{C}) \ge 2$:] Since $g(\tilde{C}) \ge 2$ and we have shown that $\tilde{C}/\langle\tilde{\sigma}\rangle \simeq \PP^1_k$, the lift $\tilde{\sigma}$ is, by definition, a hyperelliptic involution on the smooth curve $\tilde{C}$. The same holds for $\tilde{\tau}$. By the classical theorem, a smooth curve of genus $g \ge 2$ has at most one hyperelliptic involution. Therefore, we must have $\tilde{\sigma} = \tilde{\tau}$.

        \item[$g(\tilde{C}) = 1$:] The normalization $\tilde{C}$ is an elliptic curve. For $C$ to be stable, there must be at least one special point on $\tilde{C}$. We seek involutions on $\tilde{C}$ whose quotient is $\PP^1_k$.
        
        If there is a marked point $p \in P$ on $\tilde{C}$, both $\tilde{\sigma}$ and $\tilde{\tau}$ must fix it. There is a unique involution on an elliptic curve with a fixed point that yields a genus $0$ quotient (namely, the map $x \mapsto -x$ in terms of the group law with origin $p$). Thus $\tilde{\sigma} = \tilde{\tau}$.
        
        If there are no marked points, there is at least one node on $C$. Let $\{q, q'\}$ be its preimages on $\tilde{C}$. Both $\tilde{\sigma}$ and $\tilde{\tau}$ must swap these two points. While two distinct involutions might swap these points (if their difference is a 2-torsion point), they yield quotients of different genera. The reflection-type involution (e.g., $x \mapsto (q+q')-x$) yields a genus $0$ quotient, whereas a translation-type involution yields a genus $1$ quotient. Since we established that the quotient must be $\PP^1_k$, the involution is uniquely determined. Thus, $\tilde{\sigma} = \tilde{\tau}$.

        \item[$g(\tilde{C}) = 0$:] The normalization $\tilde{C}$ is isomorphic to $\PP^1_k$. The stability of $(C,P)$ requires that there are at least three special points on $\tilde{C}$. An automorphism of $\PP^1_k$ is uniquely determined by its action on three distinct points. Since $\tilde{\sigma}$ and $\tilde{\tau}$ have the exact same action (fixing preimages of $P$ and swapping preimages of nodes) on this set of at least three points, the automorphisms must be identical: $\tilde{\sigma} = \tilde{\tau}$.
    \end{description}
    In all subcases, the lifted involutions are identical, which implies $\sigma = \tau$.
\end{case}

\begin{case}
    $C$ is reducible.
    First, we determine the structure of $C$. The involution $\sigma$ acts on the set of irreducible components of $C$. Since the quotient curve $C/\langle\sigma\rangle \simeq \PP^1_k$ is irreducible, this action must be transitive.
    
    Because $C$ is reducible, it consists of at least two components. A transitive action of a group of order 2 (generated by $\sigma$) on a set of size $\ge 2$ implies that the set has exactly two elements. Thus, $C$ consists of exactly two irreducible components, $C_1$ and $C_2$, which are swapped by $\sigma$ (and by the exact same group-theoretic argument, they are swapped by $\tau$).
    
    This structure implies that there can be no marked points ($n=0$), as any marked point would have to be fixed by the involution, contradicting the fact that the components containing them are swapped. Let $\pi\colon C \to \PP^1_k$ be the quotient map. Since $C_1$ and $C_2$ are swapped, the restrictions $\pi|_{C_1}$ and $\pi|_{C_2}$ are both surjective birational maps onto $\PP^1_k$. A birational morphism to a smooth proper curve is an isomorphism, so the components must be smooth and rational: $C_1 \simeq C_2 \simeq \PP^1_k$. For the unpointed curve $C$ to be stable, each rational component must have at least three special points, which must be the nodes where they intersect. Combining this stability condition with the genus formula, it follows that the components meet at $g(C)+1 \ge 3$ nodes.
    
    Finally, we prove uniqueness. The lifts $\tilde{\sigma}$ and $\tilde{\tau}$ are automorphisms of the normalization $\tilde{C} = C_1 \sqcup C_2$. Their restrictions to $C_1$, $\tilde{\sigma}|_{C_1}$ and $\tilde{\tau}|_{C_1}$, are isomorphisms from $C_1$ to $C_2$. An isomorphism between copies of $\PP^1_k$ is uniquely determined by its action on three points. Since they connect at $\ge 3$ nodes, and both isomorphisms must map the preimages of the nodes on $C_1$ to the corresponding preimages on $C_2$, they are identical. Thus $\tilde{\sigma} = \tilde{\tau}$, which implies $\sigma=\tau$.
\end{case}
\end{proof}

Putting these results together, we deduce the general uniqueness theorem.

\begin{proposition} \label{prop:uniqueness_hyperelliptic}
    Let $(C, P)$ be an n-pointed stable hyperelliptic curve. Then its hyperelliptic involution is uniquely determined.
\end{proposition}

\begin{proof}
    The proof is by induction on the number of separating nodes and separating pairs. Lemma~\ref{lem:unique_2_inse} establishes the base case of a 2-inseparable curve. The decomposition lemmas, Lemma~\ref{lem:decomp_sep_node} and Lemma~\ref{lem:decomp_sep_pair}, provide the inductive step, as they establish a one-to-one correspondence between involutions on a curve and the involutions on its simpler components.
\end{proof}

\begin{remark}
    With the decomposition lemmas (Lemma~\ref{lem:decomp_sep_node} and Lemma~\ref{lem:decomp_sep_pair}), one can decompose any stable pointed curve of semicompact type into its 2-inseparable parts and determine if they are hyperelliptic using the criteria established in Lemma~\ref{lem:unique_2_inse} and its proof (see Figure \ref{fig:decomposition}). In the hyperelliptic case, the 2-inseparable parts are then in bijection with the irreducible components of the associated rational tree.
\end{remark}

Let us conclude this section by explaining the distinct behavior of hyperellipticity in characteristic $2$. For many combinatorial types, the property of being hyperelliptic is independent of the characteristic of the underlying field; it depends only on the presence of certain rational components.

\begin{proposition} \label{prop:hyperelliptic_genus2}
    Let $(C, P)$ be a stable pointed curve, and let $\sigma$ be an involution on $C$ (i.e., $\sigma^2 = \id_C$) that fixes every point in $P$ and for which the quotient curve $T = C/\langle \sigma \rangle$ is a rational tree. Let $Z \subset C$ be an irreducible component. Then the following are equivalent:
    \begin{enumerate}[label=\EquivListLabel]
        \item \label{prop:hyperelliptic_genus2_i} The restriction of the automorphism $\sigma$ to $Z$ is the identity.
        \item \label{prop:hyperelliptic_genus2_ii} The component $Z$ is rational ($Z \simeq \PP^1_k$) and intersects the rest of the curve only at separating nodes.
    \end{enumerate}
    In particular, $(C,P)$ is hyperelliptic if and only if either $\Char k = 2$ or it does not contain any rational component $Z$ satisfying property~\ref{prop:hyperelliptic_genus2_ii}.
\end{proposition}

\begin{proof}
The characterization of nodes in Proposition~\ref{prop:node_classification} applies, since its arguments do not require the involution to have isolated fixed points (Remark~\ref{rem:types_of_points_fixpoints}~\ref{rem:types_of_points_fixpoints_ii}).

\ref{prop:hyperelliptic_genus2_i} $\implies$ \ref{prop:hyperelliptic_genus2_ii}: Assume $\sigma|_Z$ is the identity. Then the quotient map $\pi|_Z\colon Z \to T$ is an isomorphism onto a component of the rational tree $T$, so $Z$ must be rational. Any node $p$ connecting $Z$ to the rest of the curve is fixed by $\sigma$ with its branches preserved. By Proposition~\ref{prop:node_classification}~\ref{lem:properties_nodes_sep}, $p$ must be a separating node.

\ref{prop:hyperelliptic_genus2_ii} $\implies$ \ref{prop:hyperelliptic_genus2_i}: Conversely, assume $Z \simeq \PP^1_k$ intersects the rest of the curve only at separating nodes $S_Z$. By Proposition~\ref{prop:node_classification}~\ref{lem:properties_nodes_sep}, any such node is fixed by $\sigma$, so $\sigma$ must preserve the component $Z$. The restriction $\sigma|_Z$ is thus an automorphism of $Z$ that fixes all points in $S_Z$ and in $P \cap Z$. Since the stability of $(C,P)$ requires $|S_Z \cup (P \cap Z)| \ge 3$, and the only automorphism of $\PP^1_k$ with at least three fixed points is the identity, we conclude that $\sigma|_Z = \id_Z$.

By the uniqueness of an involution fixing the marked points with a rational tree as quotient (Proposition~\ref{prop:uniqueness_hyperelliptic} and Remark~\ref{rem:types_of_points_fixpoints}~\ref{rem:types_of_points_fixpoints_ii}), the pointed curve $(C,P)$ is hyperelliptic if and only if the given involution $\sigma$ is itself a hyperelliptic involution. In $\Char k = 2$, this holds by definition. In $\Char k \neq 2$, the definition additionally requires that $\sigma$ has isolated fixed points, which is equivalent to it not being the identity on any irreducible component. The final claim of the proposition is then a direct consequence.
\end{proof}

\begin{remark}\label{rem:lcrt_catanese}
A component as described in \ref{prop:hyperelliptic_genus2_ii} is what Catanese \cite[Definition 3.2]{catanese1982pluricanonical} terms a "loosely connected rational tail".
\end{remark}

\subsection{\texorpdfstring{Stable Curves of Genus $g = 3$}{Stable Curves of Genus g = 3}}

We now consider the case of genus $g=3$ in detail. As mentioned in Remark~\ref{rem:comb_types}, there are 42 distinct combinatorial types of stable curves of this genus. To describe them, it is useful to observe that they are all constructed from a simpler ``core'' by attaching standard components called ``tails''.

\begin{definition}
    Let $C$ be a stable curve over $k$, and let $E$ be an irreducible component of $C$.
    \begin{enumerate}[label=\DefListLabel]
        \item $E$ is called a \emph{$1$-tail} if it is attached to the rest of the curve at a single node and has an arithmetic genus of $1$.
        \item An \emph{elliptic tail} is a $1$-tail that is smooth (i.e., has geometric genus $1$).
        \item A \emph{pig tail}\footnote{This is an actual term used in the literature; see, for instance, \cite[p. 122, 175]{harris1998moduli}.} is a $1$-tail that has a single node (i.e., has geometric genus $0$).
    \end{enumerate}
    The node at which a $1$-tail $E$ is attached is called the \emph{attachment point} of $E$.
\end{definition}

\begin{remark} \label{rem:elliptic_pig_tails}
    By Proposition~\ref{prop:stable_via_canonical}, any stable $1$-pointed curve $(E, p)$ of genus $1$ can have at most $2(1)-2+1 = 1$ irreducible components. Thus, $E$ must be irreducible and is therefore either smooth or has a single node. By a slight abuse of notation, we also call the pointed curve $(E, p)$ an \emph{elliptic tail} if $E$ is smooth, and a \emph{pig tail} if $E$ has a node.
\end{remark}

\begin{lemma} \label{lem:core_tail_lemma}
    Let $C$ be a stable curve over $k$ of genus $3$, and let $r$ be the number of $1$-tails of $C$. Then $0 \leq r \leq 3$. The curve has exactly $r$ separating nodes, which are precisely the attachment points of the $1$-tails. The partial normalization of $C$ at these nodes results in the disjoint union of $r$ $1$-tails and an inseparable curve $C_c$ of genus $g(C_c) = 3 - r$.
\end{lemma}

\begin{definition}
    For a stable curve $C$ over $k$ of genus $3$, we call the curve $C_c$ from Lemma~\ref{lem:core_tail_lemma} the \emph{core} of $C$.
\end{definition}

\begin{proof}
    Any attachment point of a $1$-tail is by definition a separating node. Conversely, let $p$ be a separating node of $C$, and let $(C_1, p)$ and $(C_2, p)$ be the decomposition induced by $p$. Let their arithmetic genera be $g_1$ and $g_2$, with $g_1 \leq g_2$. The genus is given by $g(C) = g_1 + g_2 = 3$. Since $C$ is stable, the pointed curve $(C_1, p)$ must also be stable. By Remark~\ref{rem:stable_curves_low_genus}, this requires its genus to be $g_1 \geq 1$, forcing $g_1 = 1$ and $g_2 = 2$. This means the component $C_1$ is, by definition, a $1$-tail. As the other component $C_2$ has genus $2$, it is not a $1$-tail, and hence $p$ is the attachment point of just one $1$-tail.

    It follows that $C$ has exactly $r$ separating nodes, which are precisely the attachment points of its $1$-tails. By Lemma~\ref{lem:chains_of_cuts}~\ref{lem:chains_of_cuts_i}, the partial normalization at the $r$ separating nodes yields a decomposition into $r+1$ connected components: the $r$ $1$-tails and a remaining curve, the core $C_c$. By construction, the core $C_c$ has no separating nodes, i.e., it is inseparable. Since attaching a $1$-tail to a curve increases its genus by one, we have the relation $g(C) = g(C_c) + r$. As $g(C)=3$, this gives $g(C_c) = 3-r$, which forces $0 \leq r \leq 3$ since the genus of the core must be non-negative.
\end{proof}

To efficiently reference the 42 types, we use a compact naming convention, largely following \cite{bommelcayley}. The name precisely describes the curve's combinatorial structure:

\begin{itemize}
    \item A number (\texttt{0}, \texttt{1}, \texttt{2}, or \texttt{3}) indicates an irreducible component of that \emph{geometric} genus.
    \item The letter \texttt{n} appended to a genus number (e.g., \texttt{0n}) signifies a self-intersection (a node) on that component.
    \item The letter \texttt{e} denotes an \emph{elliptic tail} attached to the previously mentioned component.
    \item The letter \texttt{m} denotes a \emph{pig tail} (from ``multiplicative reduction'') attached to the previously mentioned component.
    \item Special characters describe intersections between two components: \texttt{=} for two intersection points, \texttt{-{}-{}-} for three, and \texttt{-{}-{}-{}-} for four.
    \item \texttt{Z} is a convenient shorthand for the common configuration \texttt{0=0}, a binary curve with two nodes.
    \item Two exceptional graphs have special names, following the notation in \cite{bouw2021reduction}: \texttt{CAVE} and \texttt{BRAID}.
\end{itemize}

\begin{remark}
    In the naming convention used in works such as \cite{bommelcayley} and \cite{van2025reduction}, the hyperelliptic property is incorporated directly into the notation for the reduction type. For instance, they use \texttt{(2n)\textsubscript{H}} for a hyperelliptic curve of geometric genus $2$ with one node, while \texttt{2n} denotes the corresponding non-hyperelliptic type. In this thesis, we adopt a different convention, treating hyperellipticity as a separate property of a combinatorial type rather than as part of its name.
\end{remark}

\begin{proposition}[Classification of stable curves of genus $3$] \label{prop:classification}
  There are exactly $42$ combinatorial types of stable curves of genus $3$. They may be classified according to the properties of their core, yielding two main families: those with a 2-inseparable core (27 types) and those with a 2-separable core (15 types):
  \needspace{10cm}
    \begin{description}
        \item Core is 2-inseparable (27 cases)
        \begin{description}
            \item Core is irreducible (20 cases)
            \par
            \vspace{0.5em}
            \centerline{%
            \begin{tabular}{Sc|Sc|Sc|Sc}
            \small0 1-tails (4) & \small1 1-tail (6) & \small2 1-tails (6) & \small3 1-tails (4) \\
            \hline
            \begin{tabular}{@{}cc@{}}
                \scalebox{0.9}{\HobbyCurve{3}} & \scalebox{0.9}{\HobbyCurve{2n}} \\[-0.40em]
                \footnotesize\texttt{3} & \footnotesize\texttt{2n} \\[-0.125em]
                \scalebox{0.9}{\HobbyCurve{1nn}} & \scalebox{0.9}{\HobbyCurve{0nnn}} \\[-0.40em]
                \footnotesize\texttt{1nn} & \footnotesize\texttt{0nnn}
            \end{tabular}
            &
            \begin{tabular}{@{}cc@{}}
                \scalebox{0.9}{\HobbyCurve{2e}} & \scalebox{0.9}{\HobbyCurve{2m}} \\[-0.40em]
                \footnotesize\texttt{2e} & \footnotesize\texttt{2m} \\[-0.125em]
                \scalebox{0.9}{\HobbyCurve{1ne}} & \scalebox{0.9}{\HobbyCurve{1nm}} \\[-0.40em]
                \footnotesize\texttt{1ne} & \footnotesize\texttt{1nm} \\[-0.125em]
                \scalebox{0.9}{\HobbyCurve{0nne}} & \scalebox{0.9}{\HobbyCurve{0nnm}} \\[-0.40em]
                \footnotesize\texttt{0nne} & \footnotesize\texttt{0nnm}
            \end{tabular}
            &
            \begin{tabular}{@{}cc@{}}
                \scalebox{0.9}{\HobbyCurve{1ee}} & \scalebox{0.9}{\HobbyCurve{1me}} \\[-0.40em]
                \footnotesize\texttt{1ee} & \footnotesize\texttt{1me} \\[-0.125em]
                \scalebox{0.9}{\HobbyCurve{1mm}} & \scalebox{0.9}{\HobbyCurve{0nee}} \\[-0.40em]
                \footnotesize\texttt{1mm} & \footnotesize\texttt{0nee} \\[-0.125em]
                \scalebox{0.9}{\HobbyCurve{0nme}} & \scalebox{0.9}{\HobbyCurve{0nmm}} \\[-0.40em]
                \footnotesize\texttt{0nme} & \footnotesize\texttt{0nmm}
            \end{tabular}
            &
            \begin{tabular}{@{}cc@{}}
                \scalebox{0.9}{\HobbyCurve{0eee}} & \scalebox{0.9}{\HobbyCurve{0mee}} \\[-0.40em]
                \footnotesize\texttt{0eee} & \footnotesize\texttt{0mee} \\[-0.125em]
                \scalebox{0.9}{\HobbyCurve{0mme}} & \scalebox{0.9}{\HobbyCurve{0mmm}} \\[-0.40em]
                \footnotesize\texttt{0mme} & \footnotesize\texttt{0mmm}
            \end{tabular}
            \end{tabular}
            }
    
            \item Core is reducible (7 cases)
                \par
                \centerline{%
                \renewcommand{\arraystretch}{1.2}
                \begin{tabular}{Sc|Sc}
                \small0 1-tails (5) & \small1 1-tail (2) \\
                \hline
                \begin{tabular}{@{}c@{}}
                    \setlength{\tabcolsep}{1pt}
                    \begin{tabular}{ccccc}
                    \scalebox{0.9}{\HobbyCurve{1---0}} & \scalebox{0.9}{\HobbyCurve{0---0n}} & \scalebox{0.9}{\HobbyCurve{0----0}} & \scalebox{0.9}{\HobbyCurve{CAVE}} & \scalebox{0.9}{\HobbyCurve{BRAID}} \\[-0.60em]
                    \footnotesize\texttt{1-{}-{}-0} & \footnotesize\texttt{0-{}-{}-0n} & \footnotesize\texttt{0-{}-{}-{}-0} & \footnotesize\texttt{CAVE} & \footnotesize\texttt{BRAID}
                    \end{tabular}
                \end{tabular}
                &
                \begin{tabular}{@{}c@{}}
                    \begin{tabular}{cc}
                    \scalebox{0.9}{\HobbyCurve{0---0e}} & \scalebox{0.9}{\HobbyCurve{0---0m}} \\[-0.60em]
                    \footnotesize\texttt{0-{}-{}-0e} & \footnotesize\texttt{0-{}-{}-0m}
                    \end{tabular}
                \end{tabular}
                \end{tabular}
                }
        \end{description}
    
        \item Core is 2-separable (15 cases)
        \begin{description}
            \item Core has 2 components (10 cases)
                \par
                \vspace{0.5em}
                \centerline{%
                \renewcommand{\arraystretch}{1.2}
                \begin{tabular}{Sc|Sc|Sc}
                \small0 1-tails (3) & \small1 1-tail (4) & \small2 1-tails (3) \\
                \hline
                \begin{tabular}{@{}c@{}}
                    \setlength{\tabcolsep}{3pt}
                    \begin{tabular}{ccc}
                    \scalebox{0.9}{\HobbyCurve{1=1}} & \scalebox{0.9}{\HobbyCurve{1=0n}} & \scalebox{0.9}{\HobbyCurve{0n=0n}} \\[-0.50em]
                    \footnotesize\texttt{1=1} & \footnotesize\texttt{1=0n} & \footnotesize\texttt{0n=0n}
                    \end{tabular}
                \end{tabular}
                &
                \begin{tabular}{@{}c@{}}
                    \setlength{\tabcolsep}{2pt}
                    \begin{tabular}{cccc}
                    \scalebox{0.9}{\HobbyCurve{1=0e}} & \scalebox{0.9}{\HobbyCurve{1=0m}} & \scalebox{0.9}{\HobbyCurve{0n=0e}} & \scalebox{0.9}{\HobbyCurve{0n=0m}} \\[-0.50em]
                    \footnotesize\texttt{1=0e} & \footnotesize\texttt{1=0m} & \footnotesize\texttt{0n=0e} & \footnotesize\texttt{0n=0m}
                    \end{tabular}
                \end{tabular}
                &
                \begin{tabular}{@{}c@{}}
                    \setlength{\tabcolsep}{3pt}
                    \begin{tabular}{ccc}
                    \scalebox{0.9}{\HobbyCurve{0e=0e}} & \scalebox{0.9}{\HobbyCurve{0m=0e}} & \scalebox{0.9}{\HobbyCurve{0m=0m}} \\[-0.50em]
                    \footnotesize\texttt{0e=0e} & \footnotesize\texttt{0m=0e} & \footnotesize\texttt{0m=0m}
                    \end{tabular}
                \end{tabular}
                \end{tabular}
                }
    
            \item Core has at least 3 components (5 cases)
                \par
                \vspace{0.5em}
                \centerline{%
                \renewcommand{\arraystretch}{1.2}
                \begin{tabular}{Sc|Sc}
                \small0 1-tails (3) & \small1 1-tail (2) \\
                \hline
                \begin{tabular}{@{}c@{}}
                    \setlength{\tabcolsep}{3pt}
                    \begin{tabular}{ccc}
                    \scalebox{0.9}{\HobbyCurve{Z=1}} & \scalebox{0.9}{\HobbyCurve{Z=0n}} & \scalebox{0.9}{\HobbyCurve{Z=Z}} \\[-0.50em]
                    \footnotesize\texttt{Z=1} & \footnotesize\texttt{Z=0n} & \footnotesize\texttt{Z=Z}
                    \end{tabular}
                \end{tabular}
                &
                \begin{tabular}{@{}c@{}}
                    \begin{tabular}{cc}
                    \scalebox{0.9}{\HobbyCurve{Z=0e}} & \scalebox{0.9}{\HobbyCurve{Z=0m}} \\[-0.50em]
                    \footnotesize\texttt{Z=0e} & \footnotesize\texttt{Z=0m}
                    \end{tabular}
                \end{tabular}
                \end{tabular}
                }
        \end{description}
    \end{description}
The geometric genus of each component ($0$, $1$, $2$, or $3$) is indicated by its line thickness, from thinnest to thickest.
\end{proposition}

The diagrams in the classification above are adapted from \cite[Figure 2.2]{bommelcayley}, with slight modifications made using the Hobby-Editor tool \cite{hobbyeditor}. The enumeration of all stable types of a genus $g$ curve is a purely combinatorial problem and is well known (see for example \cite[Theorem 2.2.12]{chan2013tropical}). For completeness, we provide a proof demonstrating how all 42 types for genus $3$ can be systematically derived using the ``core and tails'' structure.

\begin{proof}
    Let $C$ be a genus $3$ stable curve with $r$ $1$-tails and its core $C_c$. By definition, the core is an inseparable curve of genus $g(C_c) = 3-r$. Note that by Proposition \ref{prop:stable_via_canonical}, $C$ can have at most $2 g(C) - 2 = 4$ irreducible components, and hence the core $C_c$ can have at most $4-r$.

    \begin{case}
        The Core is 2-Inseparable
        
        We can classify the combinatorial types of $C_c$ by analyzing its component graph, defined as $\Gamma_c := \Graph^S(C_c, \emptyset)$, where $S$ is the set of nodes connecting distinct irreducible components. Its vertices correspond to the $n \le 4-r$ irreducible components of $C_c$, with weights given by their arithmetic genera $g_1, \dots, g_n$. Its edges correspond to the $i:=|S|$ nodes in $S$ that connect distinct components. By construction, this graph has no legs, as there are no marked points, and no loops, as $S$ only contains nodes connecting distinct components. The genus of the core is given by the genus of its graph:
        \[
            g(C_c) = g(\Gamma_c) = \sum_{j=1}^n g_j + 1 - \chi(\Gamma_c) =
            \sum_{j=1}^n g_j + 1 + i - n.
        \]
        Substituting $g(C_c) = 3-r$, we get our main equation for this case: 
        \[ \sum_{j=1}^n g_j = 2+n-i-r. \]

        \begin{subcase}
            The Core is Irreducible ($n=1$)

            The core is a single irreducible component, so $i=0$. The equation gives its arithmetic genus as $g(C_c) = 3-r$. It can have a geometric genus $g_c$ from $0$ to $3-r$, with a corresponding number of self-intersection nodes $d$ such that $g_c+d=3-r$. A systematic enumeration of these possibilities for each $r \in \{0,1,2,3\}$, combined with the possible types of attached $1$-tails (elliptic or pig tail), yields the \emph{20 types} with an irreducible core.
        \end{subcase}

        \begin{subcase}
            The Core is Reducible ($n \ge 2$)

            Since $C_c$ is 2-inseparable, its component graph $\Gamma_c$ must be 3-edge-connected. This implies every vertex has degree at least $3$, so $2i \ge 3n$.
            
            \begin{description}
                \item[$n=2$:] Requires $i \ge 3$. The genus formula is $g_1+g_2 = 4-i-r$. Since $g_j \ge 0$, we must have $i \le 4$. This leaves two possibilities:
                \begin{description}
                    \item[$i=3$:] Then $g_1+g_2=1-r$. This requires $r=0$ (for $\{g_1,g_2\}=\{1, 0\}$, giving types \texttt{1-{}-{}-0} and \texttt{0n-{}-{}-0}) or $r=1$ (for $g_1=g_2=0$, giving types \texttt{0-{}-{}-0e} and \texttt{0-{}-{}-0m}).
                    \item[$i=4$:] Then $g_1+g_2=-r$. This requires $r=0, g_1=g_2=0$. This is the type \texttt{0-{}-{}-{}-0}.
                \end{description}
                \item[$n=3$:] Requires $2i \ge 9$, so $i \ge 5$. The formula is $g_1+g_2+g_3=5-i-r$, which forces $i=5$, $r=0$, and $g_1=g_2=g_3=0$. Hence $\Gamma_c$ coincides with the dual graph of $C$ and is uniquely determined as the unique 3-vertex graph with no loops and $5$ edges where each vertex has a degree of at least $3$, which is the dual graph of \texttt{CAVE}.
                \item[$n=4$:] Requires $i \ge 6$. The formula is $\sum g_j = 6-i-r$, forcing $i=6$, $r=0$, and all $g_j=0$. Here again, the component graph $\Gamma_c$ coincides with the dual graph of $C$ and is uniquely determined as the complete graph $K_4$ (the unique 4-vertex loop-less graph with $6$ edges where each vertex has degree $3$), which is the dual graph of \texttt{BRAID}.
            \end{description}
            This analysis yields the \emph{7 types} with a reducible, 2-inseparable core.
        \end{subcase}
    \end{case}

    \begin{case}
        The Core is 2-Separable

        If the core $C_c$ is 2-separable, it possesses a separating pair $\{p,q\}$, which is also a separating pair for the entire curve $C$. The decomposition of $C$ induced by this pair yields two pointed curves, $(D_1;p,q)$ and $(D_2;p,q)$, which must themselves be stable (Lemma \ref{lem:stable_decomposition}). The genus formula for a decomposition at a separating pair is $g(C) = g(D_1) + g(D_2) + 1$. Since $g(C)=3$, it follows that $g(D_1) + g(D_2) = 2$. Stability for a $2$-pointed curve $(D; p, q)$ requires $2g(D)-2+2 > 0$, hence $g(D) \ge 1$. Therefore, the only possibility is $g(D_1) = 1$ and $g(D_2) = 1$.
        
        The classification of these 15 types thus reduces to enumerating all combinations of two stable, $2$-pointed curves of genus $1$. These building blocks are classified in the following claim.

        \begin{claimstar}[Classification of stable $2$-pointed curves of genus 1]
            Let $(D;p,q)$ be a stable $2$-pointed curve of arithmetic genus $g(D)=1$. Then its combinatorial type must be one of the following 5 possibilities, categorized by structure:
            \begin{enumerate}[label=\DefListLabel]
                \item \label{class_2-point_1} \textbf{Irreducible} (2 types): $D$ is smooth (type \texttt{1}) or nodal (type \texttt{0n}).
                \item \label{class_2-point_2} \textbf{Reducible and Separable} (2 types): $D$ consists of a $1$-tail attached to a rational component, with both marked points $p,q$ on the rational component (type \texttt{0e} or \texttt{0m}).
                \item \label{class_2-point_3} \textbf{Reducible and Inseparable} (1 type): $D$ is a binary curve with two nodes (type \texttt{Z}), and $p,q$ lie on distinct components.
            \end{enumerate}
        \end{claimstar}

        \begin{proof}[of Claim]
            By Proposition \ref{prop:stable_via_canonical}, the curve $D$ has at most $2g(D) - 2 + |\{p,q\}| = 2$ irreducible components. The irreducible case corresponds to the two types listed. If $D$ is reducible, it must consist of two components $D_1 \cup D_2$ with arithmetic genera $g_1, g_2$ connected by $i$ nodes. The genus formula gives $g_1 + g_2 = 2-i$, so $i \in \{1,2\}$.
            \begin{description}
                \item[$i=1$ (Separable):] We have $g_1+g_2=1$, so we may assume $g_1=0$ and $g_2=1$. By the stability condition, the genus-0 component $D_1$ requires at least three special points. As it has only one connecting node, it must contain both marked points $p$ and $q$. The genus-1 component $D_2$, which has only one special point (the connecting node), is therefore a $1$-tail. This structure yields the types \texttt{0e} and \texttt{0m}.
                \item[$i=2$ (Inseparable):] The formula $g_1 + g_2 = 0$ forces $g_1=g_2=0$. Thus, $D$ is of type \texttt{Z}. For stability, each rational component has two connecting nodes and requires a third special point. Therefore, $p$ and $q$ must lie on distinct components.
            \end{description}
        \end{proof}

        The 15 combinatorial types with a 2-separable core are formed by combining two of the 5 component types from the claim (with replacement), yielding $\binom{5+2-1}{2} = 15$ pairings. These correspond to the types listed in the classification table. The $1$-tails of the resulting curve $C$ are inherited from any components of type \ref{class_2-point_2} used in the pairing. The core of $C$ has more than $2$ irreducible components if and only if at least one chosen component is of type \ref{class_2-point_3} (\texttt{Z}); this accounts for 5 of the 15 cases. The remaining 10 pairings result in a core with exactly $2$ components. This completes the classification for types with 2-separable cores. \hfill\proofSymbol
    \end{case}
\end{proof}

\begin{remark} \label{rem:g3_semicompact}
    The preceding classification shows that a stable curve of genus $3$ has at most one separating pair (none if its core is 2-inseparable, and exactly one if its core is 2-separable). Consequently, every stable genus $3$ curve is of semicompact type.
\end{remark}

We now have the tools to determine which of the 42 types of stable genus $3$ curves are hyperelliptic. Our strategy is to apply the decomposition lemmas, which reduce the problem to analyzing the hyperellipticity of the resulting lower-genus components.

We will first address the case of curves with a 2-separable core. The following lemma provides the crucial result for this analysis, establishing that the specific building blocks that arise in this case are always hyperelliptic.

\begin{lemma} \label{lem:always_hyperelliptic}
    The following types of stable curves are always hyperelliptic:
    \begin{enumerate}[label=\StatementListLabel]
        \item \label{lem:always_hyperelliptic_i} Any stable $1$-pointed curve of arithmetic genus $1$.
        \item \label{lem:always_hyperelliptic_ii} Any stable (unpointed) curve of arithmetic genus $2$.
    \end{enumerate}
\end{lemma}

\begin{proof}
This proof relies on techniques from Section~\ref{sec:quasi-hyp}, applying Corollary~\ref{cor:canonical_embedding_general} and adapting the argument from the proof of Corollary~\ref{cor:all_the_same}.
\begin{prooflist}
\item Let $(C,p)$ be a stable $1$-pointed curve of arithmetic genus $g=1$. By Remark~\ref{rem:elliptic_pig_tails}, $C$ is irreducible. By the Riemann-Roch theorem, the linear system $|\mathcal{O}_C(2p)|$ is base-point-free and induces a finite morphism $f\colon C \to \PP^1_k$ of degree $2$. Following the argument in Corollary~\ref{cor:all_the_same}, this morphism is separable and is therefore induced by a unique involution $\sigma$. (In the smooth case, where $(C,p)$ is an elliptic curve with $p$ as the origin, this is the inversion map $x \mapsto -x$). By construction, the morphism $f$ is ramified at $p$, which must therefore be a fixed point of the involution $\sigma$. Therefore, $\sigma$ is a hyperelliptic involution for the pointed curve $(C,p)$.

\item A stable curve $C$ of arithmetic genus $g=2$ is either separable or 2-inseparable, as it cannot have a separating pair (since the genera of the resulting components would have to sum to 1, while the stability condition for $2$-pointed curves requires the genus of each to be at least 1).
\begin{description}
    \item[Separable case:] If $C$ is separable, it possesses a unique separating node $p$. The decomposition at $p$ yields two components. Since stability forbids a genus $0$ component with only one attachment point, both components must be stable $1$-pointed curves of genus $1$. By part \ref{lem:always_hyperelliptic_i}, these components are hyperelliptic. The decomposition lemma (Lemma~\ref{lem:decomp_sep_node}) then implies that $C$ is hyperelliptic.

    \item[2-Inseparable case:] If $C$ is 2-inseparable, we consider the canonical map $\Phi_C\colon C \to \PP_k^{g-1} = \PP_k^1$. Since $C$ is 2-inseparable (and thus inseparable), the canonical system is base-point-free by Theorem~\ref{thm:catanese_base_points}, so $\Phi_C$ is a well-defined morphism.
    
    If $\Phi_C$ were a closed immersion, $C$ would be isomorphic to its image $\PP^1_k$. This is impossible since $g(C)=2$ and $g(\PP^1_k)=0$. Therefore, $\Phi_C$ is not a closed immersion. By Corollary~\ref{cor:canonical_embedding_general}, it follows that $C$ must be hyperelliptic.
\end{description}
\end{prooflist}
\end{proof}

\begin{proposition} \label{prop:2-sep-hyp}
    A stable curve $C$ of genus $3$ with a 2-separable core is always hyperelliptic.
\end{proposition}

\begin{proof}
As a stable curve of genus $3$, $C$ is of semicompact type (Remark~\ref{rem:g3_semicompact}). The decomposition lemmas (Lemma~\ref{lem:decomp_sep_node} and Lemma~\ref{lem:decomp_sep_pair}) can therefore be applied, and they establish that $C$ is hyperelliptic if and only if each of the 2-inseparable curves resulting from its decomposition is hyperelliptic. As the proof of Proposition~\ref{prop:classification} for the 2-separable core case shows, these 2-inseparable components must be either stable (unpointed) curves of genus $2$ or stable $1$-pointed curves of genus $1$. Since both of these types are always hyperelliptic by Lemma~\ref{lem:always_hyperelliptic}, the condition is satisfied, and $C$ is therefore hyperelliptic.
\end{proof}

\begin{example}
    To illustrate the decomposition process, consider a curve $C$ of type \texttt{Z=0e}. As illustrated in Figure~\ref{fig:decomposition}, the process involves two steps.

    First, decomposing $C$ at its unique separating node yields two pointed curves: an elliptic tail (a stable $1$-pointed curve of geometric genus $1$) and the core, a pointed curve of type \texttt{Z=0}. Second, this core is decomposed at its separating pair. Following the procedure in Lemma~\ref{lem:decomp_sep_pair} (which involves forming contracted components), this second step yields a genus $2$ binary curve (type \texttt{0-{}-{}-0}) and a pig tail (a stable $1$-pointed nodal curve of genus $1$, type \texttt{0n}).

    Thus, the 2-inseparable components of the original curve $C$ are a genus $2$ binary curve, an elliptic tail, and a pig tail. As shown in Lemma~\ref{lem:always_hyperelliptic}, all of these components are hyperelliptic, which implies that $C$ itself is hyperelliptic.
    
    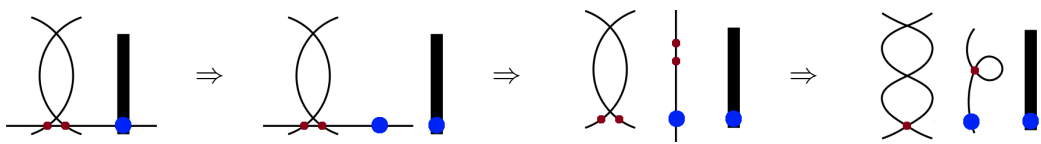
\begin{figure}[h!]
        \centering
        \adjustbox{valign=c}{%
            \scalebox{1.9}{%
                \begin{tikzpicture}[scale=0.12]
                    \useasboundingbox (-2.2,-3) rectangle (7.2,7.5);
                    \clip (-2.2,-3) rectangle (7.2,7.5);
                    
                    \draw[name path=P1] (-0.5, -0.5) to[curve through={(2, 3)}] (-0.5, 6.5);
                    \draw[name path=P2] (2.5, -0.5) to[curve through={(0, 3)}] (2.5, 6.5);
                    \draw[name path=P3] (-2, 0) to (7, 0);
                    \draw[line width = 2, name path=P4] (5, -0.5) to (5, 5.5);
                    
                    \fill[myred, name intersections={of=P1 and P3, by=Spair1}] (Spair1) circle (0.25);
                    \fill[myred, name intersections={of=P2 and P3, by=Spair2}] (Spair2) circle (0.25);
                    \fill[myblue, name intersections={of=P3 and P4, by=Snode}] (Snode) circle (0.5);
                \end{tikzpicture}%
            }%
        }
        \quad $\leadsto$ \quad
        \adjustbox{valign=c}{%
            \scalebox{1.9}{%
                \begin{tikzpicture}[scale=0.12]
                    \useasboundingbox (-2.2,-3) rectangle (9.2,7.5);
                    \clip (-2.2,-3) rectangle (9.2,7.5);
                    
                    \path[name path=P3dummy] (-2, 0) to (7, 0);
                    \path[name path=P4dummy] (5, -0.5) to (5, 5.5);
                    \path[name intersections={of=P3dummy and P4dummy, by=Snode}];
                    
                    \draw[name path=P1] (-0.5, -0.5) to[curve through={(2, 3)}] (-0.5, 6.5);
                    \draw[name path=P2] (2.5, -0.5) to[curve through={(0, 3)}] (2.5, 6.5);
                    \draw (-2, 0) to (7, 0); 
                    
                    \begin{scope}[shift={(3.5,0)}]
                        \draw[line width = 2] (5, -0.5) to (5, 5.5);
                    \end{scope}
                    
                    \fill[myred, name intersections={of=P1 and P3dummy, by=Spair1}] (Spair1) circle (0.25);
                    \fill[myred, name intersections={of=P2 and P3dummy, by=Spair2}] (Spair2) circle (0.25);
                    \fill[myblue] (Snode) circle (0.5);
                    \fill[myblue] (Snode) ++(3.5,0) circle (0.5);
                \end{tikzpicture}%
            }%
        }
        \quad $\leadsto$ \quad
        \adjustbox{valign=c}{%
            \scalebox{1.9}{%
                \begin{tikzpicture}[scale=0.12]
                    \useasboundingbox (-0.7,-3) rectangle (9.2,7.5);
                    \clip (-0.7,-3) rectangle (9.2,7.5);
                    
                    \path[name path=P3dummy] (-2, 0) to (7, 0);
                    \path[name path=P4dummy] (5, -0.5) to (5, 5.5);
                    \path[name intersections={of=P3dummy and P4dummy, by=Snode}];
                    
                    \draw[name path=P1] (-0.5, -0.5) to[curve through={(2, 3)}] (-0.5, 6.5);
                    \draw[name path=P2] (2.5, -0.5) to[curve through={(0, 3)}] (2.5, 6.5);
                    \path[name intersections={of=P1 and P3dummy, by=Spair1}];
                    \path[name intersections={of=P2 and P3dummy, by=Spair2}];
                    
                    \fill[myred] (Spair1) circle (0.25);
                    \fill[myred] (Spair2) circle (0.25);
                    
                    \draw (5, 7) to (5, -2);
                    
                    \path let \p1=(Spair1), \p2=(Snode) in coordinate (Spair1_rot) at (\x2, \x2 - \x1);
                    \path let \p1=(Spair2), \p2=(Snode) in coordinate (Spair2_rot) at (\x2, \x2 - \x1);
                    \fill[myred] (Spair1_rot) circle (0.25);
                    \fill[myred] (Spair2_rot) circle (0.25);
                    \fill[myblue] (Snode) circle (0.5);
                    
                    \begin{scope}[shift={(3.5,0)}]
                        \draw[line width = 2] (5, -0.5) to (5, 5.5);
                    \end{scope}
                    \fill[myblue] (Snode) ++(3.5,0) circle (0.5);
                \end{tikzpicture}%
            }%
        }
        \quad $\leadsto$ \quad
        \adjustbox{valign=c}{%
            \scalebox{1.9}{%
                \begin{tikzpicture}[scale=0.12]
                    \useasboundingbox (-0.7,-3) rectangle (9.2,7.5);
                    \clip (-0.7,-3) rectangle (9.2,7.5);
                    
                    \coordinate (Snode) at (5,0); 
                    
                    \draw (2.5, -0.5) to[curve through={(1, 0) (0, 1) (1, 2) (2, 3.5)}] (-0.5, 6.5);
                    \draw (-0.5, -0.5) to[curve through={(1, 0) (2, 1) (1, 2) (0, 3.5)}] (2.5, 6.5);
                    
                    \fill[myred] (1, 0) circle (0.25);
                    
                    \draw (5, -2) to[out angle = 90, curve through={(5.1, 4) (5.7, 4.5) (6.7, 4) (5.7, 3.5) (5.1, 4)}] (5, 7);
                    \fill[myred] (5.1, 4) circle (0.25);
                    
                    \fill[myblue] (4.8, 0) circle (0.5);
                    
                    \begin{scope}[shift={(3.5,0)}]
                        \draw[line width = 2] (5, -0.5) to (5, 5.5);
                    \end{scope}
                    
                    \fill[myblue] (5, 0) ++(3.5,0) circle (0.5);
                \end{tikzpicture}%
            }%
        }
        \caption{The decomposition process of a curve of type \texttt{Z=0e}. Separating nodes are marked in \textcolor{myblue}{blue} and separating pairs are marked in \textcolor{myred}{red}.}
        \label{fig:decomposition}
    \end{figure}
\end{example}

We now address the case of stable genus $3$ curves with a 2-inseparable core. For one case of the following proposition, we will use a direct generalization of Theorem~\ref{thm:canonical_embedding}, which characterizes hyperellipticity in terms of the canonical map for (possibly singular) 2-inseparable curves. This will be explained in more detail in the subsequent Section~\ref{sec:quasi-hyp}.

\begin{proposition} \label{prop:2-insep_geom_condition_hyperelliptic}
    Let $C$ be a stable curve of genus $3$ with a 2-inseparable core $C_c$ and $r$ $1$-tails. Let $x_1, \dots, x_r$ be the attachment points on the core. The hyperelliptic property of $C$ is determined as follows:
    \begin{description}
        \item[Core is irreducible:] Then $0 \leq r \leq 3$.
        \begin{description}
            \item[$r=0$:] $C$ is hyperelliptic if and only if the canonical map $\Phi_C\colon C \to \PP^2_k$ is not a closed immersion.
            \item[$r=1$:] $C$ is hyperelliptic if and only if the attachment point $x_1$ is a Weierstrass point on the genus $2$ curve $C_c$, i.e., $h^0(C_c, \OO_{C_c}(2x_1)) > 1$. Equivalently, $\omega_{C_c/k} \simeq \mathcal{O}_{C_c}(2x_1)$.
            \item[$r=2$:] $C$ is hyperelliptic if and only if there is an isomorphism of invertible sheaves $\mathcal{O}_{C_c}(2x_1) \simeq \mathcal{O}_{C_c}(2x_2)$ on the genus $1$ curve $C_c$.
            \item[$r=3$:] $C$ is hyperelliptic if and only if $\Char k = 2$.
        \end{description}
        \item[Core is reducible:] Then $0 \leq r \leq 1$.
        \begin{description}
            \item[$r=0$:] $C$ is hyperelliptic if and only if it is a binary curve of type \texttt{0-{}-{}-{}-0} and the set of four node preimages on one of its normalized components is projectively equivalent to the corresponding set of preimages on the other (i.e., they have the same cross-ratio).
            \item[$r=1$:] $C$ is never hyperelliptic.
        \end{description}
    \end{description}
\end{proposition}

\begin{proof}
    Let $P_c := \{x_1, \dots, x_r\}$ be the set of attachment points on the 2-inseparable core $C_c$. By Lemma~\ref{lem:core_tail_lemma}, these are precisely the separating nodes of $C$. Decomposing $C$ at these nodes yields the pointed core $(C_c, P_c)$ and $r$ $1$-tails. A $1$-tail, when viewed as a stable $1$-pointed curve of genus $1$, is always hyperelliptic (Lemma~\ref{lem:always_hyperelliptic}~\ref{lem:always_hyperelliptic_i}). Therefore, repeated application of the decomposition lemma (Lemma~\ref{lem:decomp_sep_node}) shows that $C$ is hyperelliptic if and only if its pointed core $(C_c, P_c)$ is hyperelliptic.

    The problem is thus reduced to analyzing the hyperellipticity of the pointed core $(C_c, P_c)$, which is a 2-inseparable stable pointed curve of arithmetic genus $g(C_c) = 3-r$. The combinatorial types of these pointed cores are implicitly given by the classification in Proposition~\ref{prop:classification}; we simply no longer distinguish between cases that differ only by the type of $1$-tail attached (elliptic versus pig tail).

    \begin{description}[leftmargin=0em]
        \item[Core is irreducible:] In this case, the number of $1$-tails can be $r=0, 1, 2,$ or $3$.

        \begin{description}
            \item[$r=0$:] Here $C = C_c$ is an irreducible (hence 2-inseparable) stable curve of genus $3$. By Corollary~\ref{cor:canonical_embedding_general}, $C$ is hyperelliptic if and only if the canonical map $\Phi_C\colon C \to \PP^2_k$ is not a closed immersion.
            
            \item[$r=1$:] The core $C_c$ has genus $2$. By Lemma~\ref{lem:always_hyperelliptic}~\ref{lem:always_hyperelliptic_ii}, $C_c$ is always hyperelliptic. Let $\sigma$ be its unique hyperelliptic involution. The pointed curve $(C_c, \{x_1\})$ is hyperelliptic if and only if $\sigma$ fixes the attachment point $x_1$.

            We characterize this condition using the canonical sheaf $\omega_{C_c/k}$. The canonical map $\Phi_{C_c}\colon\allowbreak C_c \to \PP^1_k$ is the quotient morphism by $\sigma$ (cf. Corollary~\ref{cor:all_the_same}). The fibers of $\Phi_{C_c}$ are the effective divisors $D$ such that $\mathcal{O}_{C_c}(D) \simeq \omega_{C_c/k}$. Geometrically, the fiber containing a point $x$ is the divisor $x + \sigma(x)$.
            The condition $\sigma(x_1) = x_1$ holds if and only if $\Phi_{C_c}$ is ramified at $x_1$, meaning the fiber is the divisor $2x_1$. Thus, $\sigma(x_1) = x_1$ is equivalent to the isomorphism $\mathcal{O}_{C_c}(2x_1) \simeq \omega_{C_c/k}$.

            We now relate this to the definition of a Weierstrass point, i.e., $h^0(C_c, \mathcal{O}_{C_c}(2x_1)) > 1$. Since $C_c$ is a stable curve, it is Gorenstein, and the Riemann-Roch theorem applies:
            \[ h^0(\mathcal{O}_{C_c}(2x_1)) - h^0(\omega_{C_c/k}(-2x_1)) = \deg(2x_1) - g(C_c) + 1 = 2 - 2 + 1 = 1. \]
            We have $h^0(\mathcal{O}_{C_c}(2x_1)) > 1$ if and only if $h^0(\omega_{C_c/k}(-2x_1)) \geq 1$. Since $\omega_{C_c/k}(-2x_1)$ has degree $0$, the existence of a non-zero section is equivalent to $\omega_{C_c/k}(-2x_1) \simeq \mathcal{O}_{C_c}$, which in turn is equivalent to $\omega_{C_c/k} \simeq \mathcal{O}_{C_c}(2x_1)$.

            \item[$r=2$:] The core $C_c$ has genus $1$. By Lemma~\ref{lem:always_hyperelliptic}~\ref{lem:always_hyperelliptic_i}, the $1$-pointed curve $(C_c, x_1)$ is hyperelliptic. Let $\sigma$ be its unique hyperelliptic involution. The $2$-pointed curve $(C_c, \{x_1, x_2\})$ is hyperelliptic if and only if this involution also fixes $x_2$.

            As argued in the proof of Lemma \ref{lem:always_hyperelliptic}~\ref{lem:always_hyperelliptic_i}, the involution $\sigma$ is induced by the finite degree $2$ map $f\colon C_c \to \PP^1_k$ defined by the invertible sheaf $\mathcal{L} = \mathcal{O}_{C_c}(2x_1)$, and $f$ is the quotient map by $\sigma$. The fibers of $f$ are the effective divisors $D$ such that $\mathcal{O}_{C_c}(D) \simeq \mathcal{L}$. Geometrically, the fiber containing $x$ is the divisor $x + \sigma(x)$. Thus, for all $x \in C_c$, we have an isomorphism $\mathcal{O}_{C_c}(x + \sigma(x)) \simeq \mathcal{O}_{C_c}(2x_1)$.
            The condition $\sigma(x_2)=x_2$ means $f$ is ramified at $x_2$, so the fiber is $2x_2$. This is equivalent to the condition $\mathcal{O}_{C_c}(2x_2) \simeq \mathcal{O}_{C_c}(2x_1)$.

            \item[$r=3$:] The core $C_c$ is an irreducible curve of genus $0$, hence $C_c \simeq \PP^1_k$. The resulting stable pointed curve $(C_c, \{x_1, x_2, x_3\})$ is, by Proposition~\ref{prop:hyperelliptic_genus2}, hyperelliptic if and only if $\Char k = 2$.
        \end{description}

        \item[Core is reducible:] In this case, the core $C_c$ is reducible and 2-inseparable, which by the classification in Proposition~\ref{prop:classification} occurs only when $r=0$ or $r=1$. To analyze these cases, we apply Lemma~\ref{lem:unique_2_inse}, which states that a pointed, stable, 2-inseparable, hyperelliptic curve must either be irreducible or an unpointed binary curve (the union of two smooth rational components connected by at least three nodes).

        \begin{description}
            \item[$r=0$:] The core $C = C_c$ is a 2-inseparable stable curve of genus $3$. If it is hyperelliptic, it must be a binary curve by Lemma~\ref{lem:unique_2_inse} with $g(C) + 1 = 4$ nodes; that is, it must be of type \texttt{0-{}-{}-{}-0}. The hyperelliptic involution on such a curve necessarily swaps the two rational components. The existence of such an involution is equivalent to an isomorphism between the two normalized components that maps the set of four node preimages on the first component to the corresponding set on the second. This holds if and only if the two sets of four points are projectively equivalent.

            \item[$r=1$:] The set of marked points $P=\{x_1\}$ is non-empty. By Lemma~\ref{lem:unique_2_inse}, since $C_c$ is reducible, the pointed curve $(C_c, P)$ cannot be hyperelliptic.
        \end{description}
    \end{description}
\end{proof}

\begin{remark} \label{rem:char2_dependency}
    The four combinatorial types with a genus-$0$ core and three attached tails ($r=3$) are the only instances among all 42 stable types of genus $3$ where the hyperelliptic property depends on the characteristic-dependent condition in Definition~\ref{def:hyperelliptic}. This follows from Proposition~\ref{prop:hyperelliptic_genus2}, which states that this condition is relevant only if a curve has a rational component that meets the rest of the curve only at separating nodes. For a stable genus $3$ curve, the separating nodes are precisely the attachment points of $1$-tails (see Lemma~\ref{lem:core_tail_lemma}), and stability requires at least three such points on any rational component. These conditions force the curve to be of the type with three $1$-tails attached to a genus-$0$ core.
\end{remark}

We can now summarize the classification of hyperelliptic stable curves of genus $3$.

\begin{proposition}[Hyperelliptic Classification for Genus 3] \label{prop:classification_hyp_3}
    Let $C$ be a stable curve of genus $3$ and let $C_c$ be its core. The hyperelliptic property of $C$ can be characterized based on its combinatorial type and, in some cases, the underlying field characteristic or the curve's specific geometry. The 42 combinatorial types are classified as follows:
    \begin{enumerate}[label=\StatementListLabel]
        \item \label{prop:classification_hyp_3_i} \textbf{Combinatorially Hyperelliptic (15 types):}
        If the core $C_c$ is 2-separable, then $C$ is always hyperelliptic, regardless of the characteristic or the geometric moduli.

        \item \label{prop:classification_hyp_3_ii} \textbf{Combinatorially Non-Hyperelliptic (6 types):}
        If the core $C_c$ is 2-inseparable and reducible, but $C$ is not a binary curve (i.e., of type \texttt{0-{}-{}-{}-0}), then $C$ is never hyperelliptic. (This includes cases where a $1$-tail is attached to a reducible core, and cases like \texttt{CAVE} or \texttt{BRAID}).

        \item \label{prop:classification_hyp_3_iii} \textbf{Characteristic Dependence (4 types):}
        If the core $C_c$ is rational ($C_c \simeq \PP^1_k$, corresponding to the case with three $1$-tails), then $C$ is hyperelliptic if and only if $\Char k = 2$.

        \item \label{prop:classification_hyp_3_iv} \textbf{Geometric Dependence (17 types):}
        In the remaining cases, the hyperellipticity depends on the specific geometric moduli of the curve. This includes:
        \begin{itemize}
            \item Curves with an irreducible core and $0 \leq r \leq 2$ tails (16 types).
            \item The binary curve of type \texttt{0-{}-{}-{}-0} (1 type).
        \end{itemize}
        The specific conditions are detailed in Proposition~\ref{prop:2-insep_geom_condition_hyperelliptic} (e.g., the behavior of the canonical map, Weierstrass points, linear equivalence of divisors, or cross-ratios).
    \end{enumerate}
\end{proposition}

\begin{proof}
    The classification follows directly from Proposition~\ref{prop:2-sep-hyp} (for case \ref{prop:classification_hyp_3_i}) and Proposition~\ref{prop:2-insep_geom_condition_hyperelliptic} (for cases \ref{prop:classification_hyp_3_ii}--\ref{prop:classification_hyp_3_iv}).
\end{proof}

\begin{remark}
    Case~\ref{prop:classification_hyp_3_iv} depends on geometric moduli, but in $\operatorname{char} k = 2$, wild ramification prevents these conditions from being fulfilled in some cases. This affects exactly $3$ of these $17$ types. Specifically, curves consisting of an irreducible core with geometric genus $0$ and one node, and exactly two $1$-tails, are never hyperelliptic in characteristic $2$ (see \cite[Example 4.10]{hyperelliptic}).
\end{remark}

\subsection{Quasi-Hyperelliptic Curves} \label{sec:quasi-hyp}

We introduce here a related, older notion of hyperellipticity applicable to more general curves, which we call \emph{quasi-hyperelliptic}, following Rosenlicht \cite{rosenlicht1952equivalence}. The results discussed in this section are primarily from Catanese \cite{catanese1982pluricanonical}.

The definition is motivated by Theorem~\ref{thm:canonical_embedding}. Classically, if $C$ is a smooth projective curve of genus $g \ge 2$, the global sections of the canonical sheaf, $\omega_{C/k} = \Omega_{C/k}^1$, define the canonical map $\Phi_C$ from $C$ into the projective space associated with the dual of $H^0(C, \omega_{C/k})$, which can be identified with $\PP_k^{g-1}$. For reducible curves with singularities, the correct analogue is obtained by replacing the sheaf of 1-forms, $\Omega_{C/k}^1$, with the \emph{dualizing sheaf}, $\omega_{C/k}$ (a detailed description is provided in Chapter~\ref{chap:3}). For the sections of this sheaf to define a map to projective space, one must assume that $\omega_{C/k}$ is an invertible sheaf, which is equivalent to the condition that $C$ is a \emph{Gorenstein curve}.

Let $C$ be a projective, reduced, connected Gorenstein curve over an algebraically closed field $k$. For such curves, there is still a canonical, though generally rational, map $\Phi_C$. For this map to be a morphism defined on all of $C$, there must not be any components along which all sections of $\omega_{C/k}$ vanish. This leads naturally to the following definition.

\begin{definition}
    A projective, reduced, and connected Gorenstein curve $C$ is said to be \emph{canonically positive} if the dualizing sheaf $\omega_{C/k}$ is ample. Equivalently (see \cite[Proposition 7.5.5]{liu}), for each irreducible component $Z$ of $C$, the degree of the restricted dualizing sheaf, $\omega_{C/k}|_Z$, is positive.
\end{definition}

Clearly, the property of being canonically positive is a necessary condition for a pluricanonical map to be an embedding. This condition can also be thought of as a generalization of stability; indeed by Proposition \ref{prop:stable_via_canonical} a semistable curve is stable if and only if it is canonically positive.

Catanese proved several results concerning the pluricanonical maps, i.e., the maps associated with the sheaves $\omega_{C/k}^{\otimes n}$ for $n \geq 1$. Among other results, he showed that for a canonically positive curve $C$, the sheaf $\omega_{C/k}^{\otimes n}$ is very ample for $n \geq 3$. This generalizes the well-known result of Deligne and Mumford concerning the very-ampleness of $\omega_{C/k}^{\otimes n}$ for $n \geq 3$ on stable curves~\cite[Theorem 1.2]{deligne-mumford}. However, our primary interest is the case $n=1$. For this case, Catanese provided a generalization of Theorem~\ref{thm:canonical_embedding}, which holds under suitable conditions on the connectedness of $C$ and relies on a generalized notion of hyperellipticity, which we call quasi-hyperelliptic.

For the rest of this section, $C$ will denote a projective, reduced, connected, and canonically positive Gorenstein curve over $k$. We first generalize the combinatorial notions of separating nodes and separating pairs.

\begin{definition}
    We define the notions of a \emph{separating node}, a \emph{separating pair}, and of a curve being \emph{inseparable} and \emph{2-inseparable} for the curve $C$ analogously to Definition~\ref{def:comb_term}. Moreover, we call $C$ \emph{very strongly connected} if it is 2-inseparable and there exists no decomposition into two curves $C_1, C_2$ such that $C_1 \cap C_2$ is a single point.
\end{definition}

\begin{remark}
    Our terminology for connectedness differs from that of Catanese. He refers to what we call \emph{separating nodes} as ``disconnecting nodes'' and our \emph{inseparable curves} as ``2-connected curves''~\cite[Remark 3.1]{catanese1982pluricanonical}. Furthermore, he calls our \emph{2-inseparable curves} ``strongly connected curves''~\cite[Definition 3.7]{catanese1982pluricanonical}. However, the term \emph{very strongly connected} and its definition are taken directly from his work~\cite[Definition 3.21]{catanese1982pluricanonical}. For the special case of stable curves, the notions of being 2-inseparable and very strongly connected are clearly equivalent.
\end{remark}

The presence of separating nodes is the reason why the canonical map $\Phi_C$ is only a rational map and not necessarily a morphism, as the following theorem clarifies.

\begin{theorem}[{{\cite[Theorem D]{catanese1982pluricanonical}}}] \label{thm:catanese_base_points}
    The base locus of $|\omega_{C/k}|$ consists precisely of the separating nodes and certain rational components (see~\cite[Definition 3.2]{catanese1982pluricanonical}). Consequently, the linear system $|\omega_{C/k}|$ is base-point-free if and only if $C$ is inseparable.
\end{theorem}

\begin{remark}
    For the analysis of the canonical map $\Phi_C$, it is not restrictive to assume that $C$ is inseparable. As Catanese explains, if $C$ has separating nodes, one can consider its partial normalization $\nu\colon \tilde{C} \to C$ at these nodes. Let $C_1, \dots, C_n$ be the connected components of $\tilde{C}$. Each component $C_i$ is an inseparable curve, and there is a natural isomorphism on the spaces of global sections: $H^0(C, \omega_{C/k}) \simeq \bigoplus_{i=1}^n H^0(C_i, \omega_{C_i/k})$. Geometrically, this isomorphism implies that the rational map $\Phi_C$ is given by composing the rational inverse of the normalization, $\nu^{-1}$, with the canonical morphisms $\Phi_{C_i}$ on each component $C_i$. The images of these maps lie in mutually disjoint projective subspaces~\cite[Remark 3.6]{catanese1982pluricanonical}.
\end{remark}

The next question concerns the injectivity of the canonical map $\Phi_C$, for which 2-inseparability is a necessary condition.

\begin{theorem}[{{\cite[Theorem E]{catanese1982pluricanonical}}}] \label{thm:necessary_injective}
    Assume $C$ is inseparable. Two singular points $p,q$ of $C$ have the same image under $\Phi_C$ if and only if $\{p,q\}$ is a separating pair. In particular, if $\Phi_C$ is injective, then $C$ is 2-inseparable.
\end{theorem}

We now discuss quasi-hyperelliptic curves, which are those for which the canonical map $\Phi_C$ is not birational.

\begin{definition} \label{def:quasi-hyperelliptic}
    A curve $C$ is called \emph{quasi-hyperelliptic} if there exist two smooth points $x, y$ (possibly $x = y$) such that $h^0(\OO_{C}(x+y)) = 2$.
\end{definition}

\begin{proposition}[{{\cite[Proposition 3.10]{catanese1982pluricanonical}}}] \label{prop:quasi-hyperelliptic}
    Assume $C$ is inseparable. Then the following conditions are equivalent:
    \begin{enumerate}[label=\EquivListLabel]
        \item $C$ is quasi-hyperelliptic.
        \item The canonical map $\Phi_C$ is not birational onto its image.
        \item Two smooth points of $C$ have the same image under $\Phi_C$, or $\Phi_C$ fails to be an embedding at a smooth point.
    \end{enumerate}
\end{proposition}

A quasi-hyperelliptic curve $C$ admits a map of degree $2$ to $\PP^1_k$. However, this map is not necessarily finite, as it may be constant on some components of $C$. This distinction motivates the following more restrictive notion, following the terminology of Catanese.

\begin{definition}[{{\cite[Definition 3.18]{catanese1982pluricanonical}}}]
    An \emph{honestly hyperelliptic curve} is an inseparable Gorenstein curve $C$ that admits a finite morphism $C \to \PP^1_k$ of degree $2$.
\end{definition}

Quasi-hyperelliptic curves can be characterized as those containing an honestly hyperelliptic subcurve, with any remaining components satisfying certain conditions (see~\cite[Theorem F]{catanese1982pluricanonical} for details). If one additionally assumes the curve is \emph{very strongly connected}, then the result simplifies as follows and is an analogue of Lemma~\ref{lem:unique_2_inse}.

\begin{proposition} \label{prop:hyp_vs_quasi-hyp}
    Assume $C$ is very strongly connected. Then $C$ is quasi-hyperelliptic if and only if it is honestly hyperelliptic. In this case, the structure of $C$ is one of the two following types:
    \begin{enumerate}[label=\StatementListLabel]
        \item $C$ is irreducible.
        \item $C$ is a binary curve.
    \end{enumerate}
    Furthermore, the canonical map $\Phi_C$ factors through the degree $2$ morphism $f\colon C \to \PP^1_k$, and its image is a rational normal curve in $\PP_k^{g-1}$.
\end{proposition}

\begin{proof}
    The fact that an honestly hyperelliptic curve is quasi-hyperelliptic is clear from the definitions. The converse is a consequence of~\cite[Theorem F]{catanese1982pluricanonical}. A quasi-hyperelliptic curve $C$ contains an honestly hyperelliptic subcurve $C'$, which is the locus where the map $f\colon C \to \PP^1_k$ is non-constant. Catanese shows that if $C \neq C'$, any remaining connected component must intersect $C'$ in at most two points, counted with multiplicity (cf. proof of \cite[Proposition 3.14 b)]{catanese1982pluricanonical}). This configuration conflicts with the \emph{very strongly connected} property, which forces the equality $C=C'$. Thus, $C$ must be honestly hyperelliptic itself. The assertion about the structure and the factorization is also part of~\cite[Theorem F]{catanese1982pluricanonical}.
\end{proof}

\begin{remark} \label{rem:honestly_hyperelliptic}
    Catanese also shows that the intersection nodes of the binary curve satisfy a certain cross-ratio condition; cf.~\cite[Proposition 3.14]{catanese1982pluricanonical}.
\end{remark}

This leads to a direct connection with our primary notion of hyperellipticity for stable curves.

\begin{corollary} \label{cor:all_the_same}
Let $C$ be a stable, 2-inseparable curve. Then the following are equivalent:
\begin{enumerate}[label=\EquivListLabel]
\item \label{item:hyp} $C$ is hyperelliptic.
\item \label{item:quasi} $C$ is quasi-hyperelliptic.
\item \label{item:honest} $C$ is honestly hyperelliptic.
\end{enumerate}
\end{corollary}

\begin{proof}
    A stable, 2-inseparable curve is very strongly connected, so the equivalence of~\ref{item:quasi} and~\ref{item:honest} follows from Proposition~\ref{prop:hyp_vs_quasi-hyp}.

    To show that~\ref{item:hyp} $\implies$ \ref{item:honest}, assume $C$ is hyperelliptic with hyperelliptic involution $\sigma$. The quotient map $f\colon C \to T := C/\langle \sigma \rangle$ is finite of degree $2$. By 2-inseparability, Lemma~\ref{prop:node_classification} ensures the quotient $T$ is isomorphic to $\PP_k^1$. Thus, $C$ is honestly hyperelliptic.

    To show that~\ref{item:honest} $\implies$ \ref{item:hyp}, assume $C$ is honestly hyperelliptic, admitting a finite degree $2$ morphism $f\colon C \to \PP^1_k$. We must construct the hyperelliptic involution $\sigma$. Since $C$ is a stable, 2-inseparable, and honestly hyperelliptic curve, it is very strongly connected. Thus, by Proposition~\ref{prop:hyp_vs_quasi-hyp}, its structure is either irreducible or a binary curve.

    \begin{case}
        $C$ is irreducible. The morphism $f$ corresponds to a degree $2$ extension of function fields, $k(C)/k(\PP^1_k)$. This extension must be separable. If it were not, it would be purely inseparable (as its degree is a prime, 2), and $\Char k = 2$. Then the morphism $f\circ\nu\colon \tilde{C} \to \PP^1_k$, where $\nu\colon \tilde{C} \to C$ is the normalization, must also be purely inseparable and hence a bijection on $k$-rational points. Thus, $C$ cannot have any nodes; since those are its only singularities, it must be smooth over $k$. By~\cite[Proposition 7.4.21]{liu}, it follows that $g(C) = g(\PP^1_k) = 0$, which contradicts the stability of $C$. Therefore, the extension is separable. A separable extension of degree $2$ is a Galois extension. Its Galois group is of order 2, and the non-trivial element induces the unique automorphism $\sigma$ on $C$ for which $f \circ \sigma = f$. This $\sigma$ is our hyperelliptic involution.
    \end{case}
    
    \begin{case}
        $C$ is a binary curve. By definition, $C$ is the union of two smooth rational components, $C_1$ and $C_2$. Since the morphism $f$ is finite, its restrictions to the components, $f_1\colon C_1 \to \PP^1_k$ and $f_2\colon C_2 \to \PP^1_k$, must be surjective. As the total degree is the sum $\deg(f) = \deg(f_1) + \deg(f_2) = 2$, it follows that $\deg(f_1) = \deg(f_2) = 1$. A degree-one morphism to the smooth curve $\PP^1_k$ is an isomorphism. We can thus define the involution $\sigma$ as the unique automorphism that swaps the two components via these isomorphisms, i.e., $\sigma|_{C_1} = f_2^{-1} \circ f_1$ and $\sigma|_{C_2} = f_1^{-1} \circ f_2$.
    \end{case}
    
    In both cases, we have constructed an involution $\sigma$ such that the quotient $C/\langle \sigma \rangle$ is isomorphic to $\PP^1_k$. Since $\PP^1_k$ is a rational tree and $\sigma$ is not the identity on any component, it is a hyperelliptic involution. Thus, $C$ is hyperelliptic.
\end{proof}

Catanese proved the following generalization of Theorem~\ref{thm:canonical_embedding}.

\begin{theorem}[{{\cite[Theorem G]{catanese1982pluricanonical}}}] \label{thm:canonical_embedding_catanese}
    Assume that the curve $C$ is very strongly connected and that for every singular point $x \in C$ where the conductor ideal $\mathcal{C}_x\,$\footnote{$\mathcal{C}_x := \mathrm{Ann}_{\mathcal{O}_{C,x}}(\nu_*\mathcal{O}_{\tilde{C},x} / \mathcal{O}_{C,x})$ is the largest ideal of the local ring $\mathcal{O}_{C,x}$ that is also an ideal in its normalization $\nu_*\mathcal{O}_{\tilde{C},x}$, where $\nu\colon \tilde{C} \to C$ is the normalization morphism.} is not contained in $\mathfrak{m}_x^2$, the singularity is an $A_n$-singularity for some $n \geq 1$. Then the canonical map $\Phi_C$ is a closed immersion if and only if $C$ is not quasi-hyperelliptic.
\end{theorem}

\begin{remark}
    Catanese conjectured that the technical assumption on the singular points is always satisfied, and he provided a proof for a special case of unibranch singularities (see~\cite[Remark 3.23, Proposition 3.26]{catanese1982pluricanonical}). However, to the author's knowledge, this conjecture remains unproven in general.
\end{remark}

The assumptions of the preceding theorem are fulfilled for a stable, 2-inseparable curve, as its singularities are nodes ($A_1$-singularities) and for stable curves, 2-in\-sep\-a\-ra\-bil\-i\-ty coincides with being very strongly connected. Using Corollary~\ref{cor:all_the_same}, we thus arrive at the following key result, which generalizes the classical statement for smooth curves (Theorem~\ref{thm:canonical_embedding}) to the case of 2-inseparable stable curves.

\begin{corollary} \label{cor:canonical_embedding_general}
    Let $C$ be a 2-inseparable stable curve of genus $g \geq 2$. The canonical map $\Phi_C \colon C \to \PP_k^{g-1}$ is a closed immersion if and only if $C$ is not hyperelliptic.
\end{corollary}

\begin{remark}
    In the proof of~\cite[Proposition A.1]{van2025reduction}, the authors implicitly make the claim that the canonical map embeds 2-inseparable, non-hyperelliptic stable curves of genus $3$ into $\PP^2_k$. They proceed to verify this for the specific cases of type \texttt{2n} and \texttt{BRAID} through explicit Riemann-Roch calculations. Corollary~\ref{cor:canonical_embedding_general} provides the unified theoretical foundation for these findings, generalizing them to all genera $g \geq 2$ and thus removing the need for such a case-by-case verification.
\end{remark}
\section{Proof of Main Theorem} \label{chap:3}

Having characterized hyperelliptic stable curves in the previous chapter, we now proceed to the proof of our main result. The dualizing sheaf $\omega_{\X/\OO_K}$ of the stable model will play a central role in this proof. Although familiarity with the theory as presented in \cite[Section 6.4]{liu} and \cite[Section III.7]{hartshorne} is assumed, we will recall several key facts for the reader's convenience and provide the most explicit descriptions possible of the dualizing sheaf for curves and their models. An excellent reference for the explicit description for curves is Chapter 5 of \cite{conrad2000grothendieck} and its appendix; the primary reference for the dualizing sheaf for models of curves is the work of Morrow \cite{morrow1, morrow2}. We note that Conrad's book \cite{conrad2000grothendieck}, building on the work of Hartshorne \cite{hartshorne1966residues}, treats duality theory in its modern formulation using derived categories---a context much broader than required here. Another accessible treatment of duality theory can be found in \cite{altman1970introduction}.

This chapter is structured as follows. We begin with an exposition of the dualizing sheaf: first for a general family of curves, then more explicitly for a singular curve over an algebraically closed field, and finally for a normal model. We then explain how twists of this sheaf can be used to describe morphisms from the stable model $\X$ of our smooth plane quartic $X$ into $\PP^2_{\OO_K}$, a key step in relating the abstract stable model to a concrete plane model. Applying this technique---which involves what we will later call ``twisting the $1$-tails''---in the non-hyperelliptic case, we will construct a contraction map $\X \to \X_0 \subset \PP^2_{\OO_K}$. The analysis of this map will allow us to prove one direction of the Main Theorem \ref{thm:main_theorem}. This argument crucially relies on the results from the previous chapter; specifically, we use the connection established between hyperelliptic and quasi-hyperelliptic stable curves to apply the theorems of Catanese from Section \ref{sec:quasi-hyp}. For completeness, we note that this contraction map can also be analyzed using direct Riemann-Roch calculations; these alternative proofs are outlined in Appendix \ref{app}. Subsequently, we will prove the converse direction of the Main Theorem. A more explicit version of this latter argument would yield a method for constructing the stable model $\X$ from a given GIT-stable model $\X_0$. This will be explored in a future work.

\subsection{The Dualizing Sheaf}

The dualizing sheaf generalizes the canonical sheaf to singular schemes and families. First, we will recall the general definition of the dualizing sheaf for a family of curves. We then introduce Rosenlicht's theory of regular differentials on singular curves---a theory we will apply to the special fiber of a model. Finally, we turn to the arithmetic setting of a normal model, as in our Setup~\ref{setup:field}, and describe the analogous theory of regular differentials for the model itself.

\subsubsection{General Definition and Properties}

We begin by formulating an axiomatic definition of the dualizing sheaf. Our focus will be on families of curves over an affine base. This definition generalizes the role of the canonical sheaf in the classical Serre duality theorem and provides a powerful framework for relative duality.

\begin{definition}[Dualizing Sheaf]
    Let $f\colon X \to \Spec A$ be a proper morphism of pure relative dimension one, where $A$ is a Noetherian ring. A \emph{dualizing sheaf} for $f$ is a pair $(\omega_f, \tr_f)$, where $\omega_f$ is a quasi-coherent sheaf on $X$ and $\tr_f \colon H^1(X, \omega_f) \to A$ is an $A$-module homomorphism, such that for any quasi-coherent sheaf $\Fs$ on $X$, the natural pairing followed by the trace map
    \[
        \Hom_{\OO_X}(\Fs, \omega_f) \times H^1(X, \Fs) \to H^1(X, \omega_f) \xrightarrow{\tr_f} A
    \]
    induces an isomorphism of $A$-modules:
    \[
        \Hom_{\OO_X}(\Fs, \omega_f) \simeq H^1(X, \Fs)^\vee,
    \]
    where $M^\vee := \Hom_A(M,A)$ is the dual $A$-module.
\end{definition}

\begin{notation} \label{not:dualizing}
When the structure morphism $f\colon X \to \Spec A$ is clear from the context, we will often write $\omega_{X/A}$ instead of $\omega_f$.
\end{notation}

The dualizing sheaf $(\omega_f, \tr_f)$, if it exists, is unique up to a unique isomorphism because it represents the functor $\Fs \mapsto H^1(X, \Fs)^\vee$; see \cite[Lemma 6.4.19]{liu} and \cite[Proposition III.7.2]{hartshorne}. If $\Fs$ is a locally free sheaf, this definition recovers the more familiar form of Serre Duality via the canonical isomorphism $\Hom_{\OO_X}(\Fs, \omega_f) \simeq H^0(X, \Fs^\vee \otimes \omega_f)$.

The existence of a dualizing sheaf is guaranteed under quite general hypotheses. A standard sufficient condition, which holds for all cases in this thesis, is that $f$ is a \emph{Cohen-Macaulay morphism} (i.e., flat with Cohen-Macaulay fibers). A key property for our purposes is invertibility: $\omega_f$ is an invertible sheaf if and only if $f$ is a \emph{Gorenstein morphism}. This is the case, for example, when $f$ is a local complete intersection (l.c.i.) morphism. Following \cite{liu}, we now recall the construction of $\omega_f$ for this l.c.i. setting, which can be motivated by the classical adjunction formula.

\begin{theorem} \label{thm:dualizing_existence}
    With $f\colon X \to \Spec A$ as above, assume additionally that $f$ is a projective l.c.i.\ morphism. Then the dualizing sheaf $\omega_f$ exists and is an invertible sheaf.

    It can be constructed as follows: for any closed immersion $i\colon X \to Z$ into a scheme $Z$ that is smooth over $\Spec A$ (e.g., $Z = \PP^n_A$ for a suitable $n$), $\omega_f$ is given by the formula:
    \[
        \omega_f \simeq \det(\mathcal{C}_{X/Z})^\vee \otimes i^*(\det\Omega^1_{Z/A}),
    \]
    where $\mathcal{C}_{X/Z} = \I/\I^2$ is the conormal sheaf for the ideal sheaf $\I$ of $X$.
\end{theorem}

\begin{proof}
    See \cite[Definition 6.4.7 and Theorem 6.4.32]{liu}. For an alternative construction using $\SheafExt$-functors, see \cite[Proposition III.7.5]{hartshorne} and the subsequent \cite[Theorem III.7.11]{hartshorne}.
\end{proof}

\begin{corollary} \label{cor:smooth_case}
    If $f\colon X \to \Spec A$ is a smooth and projective morphism with fibers of dimension 1, then its dualizing sheaf is the sheaf of relative differential 1-forms:
    \[
        \omega_f \simeq \Omega^1_{X/A}.
    \]
\end{corollary}

\begin{proof}
    As a smooth morphism is a local complete intersection, Theorem~\ref{thm:dualizing_existence} applies. Choosing the identity immersion $i\colon X \to X$, the conormal sheaf $\mathcal{C}_{X/X}$ is the zero sheaf, and the construction formula simplifies to $\omega_f \simeq \det(\Omega^1_{X/A})$. Since $\Omega^1_{X/A}$ is an invertible sheaf in this case, it is isomorphic to its own determinant, proving the claim.
\end{proof}

The definition of the dualizing sheaf extends to morphisms that are not necessarily proper. For a Cohen-Macaulay morphism $f\colon X \to \Spec A$ of finite type and pure relative dimension one over a Noetherian base, the dualizing sheaf is defined as $\omega_f := \mathcal{H}^{-1}(f^!A)$ (cf. \cite[\S3.5]{conrad2000grothendieck} and \cite[Exercise 9.7]{hartshorne1966residues}).

While an exposition of the underlying theory of dualizing complexes and the functor $f^!$ is beyond our scope, we will state its essential properties. The object $\omega_f$ is a coherent sheaf whose construction is local on $X$, and may be obtained by gluing certain $\SheafExt$-sheaves. Furthermore, it is invertible if and only if $f$ is a Gorenstein morphism.

If $f$ is also a quasi-projective l.c.i.\ morphism, then $\omega_f$ can be constructed concretely as in Theorem~\ref{thm:dualizing_existence} using an immersion $i\colon X \to Z$ into a scheme $Z$ that is smooth over $\Spec A$ (in \cite[Definition 6.4.7]{liu}, this construction is used to define this same sheaf under the name canonical sheaf). Moreover, Corollary~\ref{cor:smooth_case} also holds in this context.

\begin{remark}[Functorial Properties]
The dualizing sheaf construction has important functorial properties. For our purposes, the two relevant properties are its commutativity with restriction to any open subscheme and its compatibility with base change, provided either the base change morphism or the morphism $f$ itself is flat (see, e.g., \cite[\S3.6]{conrad2000grothendieck}).
\end{remark}

\subsubsection{Regular Differentials on Singular Curves}

For the description of the dualizing sheaf on a singular curve, the classical theory of residues for smooth curves is the essential tool. We therefore begin by recalling its definition and main result.

Let $C$ be a smooth, projective, and irreducible curve over our algebraically closed field $k$, and let $k(C)$ be its function field. The theory of residues provides, for each closed point $x \in C$, a unique $k$-linear map $\Res_x \colon \Omega^1_{k(C)/k} \to k$. This map is uniquely determined by the condition that, for a choice of local uniformizer $t_x \in \OO_{C,x}$, a differential $\eta \in \Omega^1_{k(C)/k}$ with an image of the form $\sum a_n t_x^n dt_x$ under the map
\[
    \Omega^1_{k(C)/k} \hookrightarrow \Omega^1_{k(C)/k} \otimes_{\OO_{C,x}} \widehat{\OO}_{C,x} \simeq k(\!(t_x)\!) dt_x
\]
has residue $\Res_x(\eta) := a_{-1}$. The well-definedness of this map (i.e., its independence from the choice of $t_x$) is a non-trivial fact (see \cite[Chapter II, 11]{serre1988algebraic}). A key result is the following.

\begin{theorem}[Residue Theorem, see {{\cite[Chapter II, 2, Proposition 6]{serre1988algebraic}}}] \label{thm:residue_thm}
    For any meromorphic differential $\eta \in \Omega^1_{k(C)/k}$ on a smooth projective curve $C$, the \linebreak sum of its residues over all closed points is zero:
    \[
        \sum_{x \in C} \Res_x(\eta) = 0.
    \]
\end{theorem}

On the smooth curve $C$, the canonical sheaf $\omega_{C/k} = \Omega^1_{C/k}$ is naturally a subsheaf of the sheaf of \emph{meromorphic differentials}, denoted $\underline{\Omega}^1_{k(C)/k}$. This is the constant sheaf on $C$ associated with the $k(C)$-vector space of meromorphic 1-forms, $\Omega^1_{k(C)/k}$.
A section $\eta$ of $\underline{\Omega}^1_{k(C)/k}$ over an open set $U \subset C$ belongs to $\omega_{C/k}(U)$ if and only if $\eta$ is \emph{regular} at every closed point $x \in U$. If we write $\eta$ as $\sum_n a_n t_x^n dt_x$, where $t_x$ is a local uniformizer in $\OO_{C,x}$, then $\eta$ is regular at $x$ if and only if $a_n = 0$ for all $n < 0$. Using residues, this regularity condition can be expressed as
\[
    \Res_x(s \cdot \eta) = 0
\]
for all $s \in \OO_{C,x}$. Our goal is now to generalize the description of $\omega_{C/k} \subset \underline{\Omega}^1_{k(C)/k}$ as the sheaf of regular differentials to singular, possibly reducible, curves.

Now, let $C$ be a proper, reduced, connected curve over an algebraically closed field $k$. As a reduced, one-dimensional scheme, $C$ is automatically \emph{Cohen-Macaulay}, which implies that $\omega_{C/k}$ exists, though it is not necessarily invertible. The following lemma suggests that $\omega_{C/k}$ can be identified with a suitable subsheaf of meromorphic differentials.

\begin{lemma}[{{\cite[Lemma 5.2.1]{conrad2000grothendieck}}}]
    Denote by $j\colon U \hookrightarrow C$ the inclusion of the dense open smooth locus of $C$. Then the natural map
    \[
        \omega_{C/k} \to j_* (\omega_{C/k} |_U) \simeq j_*(\Omega^1_{U/k})
    \]
    is injective.
\end{lemma}

Let $\{\xi_i\}$ be the set of generic points of $C$. The total ring of fractions of $C$ is $k(C) := \prod_i k(C_i)$, where $k(C_i) = k(\xi_i)$ is the function field of the irreducible component $C_i = \overline{\{\xi_i\}}$. The \emph{sheaf of meromorphic differentials} on $C$ is defined as the pushforward from the scheme of generic points:
\[ \underline{\Omega}^1_{k(C)/k} := j_* \Omega^1_{k(C)/k}, \] 
where $j \colon \Spec k(C) \to C$ is the natural inclusion.\footnote{Using the notation as in \cite{liu}, one may define this equivalently as $\Omega_{C/k}^1 \otimes_{\OO_C} \mathcal{K}_C$, where $\mathcal{K}_C$ is the sheaf of rational functions on $C$.} Its stalk at a point $x \in C$ can be readily verified to be
\[
    (\underline{\Omega}^1_{k(C)/k})_x = \prod_{C_i \ni x} \Omega^1_{k(C_i)/k}.
\]
We want to define a subsheaf $\omega_{C/k}^\text{reg} \subset \underline{\Omega}^1_{k(C)/k}$ of \emph{regular differentials}, generalizing the notion from the smooth case. To do this, we use the normalization $\nu \colon \tilde{C} \to C$. The total rings of fractions of $C$ and $\tilde{C}$ are canonically isomorphic, which gives a canonical isomorphism of $\OO_C$-modules:
\begin{equation} \label{eq:identification}
    \nu_\ast \underline{\Omega}^1_{k(\tilde{C})/k} \simeq \underline{\Omega}^1_{k(C)/k}.
\end{equation}
We identify these sheaves from now on. In particular, for each irreducible component, we identify
\[
    \Omega^1_{k(C_i)/k} \simeq \Omega^1_{k(\tilde{C}_i)/k},
\]
where $k(\tilde{C}_i) = k(\nu^{-1}(\xi_i))$ is the function field of the corresponding smooth component of $\tilde{C}$.

\begin{definition}
    Let $x$ be a closed point on $C$. We define the $k$-linear map
    \begin{align*}
        \Res_x \colon (\underline{\Omega}^1_{k(C)/k})_x \to k
    \end{align*}
    as follows: Consider an element
    \[
        \eta \in (\underline{\Omega}^1_{k(C)/k})_x = \prod_{C_i \ni x} \Omega^1_{k(C_i)/k} = \prod_{C_i \ni x} \Omega^1_{k(\tilde{C}_i)/k}.
    \]
    For any point $\tilde{x} \in \nu^{-1}(x)$, we denote by $\Res_{\tilde{x}}(\eta)$ the classical residue of the component of $\eta$ corresponding to the unique smooth component of $\tilde{C}$ containing $\tilde{x}$. Then we set
    \[
        \Res_x(\eta) := \sum_{\tilde{x} \in \nu^{-1}(x)} \Res_{\tilde{x}}(\eta).
    \]
\end{definition}

Note that $\Res_x(\eta)$ also makes sense if $\eta$ is a section of $\underline{\Omega}^1_{k(C)/k}$ over an open set $U$ containing $x$, via the natural map $\underline{\Omega}^1_{k(C)/k}(U) \to (\underline{\Omega}^1_{k(C)/k})_x$.

\begin{remark}[The Residue Theorem for Singular Curves]
    An immediate consequence of our definition is that the Residue Theorem (\ref{thm:residue_thm}) extends to the singular curve $C$. This holds because the sum of residues on $C$ is, by construction, the sum of all residues on its smooth normalization $\tilde{C}$, which is zero by the classical theorem.
\end{remark}

With this notion of residue, we can now generalize the condition of regularity to singular curves.

\begin{definition}
    Let $\eta$ be a section of $\underline{\Omega}^1_{k(C)/k}$ over an open set $U \subset C$. We call $\eta$ \emph{regular at a closed point} $x \in U$ if, for all $s \in \OO_{C,x}$, we have
    \begin{equation} \label{eq:regular}
        \Res_x(s \cdot \eta) = 0.
    \end{equation}
    We call the section $\eta$ \emph{regular on $U$} if it is regular at every closed point $x \in U$.
    
    We define the $\OO_C$-module $\omega_{C/k}^\text{reg}$ as the subsheaf of $\underline{\Omega}^1_{k(C)/k}$ whose sections over any open set $U \subset C$ are the differentials regular on $U$.
\end{definition}

\begin{remark} \label{rem:comparing_regularity}
    By our identification in \eqref{eq:identification}, a single meromorphic differential $\eta$ can be viewed on both $C$ and its normalization $\tilde{C}$. It is therefore meaningful to compare our regularity condition from \eqref{eq:regular} at $x \in C$ with the classical notion of regularity on the smooth curve $\tilde{C}$ at points $\tilde{x} \in \nu^{-1}(x)$.

    Clearly, if $\eta$ is regular at all points in the fiber $\nu^{-1}(x)$, then it is also regular at $x$. The converse, however, is not true for two main reasons. First, the residue at $x$ is the sum of all residues at the points above $x$, so it can be zero due to cancellation even if the individual residues are non-zero. Second, even if there is only one point $\tilde{x}$ above $x$, the condition for regularity at $x$ is still weaker; it only requires the residue condition to hold for multipliers $s \in \OO_{C,x}$, which is a proper subring of $\OO_{\tilde{C}, \tilde{x}}$ if $x$ is a singular point.

    It is now easy to see that these two notions of regularity coincide at a point $x$ if and only if $x$ is a smooth point of $C$.
\end{remark}

We have the following fundamental theorem, which identifies the abstract dualizing sheaf with the concrete sheaf of regular differentials.

\begin{theorem}[Rosenlicht]
    Let $C$ be a proper, reduced curve over an algebraically closed field $k$. Let $j\colon U \hookrightarrow C$ be the inclusion of the dense smooth locus. The dualizing sheaf $\omega_{C/k}$ and the sheaf of regular differentials $\omega_{C/k}^{\text{reg}}$ coincide as coherent subsheaves of $j_*(\Omega^1_{U/k}) \subset \underline{\Omega}_{k(C)/k}^1$.
\end{theorem}

\begin{proof}
    This is \cite[Theorem 5.2.3]{conrad2000grothendieck}.
\end{proof}

In this setting, the trace map $\tr_{C/k}$ can also be described concretely as a ``sum of residues'' map; we refer to \cite[Chapter 5.2]{conrad2000grothendieck} for a detailed exposition.

From now on, we will write $\omega_{C/k}$ instead of $\omega_{C/k}^{\text{reg}}$. Meromorphic differentials that are regular at all closed points of $\tilde{C}$ (i.e., sections of $\nu_* \Omega^1_{\tilde{C}/k} \subset \underline{\Omega}_{k(C)/k}^1$) are also regular on $C$ by Remark \ref{rem:comparing_regularity}. This establishes the inclusion
\[
    \nu_* \Omega^1_{\tilde{C}/k} \subset \omega_{C/k},
\]
providing a lower bound for the dualizing sheaf. The following proposition provides a corresponding upper bound, showing that regular differentials can be viewed as meromorphic differentials on the normalization with poles of bounded order along a specific divisor.

\begin{proposition}
    There exists a unique minimal effective Weil divisor $D$ on the normalization $\tilde{C}$ such that
    \[
        \nu_* \Omega^1_{\tilde{C}/k} \subset \omega_{C/k} \subset \nu_* (\Omega^1_{\tilde{C}/k}(D)).
    \]
    This divisor $D$ is called the \emph{conductor divisor}. Its support is precisely the preimage under $\nu$ of the singular locus of $C$.
\end{proposition}

\begin{proof}
    We explain how the conductor divisor $D$ is constructed; see \cite[Lemma 5.2.2]{conrad2000grothendieck} for details. Viewing $\OO_C$ as a subsheaf of $\nu_*\OO_{\tilde{C}}$, the \emph{conductor ideal} is defined as
    \[
        \mathcal{C} := \Ann_{\nu_*\OO_{\tilde{C}}}(\nu_*\OO_{\tilde{C}} / \OO_C).
    \]
    This is the largest ideal sheaf within $\nu_*\OO_{\tilde{C}}$ that is also an ideal sheaf of $\OO_C$. Its support is precisely the singular locus of $C$. Since $\nu$ is a finite morphism, $\mathcal{C}$ corresponds to a unique ideal sheaf $\widetilde{\mathcal{C}}$ on $\tilde{C}$ such that $\nu_*\widetilde{\mathcal{C}} = \mathcal{C}$. This ideal sheaf $\widetilde{\mathcal{C}}$ defines a finite closed subscheme on $\tilde{C}$, which, when viewed as an effective Weil divisor, is our divisor $D$.
\end{proof}

\begin{example}
    For a \emph{Gorenstein singularity} $x \in C$, the structure of the conductor divisor simplifies. Let $D_x$ be the part of the conductor divisor $D$ supported over the point $x$, so that $D = \sum_x D_x$. The degree of this local divisor is related to the $\delta$-invariant of the singularity, $\delta_x := \dim_k((\nu_*\OO_{\tilde{C}})_x / \OO_{C,x})$, by the formula
    \[
        \deg(D_x) = 2\delta_x.
    \]
    One can show that in addition to this pole restriction, a differential must satisfy $\delta_x$ further linear conditions to be regular at $x$.

    In the simplest case of a singularity with $\delta_x=1$ (i.e., a node or a cusp), this formula completely determines the structure of the conductor divisor.
    \begin{description}
        \item[For an \emph{$A_1$ singularity (node)}] The normalization has two points $\{p_1, p_2\}$ lying over the node. The degree is distributed between them, yielding the divisor $D_x = p_1 + p_2$.
        \item[For an \emph{$A_2$ singularity (cusp)}] The singularity is unibranch, meaning there is only one point $p$ on the normalization lying over it. The divisor's degree is concentrated at this point, yielding $D_x = 2p$.
    \end{description}
\end{example}

A section of $\nu_*(\Omega_{\tilde{C}/k}^1(D))$ belongs to $\omega_{C/k}$ if and only if it is regular at every singular point $x \in C$. This is a local condition that depends only on the local ring $\OO_{C,x}$. Since our curve $C$ is of finite type over a field, its local rings are automatically \emph{excellent}, which ensures that the condition depends only on the completion $\widehat{\OO}_{C,x}$. The following proposition describes this condition explicitly for nodes and cusps.

\begin{proposition}[Local Residue Conditions] \label{prop:local_conditions}
    Let $x \in C$ be a singular point, and let $\eta$ be a local section of $\nu_*(\Omega^1_{\tilde{C}/k}(D_x))$.
    \begin{enumerate}[label=\StatementListLabel]
        \item If $x$ is an \emph{$A_1$ singularity (a node)} with preimages $\{p_1, p_2\} \subset \tilde{C}$, so that $D_x = p_1+p_2$, then $\eta$ is a section of $\omega_{C/k}$ at $x$ if and only if the sum of its residues is zero:
        \[
            \Res_{p_1}(\eta) + \Res_{p_2}(\eta) = 0.
        \]
        \item If $x$ is an \emph{$A_2$ singularity (a cusp)} with preimage $\{p\} \subset \tilde{C}$, so that $D_x=2p$, then $\eta$ is a section of $\omega_{C/k}$ at $x$ if and only if its residue is zero:
        \[
            \Res_{p}(\eta) = 0.
        \]
    \end{enumerate}
\end{proposition}

\begin{proof}
    We verify the regularity condition \eqref{eq:regular}, $\Res_x(s \cdot \eta) = 0$ for all $s \in \OO_{C,x}$, by analyzing the completed local rings $\widehat{\OO}_{C,x}$ and their normalizations.

    \begin{prooflist}
        \item The completion $\widehat{\OO}_{C,x}$ embeds into its normalization $k[\![t_1]\!] \times k[\![t_2]\!]$ as the pairs $(f(t_1), g(t_2))$ such that $f(0)=g(0)$. The differential $\eta=(\eta_1, \eta_2)$ has at most simple poles ($D_x=p_1+p_2$). Let $a_{-1} = \Res_{p_1}(\eta)$ and $b_{-1} = \Res_{p_2}(\eta)$. 
        
        For any $s \in \widehat{\OO}_{C,x}$, let $c$ be its constant term (i.e., $f(0)=g(0)=c$). Then
        \[
            \Res_x(s\cdot \eta) = \Res_{p_1}(f\eta_1) + \Res_{p_2}(g\eta_2) = c \cdot a_{-1} + c \cdot b_{-1} = c(a_{-1}+b_{-1}).
        \]
        This vanishes for all $s$ (i.e., for all $c \in k$) if and only if $a_{-1}+b_{-1}=0$.

        \item The completion $\widehat{\OO}_{C,x}$ embeds into the normalization $k[\![t]\!]$ as the subring $k[\![t^2, t^3]\!]$, consisting of power series $s=\sum c_i t^i$ lacking the linear term ($c_1=0$). The differential $\eta$ has a pole of order at most $2$ ($D_x=2p$). Let $\eta = (\sum_{j\geq -2} a_j t^j)dt$.
        
        The residue $\Res_x(s\eta) = \Res_p(s\eta)$ is the coefficient of $t^{-1}$ in the product $s\eta$, which is the sum of terms $c_i a_j$ where $i+j=-1$. The only relevant terms are $c_0 a_{-1}$ and $c_1 a_{-2}$. Since $c_1=0$ for all $s \in \widehat{\OO}_{C,x}$, the condition $\Res_x(s\eta)=0$ for all $s$ reduces to $c_0 a_{-1}=0$ for all $c_0 \in k$. This holds if and only if $a_{-1} = \Res_p(\eta) = 0$.
    \end{prooflist}
\end{proof}

For a semistable curve, Proposition~\ref{prop:local_conditions} therefore specializes to the classical description of the dualizing sheaf found in standard references on moduli theory (see, for instance, \cite[p.\ 2, Fact b)]{deligne-mumford}, \cite[p. 91]{ACGII}, \cite[p. 82]{harris1998moduli}).

\subsubsection{The Dualizing Sheaf of a Normal Model}

We now provide a more explicit description of the dualizing sheaf for a normal model $\X$ of a curve $X$ with a reduced special fiber. We work in the standard arithmetic setting established in Setup~\ref{setup:field}.

\begin{remark} \label{rem:functoriality}
    Since our model $\X$ is a normal, 2-dimensional scheme, it is a Cohen-Macaulay scheme (see \cite[Corollary 8.2.22]{liu}). As the structure morphism $f\colon \X \to \Spec \OO_K$ is flat, this implies that $f$ is a Cohen-Macaulay morphism (see \cite[\href{https://stacks.math.columbia.edu/tag/0C0X}{Lemma 0C0X}]{stacks-project} or \cite[Corollaire (6.3.5)]{EGAIV}), which guarantees the existence of the dualizing sheaf $\omega_{\X/\OO_K}$.
\end{remark}

\begin{proposition} \label{prop:normal_model_omega}
    Let $\X$ be a normal model over $\OO_K$ of the curve $X$ from Setup~\ref{setup:field}. Assume that the special fiber of $\X$ is reduced. The smooth locus of the $\OO_K$-scheme $\X$, denoted by $U$, is then the complement of a finite set of closed points in the special fiber. Let $j \colon U \hookrightarrow \X$ be the inclusion of this locus. Then the natural restriction map
    \[
        \omega_{\X/\OO_K} \to j_*(\omega_{\X/\OO_K}|_U) \simeq j_*(\Omega^1_{U/\OO_K})
    \]
    is an isomorphism.
\end{proposition}

\begin{proof}
    See \cite[p.~229ff.]{conrad2000grothendieck} for the argument.
\end{proof}

For the remainder of this subsection, we fix a model $\X$ satisfying the conditions of the proposition, i.e., a normal model with a reduced special fiber. The isomorphism $\omega_{\X/\OO_K} \simeq j_*(\Omega^1_{U/\OO_K})$ from Proposition~\ref{prop:normal_model_omega} provides a more explicit description of the dualizing sheaf. Since the structure morphism $\X \to \Spec\OO_K$ is flat, the formation of the dualizing sheaf commutes with base change. Restricting this isomorphism to the special fiber $\X_s$ identifies the dualizing sheaf $\omega_{\X_s/k}$ with a specific subsheaf of meromorphic differentials; this subsheaf is precisely the sheaf of \emph{regular differentials} introduced in the previous section. On the generic fiber, we have $\omega_{\X/\OO_K}|_X = \omega_{X/K} = \Omega^1_{X/K}$, as $X$ is a smooth curve over $K$.

The function field of $\X$ is the same as that of its generic fiber, $K(X)$. The sheaf $\omega_{\X/\OO_K}$ can therefore be viewed as a subsheaf of the constant sheaf of meromorphic differentials, $\underline{\Omega}^1_{K(X)/K}$, consisting of those sections that are `regular' on the model $\X$. This notion of regularity needs to be carefully defined and requires the theory of residues on the model.

\begin{remark}
A general theory of residues on algebraic surfaces over a perfect field was developed by A. Par\v{s}in; see \cite{parvsin1976arithmetic}. We require a relative version of this theory that describes residues along the fibers of a morphism from a surface to a curve, which has been investigated in a geometric setting by D. Osipov \cite{osipov1997adele}. The most relevant reference for our arithmetic context is the work of M. Morrow \cite{morrow1, morrow2}. Morrow's results are stated for characteristic $0$ local fields (i.e., complete discretely valued fields with a finite residue field). This theory can be generalized to our setting, which makes no assumptions on the characteristic of $K$ and requires only that $k$ be perfect (a condition satisfied in our case, since $k$ is algebraically closed). A detailed development of this generalization would extend beyond the scope of this work, as we only need the theory for a single result (Proposition \ref{prop:subspace_V_W}). We will therefore limit our discussion to the necessary definitions and briefly comment on the adjustments required for this more general framework.
\end{remark}

On a 2-dimensional model, a residue is no longer defined at a single point of codimension $1$ (as is the case for a smooth curve), but is instead defined for a geometric configuration called a \emph{flag}.

\begin{definition}
A \emph{flag} on $\X$ is a pair $\xi=(x,y)$, where $y$ is a codimension $1$ point on $\X$ specializing to a codimension $2$ point $x$ on $\X$, i.e., a closed point with $x \in \overline{\{y\}}$.
\end{definition}

In geometric terms, a flag $\xi = (x,y)$ consists of a closed point $x$ (which must lie on the special fiber) and an irreducible curve $\overline{\{y\}}$ passing through $x$. The curve $\overline{\{y\}}$ is either \emph{horizontal}, meaning it is the closure of a closed point on the generic fiber $X$, or \emph{vertical}, meaning it is an irreducible component of the special fiber $\X_s$. The flag $\xi$ is called horizontal or vertical accordingly.

To define the residue map $\Res_\xi \colon \Omega^1_{K(X)/K} \to K$ for a flag $\xi = (x,y)$, we first associate with $\xi$ a $K(X)$-algebra $\widehat{K(X)}_\xi$, which is a product of fields. The construction of these fields begins at the closed point $x$ of the flag. Let $\OO_x = \OO_{\X,x}$ be the local ring, which is a 2-dimensional normal local domain. Since the model $\X$ is excellent\footnote{Since $\OO_K$ is a complete DVR, it is excellent. As $\X$ is proper (hence of finite type) over $\Spec \OO_K$, the model $\X$ is also an excellent scheme.}, its completion, $\widehat{\OO}_x$, is also a normal ring (see \cite[Scholie 7.8.3 (v)]{EGAIV}). In this completion, the extension of the prime ideal $\mathfrak{p}_y$ defining the curve $\overline{\{y\}}$ gives rise to one or more height-1 prime ideals, $\mathfrak{P}_1, \dots, \mathfrak{P}_n \subset \widehat{\OO}_x$. These primes correspond to the \emph{formal branches} of the curve $\overline{\{y\}}$ at $x$.

For each formal branch $\mathfrak{P}$, we construct an associated field $L_{\mathfrak{P}}$. We proceed by localization followed by completion. First, we localize the normal ring $\widehat{\OO}_x$ at the height-1 prime $\mathfrak{P}$. The resulting ring, $\OO_{\mathfrak{P}} := (\widehat{\OO}_x)_{\mathfrak{P}}$, is a discrete valuation ring (DVR) but is generally not complete. Second, we complete this DVR to obtain the ring $\widehat{\OO}_{\mathfrak{P}}$. Finally, we define the field $L_{\mathfrak{P}}$ as the field of fractions, $L_{\mathfrak{P}} := \Frac(\widehat{\OO}_{\mathfrak{P}})$. This entire construction, starting from $\OO_x$, involves two distinct completion steps (Completion-Localization-Completion).

By construction, $L_{\mathfrak{P}}$ is a complete discrete valuation field (CDVF). Since the completion of a local ring preserves the residue field, the residue field $l_{\mathfrak{P}}$ of $L_{\mathfrak{P}}$ is the same as that of $\OO_{\mathfrak{P}}$. By the definition of localization, this is $l_{\mathfrak{P}} = \Frac(\widehat{\OO}_x/\mathfrak{P})$. This field $l_{\mathfrak{P}}$ is itself a CDVF because $\widehat{\OO}_x/\mathfrak{P}$ is a 1-dimensional complete local domain. The final residue field (the residue field of $l_{\mathfrak{P}}$) is $k(x)$, which equals $k$ as $k$ is algebraically closed. Thus, the field $L_{\mathfrak{P}}$ is a CDVF whose own residue field is also a CDVF.\footnote{If the final residue field $k$ were finite, $L_{\mathfrak{P}}$ would be called a 2-dimensional local field.} The algebra for the flag is the product of these fields over all branches: $\widehat{K(X)}_\xi := \prod_{i=1}^n L_{\mathfrak{P}_i}$.

The product simplifies to a single field if the curve $\overline{\{y\}}$ has only one formal branch at $x$ (i.e., if $n=1$), meaning the curve is formally irreducible at $x$. This is guaranteed if the curve $C := \overline{\{y\}}$ is regular (smooth) at $x$. Indeed, regularity of $C$ at $x$ implies that its local ring $\OO_{C,x} \simeq \OO_x/\mathfrak{p}_y$ is a regular local ring. Since the model $\X$ is an excellent scheme, this local ring is also excellent. For such rings, the completion $\widehat{\OO}_{C,x}$ is again a regular local ring and therefore an integral domain. The number of formal branches $n$ is the number of minimal prime ideals of this completion; since an integral domain has only one minimal prime ideal (the zero ideal), we conclude that $n=1$. Another sufficient condition here, is that $C$ is horizontal. Indeed, the morphism $C \to \Spec \OO_K$ is proper (as $C$ is closed in the proper scheme $\X$) and has finite fibers; hence it is a finite morphism. Thus $C = \Spec R$, where $R$ is a finite $\OO_K$-algebra and a domain (since $C$ is integral). Since $\OO_K$ is a complete local ring, any finite algebra over it is a product of complete local rings. As $R$ is a domain, this product must be trivial; hence $R$ is a complete local domain. The local ring $\OO_{C,x}$ is $R$, which is already complete and a domain. Thus, $C$ is formally irreducible at $x$, implying $n=1$.

The precise structure of the field $L = L_{\mathfrak{P}}$ as an extension of $K$ depends directly on whether the flag is horizontal or vertical.

If the flag $\xi$ is horizontal, the curve $\overline{\{y\}}$ is not contained in the special fiber, so the uniformizer $\pi$ of $K$ is not in the ideal $\mathfrak{P}$. In this case, the valuation $v_L$ is trivial on $K$ (i.e., $v_L(\pi)=0$). The residue field $l$ is the residue field $\kappa(y)$ of the point $y$ on the generic fiber. It is a finite (totally ramified since we assume $k$ to be algebraically closed) extension of $K$. By the Cohen Structure Theorem, $L$ is a field of formal Laurent series over its residue field,
\[
    L \simeq l(\!(t)\!),
\]
where $t$ is some uniformizer of $L$.

If the flag $\xi$ is vertical, the curve $\overline{\{y\}}$ is a component of the special fiber, so $\pi \in \mathfrak{P}$. The valuation $v_L$ extends $v_K$. Since we assume the special fiber of $\X$ is reduced, the extension is unramified, meaning the ramification index is $e(L/K)=1$. The residue field is $l \simeq k(\!(t)\!)$. As the extension is unramified, the field $L$ is itself the unique unramified CDVF extension of $K$ with this residue field. Therefore, $L$ is isomorphic to the field of generalized Laurent series $K\{\!\{t\}\!\}$, constructed as:
\[
    K\{\!\{t\}\!\} := \left\{ \sum_{j=-\infty}^{\infty} a_j t^j \;\middle|\;
    \begin{array}{l}
        a_j \in K, \; \inf_{j \in \mathbb{Z}} v_K(a_j) > -\infty, \\
        \text{and} \; \lim_{j \to -\infty} v_K(a_j) = \infty
    \end{array}
    \right\}.
\]
Under this isomorphism, the valuation ring of $L$ is the subring $\OO_K\{\!\{t\}\!\} \subset K\{\!\{t\}\!\}$ consisting of formal series with coefficients in $\OO_K$. Its maximal ideal consists of those series whose coefficients lie in the maximal ideal of $\OO_K$.

\begin{example}[Flag at a smooth point]
Let $\xi=(x,y)$ be a flag where $x$ is a smooth point of the special fiber. The completed local ring is $\widehat{\OO}_{\X, x} \simeq \OO_K[\![t]\!]$. In this setting, the algebra for the flag, $\widehat{K(X)}_\xi$, simplifies to a single field $L$.
\begin{enumerate}[label=\StatementListLabel]
    \item If $\overline{\{y\}}$ is a horizontal curve, it corresponds to a prime ideal $\mathfrak{P} = (h(t))$ for some irreducible Weierstrass polynomial $h(t) \in \OO_K[t]$. The residue field $l = \Frac(\OO_K[\![t]\!]/(h(t)))$ is a finite extension of $K$. The associated field is $L \simeq l(\!(s)\!)$, where $s$ can be taken to be $h(t)$.
    \item If $\overline{\{y\}}$ is the unique vertical curve through $x$, it corresponds to the prime ideal $\mathfrak{P} = (\pi)$. The residue field is $l = \Frac(\OO_K[\![t]\!]/(\pi)) \simeq k(\!(t)\!)$. The associated field is $L \simeq K\{\!\{t\}\!\}$.
\end{enumerate}
\end{example}

\begin{example}[Flag at a node]
Let $\xi=(x,y)$ be a flag where $x$ is a nodal singularity of thickness $d \ge 1$ on the special fiber, described locally by $\widehat{\OO}_{\X, x} \simeq \OO_K[\![u,v]\!]/(uv - \pi^d)$.
\begin{enumerate}[label=\StatementListLabel]
    \item If $\overline{\{y\}}$ is a horizontal curve, it is regular at $x$ and corresponds to a unique formal branch. The associated field is $L \simeq l(\!(t)\!)$, where $l$ is a finite extension of $K$ and $t$ is a local parameter for the curve. The algebra for the flag is therefore the single field $\widehat{K(X)}_\xi = L$.
    \item If $\overline{\{y\}}$ is a vertical curve, the local ring of the special fiber is $\widehat{\OO}_{\X,x}/\pi\widehat{\OO}_{\X,x} \simeq k[\![u,v]\!]/(uv)$, which has two branches at $x$ given by $u=0$ and $v=0$. We have two geometric possibilities:
    \begin{description}
        \item[Intersection of two components:] The curve $\overline{\{y\}}$ corresponds to just one of the local branches (e.g., the one given by $u=0$). In this case, there is only one formal branch associated with the flag. The algebra for the flag is the single field:
        \[ \widehat{K(X)}_\xi \simeq K\{\!\{v\}\!\}. \]
        \item[Self-intersection:] The curve $\overline{\{y\}}$ is locally self-intersecting at the node, meaning it corresponds to both local branches. In this case, there are two formal branches. The algebra for the flag is a product of two fields:
        \[ \widehat{K(X)}_\xi \simeq K\{\!\{u\}\!\} \times K\{\!\{v\}\!\}. \]
    \end{description}
\end{enumerate}
\end{example}

\begin{remark}
In the special case where the fields $K$ and $k$ have the same characteristic, we can describe the field $L$ arising from a vertical flag---which we identified as being isomorphic to $K\{\!\{t\}\!\}$---in an alternative way. By the Cohen Structure Theorem, the base field $K$ is isomorphic to a field of formal Laurent series over its residue field, so $K \simeq k(\!(\pi)\!)$. Since $L$ is the unique unramified extension of $K$ with residue field $l \simeq k(\!(t)\!)$, it must be isomorphic to the iterated Laurent series field $L \simeq k(\!(t)\!)(\!(\pi)\!)$ (compare with the analogous construction for a surface over a curve in \cite[1.2, before Definition 6]{osipov1997adele}). The order of the variables here is crucial. An element in $k(\!(t)\!)(\!(\pi)\!)$ is a Laurent series in $\pi$ whose coefficients are themselves Laurent series in $t$. This is structurally different from reversing the order to get $K(\!(t)\!) = k(\!(\pi)\!)(\!(t)\!)$. By ``swapping'', i.e., by writing an element of $k(\!(t)\!)(\!(\pi)\!)$ as a series in $t$ with coefficients in $K = k(\!(\pi)\!)$, we get an isomorphism $k(\!(t)\!)(\!(\pi)\!) \simeq k(\!(\pi)\!)\{\!\{t\}\!\}$ (see \cite[Lemma 2.10]{morrow2012introduction}). The construction $\{\!\{\cdot\}\!\}$ provides a description that is valid in general, without assumptions on the characteristic.

Note that our assumption that $\X_s$ is reduced simplifies the situation significantly. In general, in the vertical case, we would need to take a finite totally ramified extension of $K\{\!\{t\}\!\}$. In the terminology of \cite{morrow1}, our assumption forces the field $L$ to be ``standard,'' i.e., of the form $K\{\!\{t\}\!\}$. The fields arising here are called 2-dimensional local fields (though it is often assumed that $k$ is finite, this can be relaxed to $k$ being perfect). For more background and proofs, we refer to \cite{morrow2012introduction} and \cite{fesenko2000invitation}.
\end{remark}

Let $L$ be a factor of the product of fields $\widehat{K(X)}_\xi$. We define the module of continuous differentials as
\[
    \Omega_{L/K}^{1,\mathrm{cts}} := \Omega_{K(X)/K} \otimes_{K(X)} L.
\]
This is a one-dimensional vector space over $L$. Alternatively, this vector space can be constructed intrinsically from $L$, without reference to the global differentials $\Omega^1_{K(X)/K}$. The module of Kähler differentials $\Omega^1_{L/K}$ is typically too large, as it does not account for the valuation of $L$; it is typically an infinite-dimensional vector space over $L$ (see \cite[II.11]{serre1988algebraic}). However, one can form its associated separated module, also called the module of continuous differentials. By \cite[Lemma 3.11]{morrow1}, this construction yields a module isomorphic to the one in our definition.

We now define a residue map for the field $L$:
\[
    \Res_{L} \colon \Omega_{L/K}^{1,\mathrm{cts}} \to K.
\]
If the flag $\xi$ is horizontal, we recall the isomorphism $L \simeq l(\!(t)\!)$ for a choice of uniformizer $t$. Any differential $\omega \in \Omega_{L/K}^{1,\mathrm{cts}}$ can be written as $\omega = \sum a_n t^n \, dt$. The residue is then defined by taking the trace of the coefficient $a_{-1}$:
\[
    \Res_{L} \colon \Omega_{L/K}^{1,\mathrm{cts}} \to K, \quad \omega = \sum a_n t^n \, dt \mapsto \Trace_{l/K}(a_{-1}).
\]
Here $\Trace_{l/K} \colon l \to K$ is the trace map for the finite extension $l/K$.

If the flag $\xi$ is vertical, we have the isomorphism $L \simeq K\{\!\{t\}\!\}$, where $t$ is a lift of a uniformizer of the residue field $l$. Any differential can again be written as $\omega = \sum a_n t^n \, dt$. In this case, the residue is defined as:
\[
    \Res_{L} \colon \Omega_{L/K}^{1,\mathrm{cts}} \to K, \quad \omega = \sum a_n t^n \, dt \mapsto -a_{-1}.
\]

\begin{remark}
The extra minus sign is inserted to ensure the reciprocity laws hold (see Theorem \ref{thm:reciprocity_laws}), creating the necessary sign asymmetry between horizontal and vertical residues. In the analogous geometric theory using 2-forms developed in \cite{osipov1997adele}, this asymmetry arises implicitly. The local residue definitions there do not feature explicit signs, but they require the 2-form to be written in a specific order of differentials. Fixing an orientation (e.g., $du \wedge dt$) forces a reordering (to $dt \wedge du = -du \wedge dt$) when applying one type of residue definition but not the other, thus introducing the required asymmetry via the anti-commutativity of the wedge product.
\end{remark}

\begin{remark}
Morrow's works \cite{morrow1, morrow2} fundamentally distinguish between the ``mixed- and equal-characteristic'' cases. In his setup, he always has $\mathrm{char}(K) = 0$ and $\mathrm{char}(k) = p$. Hence, the pair of fields $(L, l)$ itself has either characteristic $(0,0)$ (equal characteristic) or $(0,p)$ (mixed characteristic). The first case happens if and only if the flag is horizontal, and the second if and only if the flag is vertical. In our more general setup, this equivalence breaks down, and we can also have the equal characteristic case for a vertical flag. In this setting, one must distinguish between vertical and horizontal flags directly.
\end{remark}

\begin{lemma}
The definition of $\Res_L$ does not depend on the choice of the (lift of) local parameter used to establish the isomorphism $L \simeq l(\!(t)\!)$ or $L \simeq K\{\!\{t\}\!\}$.
\end{lemma}

\begin{proof}
In the horizontal case, this is the classical statement for the invariance of residues on curves (see \cite[II.7]{serre1988algebraic}). For the vertical, mixed-characteristic case (where $\mathrm{char}(L) = 0$ and $\mathrm{char}(l) = p$), the result is given in \cite[Prop. 2.19]{morrow1}. The remaining vertical, equal-characteristic case can then be deduced using the technique of ``prolongation of algebraic identities,'' as in \cite[II.7]{serre1988algebraic}; compare also with \cite[Proposition 4]{parvsin1976arithmetic}.
\end{proof}

We now use the preceding constructions to define the residue for a flag.

\begin{definition}
Let $\xi = (x,y)$ be a flag on $\X$ and let $\widehat{K(X)}_\xi = \prod_i L_i$. The residue map $\Res_{\xi} \colon\allowbreak \Omega^1_{K(X)/K} \to K$ is defined as the composition of the natural embedding of the differentials into the product of the local continuous differentials, followed by the sum of the local residues defined above:
\[
    \Omega^1_{K(X)/K} \to \Omega^1_{K(X)/K} \otimes_{K(X)} \widehat{K(X)}_\xi \simeq \prod_i \Omega_{L_i/K}^{1,\mathrm{cts}} \xrightarrow{\sum \Res_{L_i}} K.
\]
\end{definition}

This new residue map unifies the concepts of residues on the generic and special fibers. We have not yet formally defined residues on the generic fiber $X$, as it is defined over the possibly non-algebraically closed field $K$. The definition is standard: the residue at a closed point $y \in X$ is obtained by first taking the classical residue in the completion of $K(X)$ at $y$, which yields an element in the residue field $\kappa(y)$. Since $\kappa(y)$ is a finite extension of $K$, one then applies the field trace map, $\Trace_{\kappa(y)/K} \colon \kappa(y) \to K$, to obtain an element in $K$. The following proposition connects these two definitions.

\begin{proposition} \label{prop:connection_generic_residue}
Let $\xi = (x,y)$ be a flag on $\X$ such that $\overline{\{y\}}$ is horizontal. Then for any $\eta \in \Omega^1_{K(X)/K}$ we have
\[
    \Res_\xi(\eta) = \Res_{y}(\eta) \in K,
\]
where the right-hand side is the residue at the closed point $y \in X$.
\end{proposition}

\begin{proof}
This follows immediately from the definitions, as for a horizontal flag, the residue field is $l=\kappa(y)$, and both $\Res_\xi$ and $\Res_y$ are defined as the trace of the classical residue; see also \cite[Lemma 3.8]{morrow2}.
\end{proof}

The relation to the residues on the special fiber requires considering how differentials reduce modulo $\pi$. This reduction is only meaningful for differentials that are integral.

\begin{definition}
Let $y \in \X$ be a codimension $1$ point of $\X$. A differential $\eta \in \Omega^1_{K(X)/K}$ is called \emph{integral at $y$} if it belongs to the $\OO_{\X, y}$-module $\Omega^1_{\OO_{\X,y}/\OO_K} \subset \Omega^1_{K(X)/K}$.
\end{definition}

The module $\Omega^1_{\OO_{\X,y}/\OO_K}$ is the stalk of $\Omega^1_{U/\OO_K}$ at $y$, where $U$ is the smooth locus of $\X$. By Proposition~\ref{prop:normal_model_omega}, this module is also the stalk of the dualizing sheaf $\omega_{\X/\OO_K}$ at $y$. Note that there is a well-defined reduction map $\Omega^1_{\OO_{\X,y}/\OO_K} \to \Omega^1_{\kappa(y)/k}$ to the space of meromorphic differentials on the curve $\overline{\{y\}}$.

\begin{proposition} \label{prop:connection_special_residue}
    Let $\xi = (x,y)$ be a vertical flag. If $\eta \in \Omega^1_{K(X)/K}$ is integral at $y$, then $\Res_\xi(\eta) \in \OO_K$.
    Moreover, the reduction of the residue modulo $\pi$ is given by
    \[
        \overline{\Res_\xi(\eta)} = -\Res_{x}(\overline{\eta}) \in k,
    \]
    where the right-hand side is the residue on the (possibly singular) curve $C := \overline{\{y\}}$ over the algebraically closed field $k$, and $\overline{\eta}$ is the image under the map $\Omega^1_{\OO_{\X,y}/\OO_K} \to \Omega^1_{\kappa(y)/k}$.
\end{proposition}

\begin{proof}
    Let $\widehat{K(X)}_\xi = \prod_i L_i$ be the algebra associated with the flag $\xi=(x,y)$, where each field $L_i \simeq K\{\!\{t_i\}\!\}$ corresponds to a formal branch of the curve $C=\overline{\{y\}}$ at $x$. We have a natural embedding $\OO_{\X,y} \to \OO_{L_i} \simeq \OO_K\{\!\{t_i\}\!\}$, and we define the module of integral continuous differentials as\footnote{This module can also be defined intrinsically. For ``standard'' fields, such as the ones considered here, the definition via base change is equivalent to the intrinsic one (cf. \cite[Lemma 2.18]{morrow1}).}
    \[
        \Omega^{1,\mathrm{cts}}_{\OO_{L_i}/\OO_K} := \Omega^1_{\OO_{\X,y}/\OO_K} \otimes_{\OO_{\X,y}} \OO_{L_i}.
    \]
    This is a free $\OO_{L_i}$-module generated by $dt_i$ and is naturally an $\OO_{L_i}$-submodule of $\Omega^{1,\mathrm{cts}}_{L_i/K}$. Since the differential $\eta$ is integral at $y$, its image in each component $\Omega^{1,\mathrm{cts}}_{L_i/K}$ must lie within this submodule. This implies that in the local expansion $\eta_i = \sum_j a_{j,i} t_i^j dt_i$, all coefficients $a_{j,i}$ are in $\OO_K$. The residue is given by $\Res_\xi(\eta) = \sum_i (-a_{-1,i})$, which is therefore an element of $\OO_K$. This proves the first assertion.

    For the second assertion, we relate the reduction $\overline{\Res_\xi(\eta)} = -\sum_i \overline{a_{-1,i}}$ to the residue on the special fiber, $\Res_x(\overline{\eta}) = \sum_{\tilde{x} \in \nu^{-1}(x)} \Res_{\tilde{x}}(\overline{\eta})$, where $\nu\colon \tilde{C} \to C$ is the normalization of the curve $C := \overline{\{y\}}$. The key is the one-to-one correspondence between the formal branches (indexed by $i$) and the geometric branches (the points $\tilde{x} \in \nu^{-1}(x)$). This correspondence is guaranteed because the scheme $\X$ is excellent (see \cite[Theorem 6.5 (4)]{TopicsLocalAlgebra}).

    Under this correspondence, the reduction of the local expansion of $\eta$ at a formal branch is precisely the local expansion of the reduced differential $\overline{\eta}$ at the corresponding geometric branch. It follows directly that the classical residue of $\overline{\eta}$ is the reduction of the coefficient:
    \[
        \Res_{\tilde{x}_i}(\overline{\eta}) = \overline{a_{-1,i}}.
    \]
    Summing over all branches, we obtain
    \[
        \Res_x(\overline{\eta}) = \sum_i \Res_{\tilde{x}_i}(\overline{\eta}) = \sum_i \overline{a_{-1,i}}.
    \]
    Comparing this with the reduction of the arithmetic residue, we conclude
    \[
        \overline{\Res_\xi(\eta)} = -\sum_i \overline{a_{-1,i}} = -\Res_x(\overline{\eta}).
    \]
\end{proof}

We can now explicitly characterize $\omega_{\X/\OO_K}$ using these residues, generalizing the definition of regular differentials. Although this result will not be used in the remainder of the thesis, we state it for completeness as it provides an elegant generalization.

\begin{definition}[Regular Differentials on $\X$]
A meromorphic differential form \linebreak $\omega \in \Omega^1_{K(X)/K}$ is called \emph{regular} at a closed point $x \in \X$ if for every curve $\overline{\{y\}}$ passing through $x$ (i.e., for every flag $\xi = (x,y)$) and every function $g \in \OO_{\X,y}$, the residue is integral:
\[
    \Res_\xi(g \cdot \omega) \in \OO_K.
\]
It is called \emph{regular} on an open subset $U \subset \X$ if it is regular at every closed point $x \in U$.
\end{definition}

\begin{theorem} \label{thm:dualizing_sheaf_residues}
The relative dualizing sheaf $\omega_{\X/\OO_K}$ is precisely the sheaf of regular differentials on $\X$, i.e.,
\[
    \omega_{\X/\OO_K}(U) = \{ \omega \in \Omega^1_{K(X)/K} \ssep \omega \text{ is regular on } U \}.
\]
\end{theorem}

\begin{proof}
    This is \cite[Theorem 1.1 and Theorem 5.7]{morrow1}. We note that the proof in the cited source is given under more restrictive assumptions on the field $K$, though the result is expected to hold in greater generality.
\end{proof}

\begin{remark}
For a flag $\xi = (x,y)$ on $\X$ where $\overline{\{y\}}$ is horizontal, the regularity condition is stronger than it might initially appear. The condition $\Res_\xi(g \cdot \omega) \in \OO_K$ for all $g \in \OO_{\X,y}$ is actually equivalent to $\Res_{\xi}(g \cdot \omega) = 0$. This is because $\pi$ is invertible in the DVR $\OO_{\X,y}$ (as $\pi$ does not vanish on horizontal curves). If $R = \Res_\xi(g \cdot \omega) \in \OO_K$, then we must also have $\Res_\xi((g/\pi^n) \cdot \omega) = R/\pi^n \in \OO_K$ for all $n \geq 1$, since $g/\pi^n \in \OO_{\X,y}$. The only element in the complete DVR $\OO_K$ divisible by arbitrary powers of $\pi$ is zero. Thus $R=0$.
\end{remark}

We have a generalization of the residue theorem (Theorem \ref{thm:residue_thm}), known as the Reciprocity Laws. It comes in two versions, which will be essential in the proof of Proposition \ref{prop:subspace_V_W}.

\begin{theorem}[Reciprocity Laws] \label{thm:reciprocity_laws}
For any meromorphic differential form $\omega \in \Omega^1_{K(X)/K}$, the following hold:
\begin{enumerate}[label=\StatementListLabel]
    \item \label{thm:reciprocity_laws_i} \textbf{Reciprocity Around a Point:} For any closed point $x \in \X$, we have $\Res_{(x,y)}(\omega) = 0$ for almost all curves $\overline{\{y\}}$ containing $x$, and
    \[
        \sum_{\overline{\{y\}} \ni x} \Res_{(x,y)}(\omega) = 0.
    \]
    \item \label{thm:reciprocity_laws_ii} \textbf{Reciprocity Along a Vertical Curve:} For any vertical curve $\overline{\{y\}} \subset \X$, the sum converges in $K$ and
    \[
        \sum_{x \in \overline{\{y\}}} \Res_{(x,y)}(\omega) = 0.
    \]
\end{enumerate}
\end{theorem}
\begin{proof}
See \cite[Theorem 4.1]{morrow1} for the first statement and \cite[Theorem 3.1]{morrow2} for the second. The proofs can be adjusted to our case; see also \cite[Proposition 6, 7]{osipov1997adele} for the case $\mathrm{char}(K) = \mathrm{char}(k)$.
\end{proof}

The following lemma, a straightforward consequence of the preceding theorem, will be useful for a later proof.

\begin{lemma} \label{lem:reside_one_vertical}
Let $x$ be a closed point on the special fiber, such that $\overline{\{y\}}$ is the only vertical curve passing through $x$. Then for a meromorphic differential form $\eta \in \Omega^1_{K(X)/K}$ we have
\[
    -\Res_{(x,y)}(\eta) = \sum_{z} \Res_z(\eta),
\]
where the sum runs over all closed points $z$ of the generic fiber $X$ that specialize to $x$.
\end{lemma}

\begin{proof}
Since $\overline{\{y\}}$ is the only vertical curve passing through $x$, all other curves passing through $x$ are horizontal, i.e., of the form $\overline{\{z\}}$ with $z$ a closed point of $X$ specializing to $x$. The lemma now immediately follows from the reciprocity law around a point together with Proposition \ref{prop:connection_generic_residue}.
\end{proof}

\subsection{Contraction Maps of the Stable Model} \label{sec:contraction}

We now apply the theory to our specific case. We use the notation as in Setup \ref{setup:field}, and assume additionally that $X \subset \PP^2_K$ is a smooth plane quartic. We further assume that its stable model $\X$ is defined over $\OO_K$, and let $\oo{X} := \X_s$ be its special fiber. If a plane model $\X_0 \subset \PP_{\OO_K}^2$ is dominated by $\X$, it provides a morphism
\[
    \phi\colon \X \to \PP^2_{\OO_K},
\]
with the property that its restriction to the generic fiber is the canonical embedding $X \subset \PP^2_K$. Any such morphism is determined by the line bundle $\LL := \phi^* \OO_{\PP^2_{\OO_K}}(1)$, which must be generated by its global sections, and a choice of three such sections. The property that $\phi$ on the generic fiber is the canonical embedding translates to
\[
    \LL|_{X} \simeq \omega_{X/K} \simeq \OO_{\PP^2_{K}}(1)|_X.
\]
Thus, the three chosen sections, when restricted to the generic fiber, must yield a $K$-basis for the three-dimensional $K$-vector space $H^0(X, \LL|_X)$. Moreover, the $\OO_K$-module $H^0(\X, \LL)$ is finitely generated and torsion-free, as $\LL$ is a line bundle on the integral scheme $\X$. Since $\OO_K$ is a principal ideal domain, this implies that $H^0(\X, \LL)$ is a free $\OO_K$-module. Since the base change $\Spec K \to \Spec \OO_K$ is flat,
\[
    H^0(\X, \LL) \otimes_{\OO_K} K \simeq H^0(X, \LL|_X),
\]
and hence the rank of $H^0(\X, \LL)$ is three. Thus, the choice of three sections must be an $\OO_K$-basis of this free module.

Now we consider the base change to the special fiber:
\[
    \oo{\phi} := \phi|_{\oo{X}} \colon \oo{X} \to \PP^2_k.
\]
This morphism corresponds to the line bundle $\oo{\LL} := \LL|_{\oo{X}}$. To analyze its global sections, we consider the standard short exact sequence arising from multiplication by $\pi$:
\[
    0 \to \LL \xrightarrow{\cdot \pi} \LL \to \oo{\LL} \to 0.
\]
This induces the long exact sequence in cohomology:
\[
    0 \to H^0(\X, \LL) \xrightarrow{\cdot \pi} H^0(\X, \LL) \to H^0(\oo{X}, \oo{\LL}) \to H^1(\X, \LL) \xrightarrow{\cdot \pi} H^1(\X, \LL) \to \dots
\]
From this, we deduce the following exact sequence of $k$-vector spaces:
\begin{equation} \label{eq:ex_seq}
    0 \to H^0(\X, \LL) \otimes_{\OO_K} k \to H^0(\oo{X}, \oo{\LL}) \to H^1(\X, \LL)[\pi] \to 0.
\end{equation}
This exact sequence allows us to draw several conclusions.

First, the dimension of the $k$-vector space $H^0(\oo{X}, \oo{\LL})$ is at least the rank of the free module $H^0(\X, \LL)$, so $\dim_k H^0(\oo{X}, \oo{\LL}) \ge 3$. This is an instance of ``upper semicontinuity'' (\cite[Theorem 5.3.20]{liu}). Moreover, equality holds if and only if $H^1(\X, \LL)$ is torsion-free. The torsion submodule $H^1(\X, \LL)[\pi]$ precisely measures the obstruction to lifting sections from $\oo{X}$ to $\X$.

Second, the morphism $\oo{\phi}$ is determined by the image of the $\OO_K$-basis of $H^0(\X, \LL)$ in $H^0(\oo{X}, \oo{\LL})$. The injection on the left of sequence \eqref{eq:ex_seq} guarantees that these images are linearly independent over $k$. This image, let's call it $V$, is therefore a 3-dimensional subspace of $H^0(\oo{X}, \oo{\LL})$. The map $\oo{\phi}$ is induced by the linear series $(\oo{\LL}, V)$, so understanding this subspace is crucial for describing the geometry of its image.

Since the line bundle $\LL$ induces the morphism $\phi$, it must be generated by its global sections. By an application of Nakayama's Lemma, this property on $\X$ is equivalent to the base-point-freeness of the linear series $(\oo{\LL}, V)$ on the special fiber $\oo{X}$. As we could not find a suitable reference for this specific statement, we state and prove it in the following general lemma for the reader's convenience.

\begin{lemma} \label{lem:global-generation_very-ample}
    Let $\mathcal{Y}$ be a proper scheme over a discrete valuation ring $\OO_K$, and let $\mathcal{F}$ be an invertible sheaf on $\mathcal{Y}$. Let $s_0, \dots, s_d$ be global sections of $\mathcal{F}$, and denote by $\mathcal{F}_s$ and $\bar{s}_i$ their respective restrictions to the special fiber $\mathcal{Y}_s$.
    \begin{enumerate}[label=\StatementListLabel]
        \item The sections $s_i$ generate $\mathcal{F}$ if and only if the restricted sections $\bar{s}_i$ generate $\mathcal{F}_s$.
        \item If these sections generate the respective sheaves, the induced morphism $f\colon \mathcal{Y} \to \PP^d_{\OO_K}$ is a closed immersion if and only if the corresponding morphism on the special fiber, $f_s\colon \mathcal{Y}_s \to \PP^d_k$, is a closed immersion.
    \end{enumerate}
\end{lemma}

\begin{proof}
    The scheme $\mathcal{Y}$ is Noetherian because it is proper (and thus of finite type) over the Noetherian base $S = \Spec \OO_K$.
    \begin{prooflist}
        \item The forward implication holds because generation by global sections is preserved under base change. For the converse, the sections $s_0, \dots, s_d$ define a homomorphism of $\OO_\mathcal{Y}$-modules $\alpha\colon \OO_{\mathcal{Y}}^{d+1} \to \mathcal{F}$. Let $\mathcal{C} := \coker \alpha$ be its cokernel, which is a coherent sheaf. The sections generate $\mathcal{F}$ if and only if $\mathcal{C} = 0$.

        By assumption, the restricted sections $\bar{s}_i$ generate $\mathcal{F}_s$, so the induced map $\alpha_s$ is surjective. By right-exactness of the tensor product, the restriction of the cokernel is $\mathcal{C}_s \simeq \mathcal{C}/\pi\mathcal{C} = \coker(\alpha_s)$, which is zero. Thus, we have $\mathcal{C} = \pi\mathcal{C}$.

        For any $y \in \mathcal{Y}_s$, the stalk $\mathcal{C}_y$ is a finitely generated $\OO_{\mathcal{Y},y}$-module since $\mathcal{C}$ is coherent. As $\pi$ lies in the maximal ideal of $\OO_{\mathcal{Y},y}$, Nakayama's lemma implies $\mathcal{C}_y=0$. The support of $\mathcal{C}$ is therefore a closed subset of $\mathcal{Y}$ that is disjoint from the special fiber. Since the structure morphism $\mathcal{Y} \to S$ is proper (hence closed), the image of the support in $S$ is a closed subset not containing the closed point $s$. As $S$ is the spectrum of a DVR, this image must be empty, which implies that the support of $\mathcal{C}$ is empty and hence $\mathcal{C}=0$.

        \item The forward implication holds because closed immersions are stable under base change. Conversely, assume $f_s$ is a closed immersion. Since $\mathcal{Y}$ and $\PP^d_{\OO_K}$ are proper schemes over the Noetherian base $S$, and since $f_s$ is a closed immersion, \cite[Proposition 4.6.7(i)]{EGAIII1} implies that $f$ is a closed immersion over an open neighborhood of $s$ in $S$. As $S$ is the spectrum of a DVR, the only such neighborhood is $S$ itself. Therefore, $f$ is a closed immersion.
    \end{prooflist}
\end{proof}

To describe the line bundle $\LL$ more explicitly, we use the theory of divisors on $\X$. As mentioned previously, $\X$ is a normal scheme (hence in particular regular in codimension one), which allows for a well-behaved theory of Weil divisors. However, since $\X$ may not be regular, we must carefully distinguish between Cartier and Weil divisors. On a normal scheme, the group of Cartier divisors, $\Div(\X)$, embeds into the group of Weil divisors, $Z^1(\X)$ (we adopt the notation of \cite{liu}), as the subgroup of Weil divisors that are locally principal (see \cite[Remark II.6.11.2]{hartshorne}). This local condition is only relevant at the singular points of $\X$, which are precisely the nodes of the special fiber with thickness $> 1$ (see Proposition \ref{prop:thickness}). We can reformulate this condition as the exact sequence
\[
    0 \to \Div(\X) \to Z^1(\X) \to \bigoplus_{x \in \X_{\text{sing}}} \Cl(\OO_{\X,x}).
\]
While the local class groups $\Cl(\mathcal{O}_{\mathcal{X},x})$ are in general difficult to compute, in our case we may replace them with their completions, $\Cl(\widehat{\mathcal{O}}_{\mathcal{X},x})$, which are easier to analyze (see Lemma \ref{lem:Cartier_Weil}). This replacement is justified because the stable model $\mathcal{X}$, being proper over the complete DVR $\mathcal{O}_K$, is an excellent scheme. Consequently, the local rings $\mathcal{O}_{\mathcal{X},x}$ are excellent normal local rings. For such rings, the completion is again normal (see \cite[Scholie 7.8.3 (v)]{EGAIV}) and the natural homomorphism $\Cl(\mathcal{O}_{\mathcal{X},x}) \hookrightarrow \Cl(\widehat{\mathcal{O}}_{\mathcal{X},x})$ is injective (see \cite[Corollary 6.12]{fossum1973divisor}).

\begin{remark}
    We can also consider the Picard group of $\X$, $\Pic(\X)$, which may be identified with the group $\CaCl(\X)$ of Cartier divisors modulo principal divisors. The sequence from above descends to the quotients of divisor classes, yielding the well-known exact sequence relating the Picard group and the Class group (see, for instance, \cite[Proposition 2.15]{hartshorne1994generalized}):
    \[
        0 \to \Pic(\X) \to \mathrm{Cl}(\X) \to \bigoplus_{x \in \X_{\text{sing}}} \mathrm{Cl}(\OO_{\X,x}).
    \]
\end{remark}

We now compare $\LL$ with the dualizing sheaf $\omega_{\X/\OO_K}$ by considering their difference. On the generic fiber, this gives
\[
    (\LL \otimes \omega_{\X/\OO_K}^{-1})|_X \simeq \LL|_X \otimes (\omega_{\X/\OO_K}|_X)^{-1} \simeq \omega_{X/K} \otimes \omega_{X/K}^{-1} \simeq \OO_X.
\]
A line bundle on $\X$ that is trivial on the dense generic fiber must be of the form $\OO_{\X}(D)$, where $D$ is a Cartier divisor whose support is contained entirely within the special fiber $\oo{X}$. This leads to our central identification
\[
    \LL \simeq \omega_{\X/\OO_K}(D)
\]
for some Cartier divisor $D$ with $\Supp(D) \subseteq \oo{X}$. Since the irreducible components of the special fiber, $\{\oo{X}_i\}$, are the only prime divisors of $\X$ fully contained in $\oo{X}$, the divisor $D$ must be a linear combination $D = \sum m_i \oo{X}_i$ for some integers $m_i \in \ZZ$. A natural question arises: what conditions must the coefficients $(m_i)$ satisfy for this Weil divisor to be Cartier?

\begin{lemma} \label{lem:Cartier_Weil}
    Let $\X$ be our stable model, with special fiber $\oo{X} = \sum_{i=1}^n \oo{X}_i$. A Weil divisor of the form $D = \sum m_i \oo{X}_i$ is Cartier if and only if for every node $x_0 \in \oo{X}_i \cap \oo{X}_j$, the congruence
    \[
        m_i \equiv m_j \pmod{d}
    \]
    holds, where $d$ is the thickness of the singularity $x_0$ in $\X$.
\end{lemma}

\begin{proof}
    A Weil divisor is Cartier if it is locally principal, so this question is local and only relevant at the singular points. Thus, in view of Proposition \ref{prop:thickness}, we may assume for the proof that
    \[
        \X = \mathrm{Spec}(A) \quad \text{where} \quad A := \OO_K[\![u,v]\!]/(uv - \pi^d), \quad d \geq 1,
    \]
    and $D = m_1 D_1 + m_2 D_2$, where $D_1$ and $D_2$ are the prime divisors corresponding to the components of the special fiber.
    
    It is easy to see that $D_1$ and $D_2$ are the only prime Weil divisors of $A$, so the group $Z^1(A)$ is freely generated by them. Our goal is equivalent to showing that the surjective homomorphism
    \[
        \varphi: Z^1(A) \to \ZZ/d\ZZ, \quad D = m_1 D_1 + m_2 D_2 \mapsto m_1 - m_2 \pmod{d}
    \]
    has as its kernel precisely the group of principal divisors. To show that every principal divisor is in the kernel of $\varphi$, it suffices to check the divisors of the ring elements $u, v,$ and $\pi$, as any rational function is a product of powers of these elements times a unit. This follows from the straightforward calculations:
    \[
        (u) = d \cdot D_1, \quad (v) = d \cdot D_2, \quad (\pi) = D_1 + D_2.
    \]
    Conversely, if $D = m_1 D_1 + m_2 D_2$ is in the kernel of $\varphi$, then $m_1 \equiv m_2 \pmod d$, so $m_1 = m_2 + kd$ for some $k \in \ZZ$. We can then write $D$ as
    \[
        D = (m_2 + kd)D_1 + m_2 D_2 = m_2(D_1 + D_2) + k(dD_1) = (\pi^{m_2}u^k),
    \]
    which shows that $D$ is principal. This completes the proof.
\end{proof}

\begin{remark}
    The above proof shows that $\Cl(A) \simeq \ZZ/d\ZZ$. The singularity of $\Spec A$ at the origin is an arithmetic analogue of a Du Val singularity of type $A_{d-1}$ (analytically isomorphic to the plane singularity $\CC[\![x,y,z]\!]/(xy-z^d)$). It is a standard result that the local divisor class group of an $A_n$ singularity is isomorphic to $\ZZ/(n+1)\ZZ$ (see, for example, \cite[Section 2.1]{brevik2012picard}; the original result is due to \cite{lipman1969rational}).
\end{remark}

\begin{remark}[Normalization of Coefficients]
    We may assume that the coefficients of the twisting divisor $D = \sum m_i \oo{X}_i$ are normalized, i.e., $m_i \geq 0$ for all $i$ and at least one coefficient is zero. This follows from the fact that $(\pi) = \sum_i \oo{X}_i$ is a principal Cartier divisor. We can add or subtract multiples of $(\pi)$ from $D$ without changing the isomorphism class of the line bundle $\OO_\X(D)$.
\end{remark}

Summarizing this section we have:

\begin{summary} \label{summary:contractions}
    A contraction map from the stable model $\X$ to a plane model $\X_0$ corresponds to a morphism
    \[ 
        \phi\colon \X \to \PP^2_{\OO_K}, 
    \]
    which is induced by a line bundle of the form
    \[ 
        \LL \simeq \omega_{\X/\OO_K}(D) \quad \text{where} \quad D = \sum_{i=1}^n m_i \oo{X}_i, 
    \]
    which must be generated by its global sections, and a choice of an $\OO_K$-basis for $H^0(\X, \LL)$. The coefficients $m_i$ may be chosen normalized ($m_i \geq 0$ with at least one zero) and must satisfy the congruence condition of Lemma \ref{lem:Cartier_Weil}. The plane model $\X_0$ is then obtained as the scheme-theoretic image of $\phi$.

    On the special fiber, the morphism
    \[ 
        \oo{\phi} := \phi|_{\oo{X}} \colon \oo{X} \to \PP^2_k 
    \]
    contracts $\oo{X}$ onto the special fiber of the plane model, $\X_{0,s}$, and is determined by the linear series $(\oo{\LL}, V)$, where $V$ is the 3-dimensional subspace of $H^0(\oo{X}, \oo{\LL})$ spanned by the restriction of the $\OO_K$-basis of $H^0(\X, \LL)$. By Lemma \ref{lem:global-generation_very-ample}, the base-point-freeness of this linear series is equivalent to the global generation of $\LL$ on $\X$.
\end{summary}

\subsection{\texorpdfstring{Twisting $1$-Tails}{Twisting 1-Tails}}

Having established the general theory for describing contractions of the stable model $\X$, we are now equipped to prove the implication $\ref{thm:main_theorem_i} \implies \ref{thm:main_theorem_ii}$ of our Main Theorem (\ref{thm:main_theorem}). We assume that the stable reduction $\oo{X}$ is non-hyperelliptic and, from this, construct a GIT-stable plane model $\X_0$ dominated by $\X$. This requires defining a morphism
\[
    \X \to \X_0 \subset \PP_{\OO_K}^2
\]
which, on the special fiber, contracts the $1$-tails of $\oo{X}$ to cusps and is an immersion elsewhere.

Recall from Lemma~\ref{lem:core_tail_lemma} that the stable reduction $\oo{X}$ decomposes into its core and $r$ $1$-tails:
\[
    \oo{X}_c, \quad \oo{X}_1, \dots, \oo{X}_r \quad \text{where } r \in \{0,1,2,3\}.
\]
Since we are interested in the case where $\oo{X}$ is non-hyperelliptic, its core $\oo{X}_c$ must be \emph{2-inseparable} (see Definition~\ref{def:comb_term}), as shown in Proposition \ref{prop:2-sep-hyp}.

We begin with the simplest case, where no $1$-tails are present ($r=0$). In this scenario, the stable reduction $\oo{X}$ is identical to its 2-inseparable core $\oo{X}_c$. In this case, the abstract stable model $\X$ can be realized directly as a plane model for which the notions of Deligne-Mumford stability and GIT-stability coincide.

\begin{proposition} \label{prop:main_thm_2insep}
    Assume the stable reduction $\oo{X}$ of $X$ is non-hyperelliptic and 2-inseparable. Then the stable model $\X$ can be realized as a GIT-stable plane model. Specifically, the relative dualizing sheaf $\omega_{\X/\OO_K}$ is very ample and induces a closed immersion
    \[
        \X \hookrightarrow \PP_{\OO_K}^2.
    \]
\end{proposition}

\begin{proof}
    Since $\oo{X}$ is a non-hyperelliptic and 2-inseparable stable curve, its dualizing sheaf $\omega_{\oo{X}/k}$ is very ample by Corollary~\ref{cor:canonical_embedding_general}. By Lemma \ref{lem:global-generation_very-ample}, it follows that $\omega_{\X/\OO_K}$ is very ample as well. By definition, this means the global sections of $\omega_{\X/\OO_K}$ induce the stated immersion, whose image is the desired GIT-stable plane model.
\end{proof}

\begin{remark}
    Implicitly this result was mentioned in \cite[Proposition A.1]{van2025reduction}.
\end{remark}

What happens if $r > 0$? In this case, the canonical map is not a morphism because the attachment nodes of the $1$-tails are base points of the linear system $|\omega_{\oo{X}/k}|$ (see Theorem~\ref{thm:catanese_base_points}). To resolve these base points, we must twist the dualizing sheaf.

For each attachment node $x_i$, let $d_i$ be its thickness in the stable model $\X$. We define the divisor
\[
    D := \sum_{i=1}^r d_i \oo{X}_i,
\]
which is a Cartier divisor by Lemma \ref{lem:Cartier_Weil}, since each tail is attached only to the core $\oo{X}_c$, which is not being twisted. This allows us to define the line bundle
\begin{equation} \label{eq:twisted_line_bundle}
    \LL := \omega_{\X/\OO_K}(D).
\end{equation}
This new line bundle induces a map $\phi\colon \X \to \PP^2_{\OO_K}$, which is a priori only rational. Our goal is to show that $\phi$ is in fact a morphism, as stated in the following theorem:

\begin{theorem} \label{thm:main_thm_2-insep_tails}
    Let $X$ be a smooth plane quartic over $K$, and assume its stable model $\X$ is defined over $\OO_K$. Assume further that the stable reduction $\oo{X}$ has a 2-inseparable core and $r \geq 1$ $1$-tails. Let $D := \sum_{i=1}^r d_i \oo{X}_i$ be the divisor supported on the $1$-tails, where $d_i$ is the thickness of the $i$-th attachment node, and let $\LL := \omega_{\X/\OO_K}(D)$ be the corresponding twisted line bundle.

    Then $\LL$ induces a contraction morphism $\phi\colon \X \to \PP^2_{\OO_K}$ whose scheme-theoretic image, $\X_0$, is a GIT-semistable plane model of $X$.
    Furthermore, the model $\X_0$ is GIT-stable if and only if the stable reduction $\oo{X}$ is non-hyperelliptic. The geometry of the special fiber $\X_{0,s} = \phi(\oo{X})$ is determined as follows:
    \begin{enumerate}[label=\StatementListLabel]
        \item If $\oo{X}$ is non-hyperelliptic, $\X_{0,s}$ is a GIT-stable plane quartic. The induced morphism $\oo{\phi}$ on the special fiber realizes the core $\oo{X}_c$ as the normalization of $\X_{0,s}$ at its cusps and contracts each $1$-tail to a distinct cuspidal singularity.
        \item If $\oo{X}$ is hyperelliptic, $\X_{0,s}$ is a double conic supported on a smooth conic $Q \subset \PP^2_k$. The induced morphism $\oo{\phi}$ on the special fiber contracts each $1$-tail to a distinct smooth point on $Q$ and maps the core $\oo{X}_c$ two-to-one onto $Q$.
    \end{enumerate}
\end{theorem}

For the remainder of this section, we work under the assumptions of Theorem~\ref{thm:main_thm_2-insep_tails} as we build up to its proof.

\subsubsection{\texorpdfstring{Analyzing the Line Bundle $\LL$}{Analyzing the Line Bundle L}}

The first step is to understand our line bundle $\LL$ from \eqref{eq:twisted_line_bundle} better.
On the generic fiber, it is just $\omega_{X/K}$. Let us analyze the restriction $\oo{\LL} = \omega_{\X/\OO_K}(D)|_{\oo{X}}$ to the special fiber.
For this, we look at the restrictions on the components of the special fiber, i.e., to the core and the $1$-tails:
\[
    \oo{\LL}_c := \oo{\LL}|_{\oo{X}_c}, \quad
    \oo{\LL}_i := \oo{\LL}|_{\oo{X}_i}, \quad \text{ for $i = 1, \dots, r$.}
\]

Let us first concentrate on the description away from the attachment nodes $\{x_i\}$.
Let $U_c := \oo{X}_c \setminus \{x_i\}$ be the open subset of the core component not containing any attachment nodes.
Since $U_c$ is disjoint from the support of the twisting divisor $D$, we have
\begin{equation} \label{eq:iso_core_U}
    (\oo{\LL}_c)|_{U_c} \simeq \omega_{\oo{X}_c/k}|_{U_c}.
\end{equation}

Now for $i \in \{1, \dots, r\}$ we consider $\oo{\LL}_i$ on the open set $U_i := \oo{X}_i \setminus \{x_i\}$.
In this case, however, the support of $D$ is not disjoint from $U_i$.
Consider the linearly equivalent divisor $\tilde{D}_i = D - (\pi^{d_i})$.
Its support is disjoint from $U_i$, and we have an isomorphism
\begin{equation} \label{eq:iso_mult_pi}
    \LL = \omega_{\X/\OO_K}(D) \xrightarrow{\sim} \omega_{\X/\OO_K}(\tilde{D}_i)
\end{equation}
given by multiplication by $\pi^{d_i}$. This induces an isomorphism
\begin{equation} \label{eq:iso_tail_U}
    (\oo{\LL}_i)|_{U_i} \simeq \omega_{\oo{X}_i/k}|_{U_i}.
\end{equation}

It remains to describe $\oo{\LL}$ at the attachment nodes $x_i$. For this, we may work in the completed local ring $\widehat{\OO}_{\X, x_i} \simeq \OO_K[\![u,v]\!]/(uv - \pi^{d_i})$. In this local coordinate system, we let the branch of the tail component $\oo{X}_i$ be defined by $u=0$ (with local parameter $v$), and the branch of the core component $\oo{X}_c$ be defined by $v=0$ (with local parameter $u$).
It is easy to check that the dualizing sheaf $\omega_{\X/\OO_K}$ is locally at $x_i$ generated by $u^{-1} du$ (which equals $-v^{-1} dv$). The divisor $D$ restricted to this neighborhood corresponds to $d_i \oo{X}_i$ and as we have seen in Lemma \ref{lem:Cartier_Weil}, this corresponds to the principal divisor $(u)$. Thus, the sheaf $\OO_{\X}(D)$ is locally generated by $u^{-1}$. We conclude that our line bundle $\LL = \omega_{\X/\OO_K} \otimes \OO_{\X}(D)$ is locally generated by
\[
    \omega_0 := (u^{-1} du) \otimes u^{-1} = \frac{du}{u^2} = -\frac{dv}{uv} = - \frac{dv}{\pi^{d_i}}.
\]
Thus, $\oo{\LL}$ is locally at $x_i$ generated by the restriction $\oo{\omega_0}$ of $\omega_0$ to $\oo{X}$.

The restriction $\oo{\omega_0}$ to $\oo{X}_c$, which is a local generator of $\oo{\LL}_c$ at $x_i$, is a differential with a pole of order $2$ at $x_i$. By abuse of notation, we simply write $\frac{du}{u^2}$ again. We conclude that we can extend the isomorphism \eqref{eq:iso_core_U} to
\[
    \oo{\LL}_c \simeq \omega_{\oo{X}_c/k}\left(\sum_{i=1}^r 2 x_i\right).
\]

The restriction of $\oo{\omega_0}$ to $\oo{X}_i$ is a local generator of $\oo{\LL}_i$ at $x_i$.
Under the isomorphism from \eqref{eq:iso_mult_pi}, the generator $\omega_0$ corresponds to $\tilde{\omega}_0 := -dv$. Its restriction to the special fiber is therefore a non-zero, regular differential on $\oo{X}_i$, which we simply write as $-dv$.
Thus, the isomorphism from \eqref{eq:iso_tail_U} extends to
\[
    \oo{\LL}_i \simeq \omega_{\oo{X}_i/k}, \quad \text{for $i = 1, \dots, r$.}
\]
Note that since each $1$-tail $\oo{X}_i$ has arithmetic genus 1, its dualizing sheaf is trivial, i.e., $\omega_{\oo{X}_i/k} \simeq \OO_{\oo{X}_i}$ (see \cite[Proposition 1.11]{catanese1982pluricanonical}). This holds regardless of whether the tail is smooth (an elliptic tail) or nodal (a pig tail).

Having now a clear understanding of how the restrictions $\oo{\LL}_c, \oo{\LL}_1, \dots, \oo{\LL}_r$ look, let us clarify how they ``glue together'' at the attachment nodes.
This can be described formally as follows: Let $\nu \colon \oo{X}' \to \oo{X}$ be the partial normalization at the attachment nodes $\{x_i\}$.
The normalization $\oo{X}'$ is the disjoint union of the components $\oo{X}_c, \oo{X}_1, \dots, \oo{X}_r$.
We have the standard short exact sequence relating the structure sheaf of a curve to that of its normalization:
\[
    0 \to \OO_{\oo{X}} \to \nu_* \OO_{\oo{X}'} \to \mathcal{Q} \to 0,
\]
where the cokernel $\mathcal{Q}$ is a skyscraper sheaf supported only at the attachment nodes $\{x_i\}$.
Since each $x_i$ is an ordinary node, the stalk of $\mathcal{Q}$ at such a point is a one-dimensional vector space over $k$.
Thus, $\mathcal{Q} \simeq \bigoplus_{i=1}^r k(x_i)$, where $k(x_i)$ is the residue field at the node.

Since $\oo{\LL}$ is a line bundle (and therefore flat), tensoring this exact sequence with $\oo{\LL}$ preserves exactness. Using the projection formula, $(\nu_* \OO_{\oo{X}'}) \otimes_{\OO_{\oo{X}}} \oo{\LL} \simeq \nu_*(\OO_{\oo{X}'} \otimes_{\OO_{\oo{X}'}} \nu^* \oo{\LL}) = \nu_*(\nu^* \oo{\LL})$, we deduce the exact ``gluing'' sequence for our line bundle:
\begin{equation} \label{eq:gluing_sequence}
    0 \to \oo{\LL} \to \nu_*(\nu^*\oo{\LL}) \to \oo{\LL} \otimes_{\OO_{\oo{X}}} \mathcal{Q} \to 0.
\end{equation}
The pullback sheaf $\nu^*\oo{\LL}$ lives on the normalization $\oo{X}'$; since $\oo{X}'$ is a disjoint union, $\nu^*\oo{\LL}$ is simply the collection of the restricted line bundles $(\oo{\LL}_c, \oo{\LL}_1, \dots, \oo{\LL}_r)$, each living on its own separate component. The pushforward sheaf $\nu_*(\nu^*\oo{\LL})$ brings this collection back to the original curve $\oo{X}$. A section of this sheaf is a tuple of sections, one for each component, with no compatibility required. We write this sheaf as the direct sum
\begin{equation} \label{eq:unglued_sheaf}
    \nu_*(\nu^*\oo{\LL}) = \oo{\LL}_c \oplus \bigoplus_{i=1}^r \oo{\LL}_i,
\end{equation}
with the individual sheaves understood to be extended by zero to the whole of $\oo{X}$.
The new skyscraper sheaf is $\oo{\LL} \otimes \mathcal{Q} \simeq \bigoplus_{i=1}^r \oo{\LL}|_{x_i}$, which is the direct sum of the fibers $\oo{\LL}|_{x_i} = \oo{\LL}_{x_i} \otimes k(x_i)$ of $\oo{\LL}$ at each node.

Since the restrictions come from the common sheaf $\LL$, we have an isomorphism
\[
    (\oo{\LL}_c)_{x_i} \xrightarrow{\sim} (\oo{\LL}_i)_{x_i},
\]
coming from the identification of the restrictions of a local generator of the stalk $(\oo{\LL})_{x_i}$.
In particular, this allows us to identify the fibers $\oo{\LL}_c|_{x_i} = \oo{\LL}_i|_{x_i}$.
Then in view of \eqref{eq:unglued_sheaf} and the exact sequence \eqref{eq:gluing_sequence}, we may see a section $s$ of $\oo{\LL}$ as a tuple of sections $(s_c, s_1, \dots, s_r)$ with the gluing condition
\[
    s_c(x_i) = s_i(x_i) \quad \text{for $i = 1, \dots, r$.}
\]
Here the notation $s(x_i)$ means the image of the section in the common fiber $\oo{\LL}|_{x_i}$. More concretely, representing sections locally at $x_i$ as
\begin{align*}
    s_c &= (a_{-2} + a_{-1} u + \cdots) \cdot \frac{du}{u^2}, \\
    s_i &= (b_0 + b_1 v + \cdots) \cdot dv,
\end{align*}
the identification of fibers is induced by the common local generator $\omega_0$. We identify the basis $\frac{du}{u^2}$ (the restriction of $\omega_0$ to the core) with $-dv$ (the image of $\omega_0$ under the isomorphism $\oo{\LL}_i \simeq \omega_{\oo{X}_i/k}$). Therefore, the gluing condition $s_c(x_i)=s_i(x_i)$ requires their leading coefficients to satisfy $a_{-2} = -b_0$, or equivalently, $a_{-2} + b_0 = 0$.

Let us now understand the global sections $H^0(\oo{X}, \oo{\LL})$. Since $\oo{\LL}_i \simeq \omega_{\oo{X}_i/k} \simeq \OO_{\oo{X}_i}$, we have $H^0(\oo{X}_i, \oo{\LL}_i) \simeq k$. Taking global sections of the exact sequence \eqref{eq:gluing_sequence} yields
\[
    0 \to H^0(\oo{X}, \oo{\LL}) \to H^0(\oo{X}_c, \oo{\LL}_c) \oplus k^r \to \bigoplus_{i=1}^r \oo{\LL}|_{x_i} \to 0.
\]
The last map is indeed surjective, since we can choose the zero section on the core and any combination of constant sections on the tails. This defines an $r$-dimensional vector space which surjects onto $\bigoplus_{i=1}^r \oo{\LL}|_{x_i}$. Thus, the restriction map induces an isomorphism:
\begin{equation} \label{eq:iso_to_h0_core}
    \iota\colon H^0(\oo{X}, \oo{\LL}) \xrightarrow{\sim} H^0(\oo{X}_c, \oo{\LL}_c) = H^0\left(\oo{X}_c, \omega_{\oo{X}_c/k}\left(\sum_{i=1}^r 2x_i\right)\right).
\end{equation}
Now we can use the Riemann-Roch theorem to compute the dimension:
\begin{equation} \label{eq:dimension_h0}
\begin{split}
    h^0(\oo{X}, \oo{\LL}) = h^0(\oo{X}_c, \oo{\LL}_c) &=  g(\oo{X}_c) + \deg\left(\sum_{i=1}^r 2x_i\right) - 1 \\
    &= (3-r) + 2r - 1 = r+2 = 3 + (r-1).
\end{split}
\end{equation}
Here we used that $\oo{\LL}_c$ is non-special, i.e., $h^1(\oo{X}_c, \oo{\LL}_c) = h^0(\oo{X}_c, \OO_{\oo{X}_c}(-\sum_{i=1}^r 2x_i)) = 0$, which holds since we assume $r \ge 1$.

\subsubsection{\texorpdfstring{Analyzing the Subspace $V$}{Analyzing the Subspace V}}

As discussed in Summary \ref{summary:contractions}, if $\LL$ induces a morphism, the restriction to the special fiber is determined by the linear series $(\oo{\LL}, V)$. Here, $V$ is the 3-dimensional subspace
\[
    V := \im(H^0(\X, \LL) \to H^0(\oo{X}, \oo{\LL})) \subset H^0(\oo{X}, \oo{\LL}),
\]
obtained by restricting the sections of $H^0(\X, \LL)$. We have the exact sequence (see \eqref{eq:ex_seq})
\begin{equation} \label{eq:ex_seq_v2}
    0 \to V \to H^0(\oo{X}, \oo{\LL}) \to H^1(\X, \LL)[\pi] \to 0.
\end{equation}
By our dimension calculations \eqref{eq:dimension_h0}, the $k$-vector space $H^1(\X, \LL)[\pi]$ has dimension $r-1$. In particular, $V$ is a proper subspace of $H^0(\oo{X}, \oo{\LL})$ if and only if $r \geq 2$. Our goal is now to describe this subspace more explicitly.

\begin{lemma} \label{lem:subspace_W}
    The vector subspace
    \[
        W := \{ s_c \in H^0(\oo{X}_c, \oo{\LL}_c) \ssep \Res_{x_i}(s_c) = 0 \quad \text{for } i = 1, \dots, r\}
    \]
    is three-dimensional.
\end{lemma}

\begin{proof}
    By the Residue Theorem on $\oo{X}_c$, we have $\sum_{i=1}^r \Res_{x_i}(s_c) = 0$. This implies that the $r$ linear conditions defining $W$ are not independent; there are at most $r-1$ independent conditions. The subspace $W$ therefore has a dimension of at least $(r+2) - (r-1) = 3$.

    To show that the dimension is exactly three, it remains to prove that the first $r-1$ conditions are linearly independent. We may assume $r \geq 2$. For a fixed $j \in \{1,\dots, r-1\}$, it suffices to show the existence of a section $s \in H^0(\oo{X}_c, \oo{\LL}_c)$ such that
    \[
        \Res_{x_j}(s) \neq 0 \quad \text{ and } \quad
        \Res_{x_i}(s) = 0 \quad \text{for all } i \in \{1,\dots,r-1\} \setminus \{j\}.
    \]
    Such a section can be found by choosing an element of the subspace $H^0(\oo{X}_c, \omega_{\oo{X}_c/k}(x_j + 2x_r))$ that is not contained in the smaller subspace $H^0(\oo{X}_c, \omega_{\oo{X}_c/k}(2x_r))$. The existence of such a section is guaranteed because, by the Riemann-Roch theorem and the fact that the core $\oo{X}_c$ has genus $3-r$, the dimensions of these spaces differ:
    \[
        h^0(\oo{X}_c, \omega_{\oo{X}_c/k}(x_j + 2x_r)) = 5-r > 4-r = h^0(\oo{X}_c, \omega_{\oo{X}_c/k}(2x_r)).
    \]
\end{proof}

Essentially, the vector space $V$ can be identified with $W$. More precisely:

\begin{proposition} \label{prop:subspace_V_W}
    The restriction isomorphism $\iota$ from \eqref{eq:iso_to_h0_core} induces an isomorphism $V \simeq W$.
\end{proposition}

\begin{proof}
    We aim to show that $\iota(V) = W$. Since both are three-dimensional vector spaces, it suffices to prove the inclusion $\iota(V) \subseteq W$.
    Let $s \in V$. By definition, there exists a lift $\tilde{s} \in H^0(\X, \LL)$. Let $s_c := \iota(s)$ be its restriction to the core. We must show that $\Res_{x_i}(s_c) = 0$ for all attachment nodes $x_i$.
    
    We can view $\tilde{s}$ as a meromorphic differential on $\X$. Let $y_i$ be the generic point of the $1$-tail $\oo{X}_i$, and for each $i$, let $y_{c,i}$ be the generic point of the irreducible component of the core $\oo{X}_c$ to which $\oo{X}_i$ attaches. The section $\tilde{s}$ is integral at $y_{c,i}$ because it is a global section of $\LL = \omega_{\X/\OO_K}(D)$ and the twisting divisor $D$ is not supported on the core. By Proposition \ref{prop:connection_special_residue}, showing $\Res_{x_i}(s_c) = 0$ is implied by showing $\Res_{(x_i, y_{c,i})}(\tilde{s}) = 0$.
    
    Since $\tilde{s}$ is a global section of $\LL$, its restriction to the generic fiber $X$ is a regular differential (a global section of $\omega_{X/K}$). By Proposition \ref{prop:connection_generic_residue}, this implies that the residue $\Res_\xi(\tilde{s})$ is zero for any horizontal flag $\xi$ on $\X$. At the attachment node $x_i$, the only vertical curves are the $i$-th tail and the corresponding component of the core. Thus, the reciprocity law around the point $x_i$ (Theorem \ref{thm:reciprocity_laws} \ref{thm:reciprocity_laws_i})) reduces to
    \[
        \Res_{(x_i, y_i)}(\tilde{s}) + \Res_{(x_i, y_{c,i})}(\tilde{s}) = 0.
    \]
    Hence, it suffices to show that $\Res_{(x_i, y_i)}(\tilde{s}) = 0$. By the reciprocity law along the vertical curve $\oo{X}_i = \overline{\{y_i\}}$ (Theorem \ref{thm:reciprocity_laws} \ref{thm:reciprocity_laws_ii}), we have
    \[
        \sum_{x \in \oo{X}_i} \Res_{(x, y_i)}(\tilde{s}) = 0.
    \]
    For any closed point $x \in \oo{X}_i$ other than the attachment node $x_i$, the only vertical curve passing through $x$ is $\oo{X}_i$ itself. By Lemma \ref{lem:reside_one_vertical} and the vanishing of residues for all horizontal flags, it follows that $\Res_{(x, y_i)}(\tilde{s}) = 0$. The sum above thus reduces to a single term, implying that $\Res_{(x_i, y_i)}(\tilde{s}) = 0$.

    This holds for all $i$, so $s_c \in W$, which completes the proof of the inclusion $\iota(V) \subseteq W$.
\end{proof}

\subsubsection{\texorpdfstring{Geometry of the Contraction Map $\phi$}{Geometry of the Contraction Map phi}}

The following lemma shows that to understand the morphism induced by $\LL$, it suffices to analyze the linear series on the core.

\begin{lemma} \label{lem:contraction_of_tails}
    The following conditions are equivalent:
    \begin{enumerate}[label=\EquivListLabel]
        \item \label{lem:contraction_of_tails_i} The line bundle $\LL$ is generated by its global sections.
        \item \label{lem:contraction_of_tails_ii} The linear series $(\oo{\LL}, V)$ on $\oo{X}$ is base-point-free.
        \item \label{lem:contraction_of_tails_iii} The linear series $(\oo{\LL}_c, W)$ on $\oo{X}_c$ is base-point-free.
    \end{enumerate}
    If these equivalent conditions hold, they induce morphisms
    \[
        \phi\colon \X \to \PP^2_{\OO_K}, \quad \oo{\phi} \colon \oo{X} \to \PP^2_k, \quad \text{and} \quad \oo{\phi}_c \colon \oo{X}_c \to \PP^2_k,
    \]
    respectively. Furthermore, the morphism $\oo{\phi}$ contracts each $1$-tail $\oo{X}_i$ to a single point.
\end{lemma}

\begin{proof}
    The implications \ref{lem:contraction_of_tails_i} $\implies$ \ref{lem:contraction_of_tails_ii} $\implies$ \ref{lem:contraction_of_tails_iii} hold by definition and restriction, noting that $W = \iota(V)$ by Proposition \ref{prop:subspace_V_W}.

    For the converse, assume \ref{lem:contraction_of_tails_iii} holds, i.e., the linear series $(\oo{\LL}_c, W)$ on the core $\oo{X}_c$ is base-point-free. Since $\iota(V)=W$, the linear series $(\oo{\LL}, V)$ is also base-point-free on $\oo{X}_c$.
    Now consider a $1$-tail $\oo{X}_i$ attached at the node $x_i$. We have previously established that $\oo{\LL}|_{\oo{X}_i} \simeq \OO_{\oo{X}_i}$. Since $\oo{X}_i$ is a connected, proper curve, $H^0(\oo{X}_i, \OO_{\oo{X}_i}) \simeq k$, so any section of $\oo{\LL}$ restricted to $\oo{X}_i$ is constant.
    Since $x_i \in \oo{X}_c$, there exists a section $s \in V$ such that $s(x_i) \neq 0$. The restriction $s|_{\oo{X}_i}$ is a constant function. As $s(x_i) \neq 0$, this constant is non-zero, implying that $s$ is non-vanishing everywhere on $\oo{X}_i$. Thus, $(\oo{\LL}, V)$ is base-point-free on the tails as well. This establishes \ref{lem:contraction_of_tails_ii}.
    The final implication \ref{lem:contraction_of_tails_ii} $\implies$ \ref{lem:contraction_of_tails_i} follows from Lemma \ref{lem:global-generation_very-ample}.

    Finally, assume these equivalent conditions hold. They induce the morphisms $\phi$ and $\oo{\phi}$ associated with a choice of basis for $H^0(\X, \LL)$ and $V$, respectively.
    We analyze the restriction of $\oo{\phi}$ to a $1$-tail $\oo{X}_i$. The tail $\oo{X}_i$ is a connected, projective closed subscheme of $\oo{X}$. We have $\oo{\LL}|_{\oo{X}_i} \simeq \OO_{\oo{X}_i}$. By \cite[Lemma 8.3.29]{liu}, this isomorphism implies that the image $\oo{\phi}(\oo{X}_i)$ is reduced to a single point. Therefore, $\oo{\phi}$ contracts each tail $\oo{X}_i$ to a point.
\end{proof}

Therefore, we only need to concentrate on the linear series $(\oo{\LL}_c, W)$, which determines the morphism $\oo{\phi}_c$. Since this is not always a complete linear system, the analysis is not straightforward; for a more direct approach, we refer to Appendix~\ref{app}. We will instead take a key intermediate step that identifies $(\oo{\LL}_c, W)$ with the complete linear system given by the dualizing sheaf of an associated curve. This is motivated by the fact that the conditions defining the subspace $W$---vanishing residues at the attachment points for differentials with double poles---precisely match the conditions for a differential to be regular at an $A_2$ (cuspidal) singularity (see Proposition~\ref{prop:local_conditions}). This step requires the following classical construction.

\begin{lemma} \label{lem:cusp_construction}
    Let $(C_c, \{x_1, \dots, x_r\})$ be the pointed core of a stable curve $C$ of genus $3$, where $C_c$ is 2-inseparable. Then there exists a unique curve $C'$ whose only singularities are ordinary nodes and cusps, and a finite morphism $\nu\colon C_c \to C'$ such that:
    \begin{enumerate}[label=\StatementListLabel]
        \item The restriction of $\nu$ to $C_c \setminus \{x_1, \dots, x_r\}$ is an isomorphism onto its image.
        \item The image of each point, $x_i' := \nu(x_i)$, is a cuspidal singularity (an $A_2$-singularity) on $C'$.
        \item All other singularities of $C'$ are nodes, corresponding bijectively to the nodes of $C_c$.
    \end{enumerate}
    The map $\nu$ is the partial normalization of $C'$ at its cusps.
\end{lemma}

\begin{proof}
The following is a standard construction (cf. \cite[Chapter 3.IV]{serre1988algebraic}).

We construct the curve $C'$ by defining a new structure sheaf, $\OO_{C'}$, on the underlying topological space of $C_c$. The stalks of $\OO_{C'}$ are defined as follows:
\begin{itemize}
    \item For any point $p \notin \{x_1, \dots, x_r\}$, the stalk is unchanged: $(\OO_{C'})_{p} := \OO_{C_c, p}$.
    \item For each $x \in \{x_1, \dots, x_r\}$, let $\mathfrak{m}_x$ be the maximal ideal of its local ring $\OO_{C_c, x}$. We define the new stalk to be $(\OO_{C'})_{x} := k + \mathfrak{m}_x^2$.
\end{itemize}
Since each $x \in \{x_1, \dots, x_r\}$ is a smooth point on the curve $C_c$, its local ring $\OO_{C_c, x}$ is a discrete valuation ring, and we may thus choose a local uniformizer $t$. With this, the formal completions of the local rings are $\widehat{\OO}_{C_c,x} \simeq k[\![t]\!]$ and $\widehat{\OO}_{C',x} \simeq k[\![t^2, t^3]\!]$. The latter is the completed local ring of a standard cuspidal singularity.

The morphism $\nu \colon C_c \to C'$ is the identity map on the underlying topological spaces. The morphism of structure sheaves $\nu^\#\colon \OO_{C'} \to \OO_{C_c}$ is induced by the natural inclusion of stalks $\OO_{C', x} \hookrightarrow \OO_{C_c, x}$ at every point $x$. The uniqueness of the curve $C'$ as required by the lemma is clear, as its structure sheaf has been uniquely determined at every stalk.
\end{proof}

This construction is ``compatible'' with the property hyperelliptic, i.e., we have:

\begin{lemma} \label{lem:hyp_vs_quasi-hyp_v2}
    Let $C$ be a stable curve of genus $3$ with a 2-inseparable core $(C_c, \allowbreak \{x_1, \dots, x_r\})$, and let $C'$ be the associated curve with cusps at the $x_i$ as constructed in Lemma \ref{lem:cusp_construction}. Then $C'$ is canonically positive, and the following are equivalent:
    \begin{enumerate}[label=\EquivListLabel]
        \item \label{lem:hyp_vs_quasi-hyp_v2_i} $C$ is hyperelliptic
        \item \label{lem:hyp_vs_quasi-hyp_v2_ii} $(C_c, \{x_1, \dots, x_r\})$ is hyperelliptic
        \item \label{lem:hyp_vs_quasi-hyp_v2_iii} $C'$ is quasi-hyperelliptic
        \item \label{lem:hyp_vs_quasi-hyp_v2_iv} $C'$ is honestly hyperelliptic
    \end{enumerate}
\end{lemma}

\begin{proof}
    The pointed core $(C_c, \{x_1, \dots, x_r\})$ is stable because it is a component in the decomposition of the stable curve $C$ (see Lemma~\ref{lem:stable_decomposition}). By stability, the log-canonical sheaf $\omega_{C_c/k}(\sum x_i)$ is ample (Proposition~\ref{prop:stable_via_canonical}). Thus, for any irreducible component $Z$ of $C_c$ with $r_Z$ marked points, we have $\deg(\omega_{C_c/k}|_Z) + r_Z > 0$. Let $Z'$ be the corresponding component of $C'$. The relationship between the dualizing sheaves is given by
    \[
        \deg(\omega_{C'/k}|_{Z'}) = \deg(\omega_{C_c/k}|_Z) + 2r_Z = (\deg(\omega_{C_c/k}|_Z) + r_Z) + r_Z.
    \]
    Since the parenthesized term is positive by stability, we conclude that $\deg(\omega_{C'/k}|_{Z'}) > r_Z \ge 0$. Therefore, $C'$ is canonically positive.

    We now prove the equivalences. The equivalence $\ref{lem:hyp_vs_quasi-hyp_v2_i} \iff \ref{lem:hyp_vs_quasi-hyp_v2_ii}$ follows from the decomposition Lemma~\ref{lem:decomp_sep_node} and Lemma~\ref{lem:always_hyperelliptic}.
    Since $C_c$ is 2-inseparable, the curve $C'$ is also 2-inseparable, as the construction does not alter the nodal structure connecting its components. For a Gorenstein curve with only nodes and cusps, 2-inseparability is equivalent to being very strongly connected. Thus, the equivalence $\ref{lem:hyp_vs_quasi-hyp_v2_iii} \iff \ref{lem:hyp_vs_quasi-hyp_v2_iv}$ follows from Proposition~\ref{prop:hyp_vs_quasi-hyp}.
    It remains to show $\ref{lem:hyp_vs_quasi-hyp_v2_ii} \iff \ref{lem:hyp_vs_quasi-hyp_v2_iv}$. Let $\nu\colon C_c \to C'$ be the normalization map from Lemma~\ref{lem:cusp_construction}.

    $(\ref{lem:hyp_vs_quasi-hyp_v2_ii} \implies \ref{lem:hyp_vs_quasi-hyp_v2_iv})$: Assume $(C_c, \{x_i\})$ is hyperelliptic with involution $\sigma$.
    Since $C_c$ is 2-inseparable with $r \geq 1$ marked points, we conclude with Lemma \ref{lem:unique_2_inse} that $C_c$ is irreducible. The quotient $T := C_c/\langle\sigma\rangle$ is isomorphic to $\PP^1_k$. Let $\pi\colon C_c \to T$ be the quotient map.
    We define the associated degree-two map $\pi' = F \circ \pi\colon C_c \to T'$ as in Remark \ref{rem:degree2_map}.
    In our case, this means that if $\Char k = 2$ and $r = 3$, then $F = F_{T/k}\colon T \to T^{(2)} =: T'$ is the relative Frobenius morphism (and $\sigma$ is the identity in this case), while in all other cases $T = T'$ and $F$ is the identity map (see Remark~\ref{rem:char2_dependency}).
    The map $\pi'$ is a finite morphism of degree 2. We now show that $\pi'$ factors through $\nu$.
    
    The factorization of $\pi'$ is a local question that must be analyzed at each point $x_i$. Let $x=x_i$ and $x'=\nu(x)$. We consider the relevant local rings: $B := \OO_{C_c, x}$ at the smooth point $x$ (with maximal ideal $\mf{m}_{x}$); the subring $B' := \OO_{C', x'} = k + \mf{m}_{x}^2$ corresponding to the cusp $x'$ (see Lemma \ref{lem:cusp_construction}); and $A := \OO_{T', \pi'(x)}$ on the target curve. For the map $\pi'$ to factor through $\nu$, the image of the local homomorphism $\pi'^\#_x\colon A \to B$ must be contained in $B'$. To verify this, we pass to the completions of these rings. This step is justified because $B$ is a Noetherian local ring and the canonical map $B \to \widehat{B}$ to its completion is faithfully flat; consequently, the inclusion holds if and only if it holds for the completed rings. Let $t$ be a local uniformizer at $x$. Since $x$ is a smooth point and $k$ is algebraically closed, we have $\widehat{B} \simeq k[\![t]\!]$. The completion of the subring $B'$ is then $\widehat{B'} \simeq k[\![t^2, t^3]\!]$, which is the subring of formal power series with a vanishing linear term. The involution $\sigma$ extends uniquely to an automorphism of $\widehat{B}$.

    \begin{case}[$\sigma = \id$]
        This case occurs precisely when $\Char k=2$ and $r=3$. The core is then $C_c \simeq \PP^1_k$, and the involution must be the identity as it fixes the three attachment points. The map $\pi'$ is therefore the relative Frobenius morphism $F_{C_c/k}$. On the level of completions, the induced homomorphism $\widehat{\pi'^\#}\colon \widehat{A} \to \widehat{B}$ sends the formal parameter of $\widehat{A}$ to $t^2$, so its image is $k[\![t^2]\!]$, which is clearly a subring of $\widehat{B'} \simeq k[\![t^2, t^3]\!]$.
    \end{case}

    \begin{case}[$\sigma \neq \id$]
        Here, $\pi' = \pi$ is the quotient map by a non-trivial involution $\sigma \neq \id$. The image of $\widehat{A}$ in $\widehat{B}$ under $\widehat{\pi^\#_x}$ is the ring of invariants $\widehat{B}^{\sigma}$. We must show this ring is a subring of the cusp's completed ring, i.e., that $\widehat{B}^{\sigma} \subset \widehat{B'}$.

        \begin{subcase}[$\Char k \neq 2$]
            The involution $\sigma$ induces a linear action on the cotangent space $\mf{m}_x/\mf{m}_x^2$, corresponding to multiplication by a scalar $\lambda \in k$ with $\lambda^2=1$. If this action were trivial ($\lambda=1$), $\sigma$ would act as the identity on the entire formal neighborhood of $x$. For an automorphism of a projective, irreducible curve, this implies that $\sigma$ must be the global identity map, which contradicts our assumption that $\sigma$ is non-trivial. Therefore, the action on the cotangent space must be non-trivial, forcing $\lambda=-1$.
            
            Since $\sigma$ acts by $-1$ on the cotangent space, we have $\sigma(t) = -t + h(t)$ for some $h(t) \in (t^2)$. By replacing $t$ with the new formal parameter $t' := \frac{1}{2}(t - \sigma(t))$ if necessary (note that since $h(t) \in (t^2)$, the linear part of $t'$ is $t$, making it a valid uniformizer), we may assume the action is linearized, i.e., that $\sigma(t)=-t$. An element $g(t) = \sum c_i t^i \in \widehat{B}$ is therefore invariant under $\sigma$ if and only if $g(t) = g(-t)$, which requires all odd coefficients to be zero. The ring of invariants is thus $\widehat{B}^{\sigma} = k[\![t^2]\!]$, which is a subring of $\widehat{B'} \simeq k[\![t^2, t^3]\!]$.
        \end{subcase}

        \begin{subcase}[$\Char k = 2$]
            The induced action of $\sigma$ on the cotangent space must be the identity, so the action on $k[\![t]\!]$ is of the form $\sigma(t) = t + h(t)$ with $h(t) \in (t^2)$. Since $\sigma \neq \id$, we have $h(t) \neq 0$. Let $g(t) = \sum_{i \ge 0} c_i t^i$ be an invariant element in $\widehat{B}^{\sigma}$. The invariance condition $g(\sigma(t)) - g(t) = 0$ expands to:
            \[
                c_1 h(t) + c_2\left((t+h(t))^2 - t^2\right) + \dots = 0.
            \]
            In characteristic $2$, this becomes $c_1 h(t) + c_2 h(t)^2 + \dots = 0$. Since $k[\![t]\!]$ is an integral domain and $h(t) \neq 0$, we can divide by $h(t)$ to get $c_1 + c_2 h(t) + \dots = 0$. As $h(t) \in (t^2)$, the constant term of this power series is $c_1$, which implies $c_1=0$, so any invariant element has a vanishing linear term. Therefore, the ring of invariants is contained in the completed ring of the cusp, $\widehat{B}^{\sigma} \subset \widehat{B'} \simeq k[\![t^2, t^3]\!]$.
        \end{subcase}
    \end{case}

    In all cases, $\pi'$ factors locally through $\nu$, so it factors globally. This yields a factorization $\pi' = f \circ \nu$ for some morphism $f\colon C' \to T'$. Since $\pi'\colon C_c \to T'$ is a finite morphism of degree $2$ and the normalization map $\nu\colon C_c \to C'$ is finite and birational, the induced map $f$ must also be a finite morphism of degree 2. As $T' \simeq \PP^1_k$, this proves that $C'$ is honestly hyperelliptic.
    
    $(\ref{lem:hyp_vs_quasi-hyp_v2_iv} \implies \ref{lem:hyp_vs_quasi-hyp_v2_ii})$: Assume $C'$ is honestly hyperelliptic, admitting a finite degree-two morphism $f'\colon C' \to \PP^1_k$. Since $C'$ is very strongly connected and has $r \geq 1$ cusps by construction, it follows from Proposition \ref{prop:hyp_vs_quasi-hyp} that $C'$ is irreducible.

    We first handle the case where the morphism $f'$ is not separable. For a degree-two morphism, this implies that $f'$ must be purely inseparable and that $\Char k = 2$. The argument from Corollary \ref{cor:all_the_same} shows that $C'$ cannot have any nodes, so the map $\nu\colon C_c \to C'$ is its full normalization. The composition $f' \circ \nu\colon C_c \to \PP^1_k$ is then a purely inseparable morphism between smooth curves, which forces $g(C_c) = 0$ by \cite[Proposition 7.4.21]{liu}. Thus, the core is rational ($C_c \simeq \PP^1_k$), and $\sigma = \id_{C_c}$ is a valid hyperelliptic involution for $(C_c, \{x_i\})$. Note that since $C'$ has genus $3$, this case occurs only when $r=3$.
    
    Now, assume $f'$ is separable. It is then a Galois cover induced by a unique involution $\sigma'$ of order $2$ on $C'$. We show that this involution must fix the cusps. Since $C'$ is very strongly connected and honestly hyperelliptic, its canonical map $\Phi_{C'}$ factors through the map $f'$ (Proposition \ref{prop:hyp_vs_quasi-hyp}), which implies that $\Phi_{C'}(p) = \Phi_{C'}(\sigma'(p))$ for any point $p \in C'$. However, $C'$ is also 2-inseparable, so its canonical map is injective on the singular locus (Theorem \ref{thm:necessary_injective}). As each cusp $x'_i$ is a singular point, this injectivity forces the equality $x'_i = \sigma'(x'_i)$.

    Since $\nu\colon C_c \to C'$ is the normalization map at the cusps and $\sigma'$ fixes these points, the involution $\sigma'$ lifts uniquely to an involution $\sigma$ on $C_c$. Since $\sigma'$ fixes the unibranch cusps $x'_i$, the lifted involution $\sigma$ must fix their unique preimages $x_i$. The map $\nu$ is equivariant with respect to these involutions and therefore descends to an isomorphism on the quotient curves, $T := C_c/\langle\sigma\rangle \simeq C'/\langle\sigma'\rangle \simeq \PP^1_k$. This establishes the existence of a hyperelliptic involution $\sigma$ on the pointed curve $(C_c, \{x_i\})$, which completes the proof.
\end{proof}

With this intermediate step we can finally prove together with the results of Catanese one direction of our main result.

\begin{proof}[of Theorem \ref{thm:main_thm_2-insep_tails}]
    Let $C := \oo{X}$ be the stable reduction, $C_c := \oo{X}_c$ its core, and $\{x_1, \dots, x_r\}$ the attachment points of the $1$-tails $C_i$. Let $C'$ be the associated curve with cusps constructed in Lemma \ref{lem:cusp_construction}, and $\nu\colon C_c \to C'$ the partial normalization map. Let $x'_i = \nu(x_i)$ be the cusps on $C'$. Note that $g(C')=3$.

    \begin{claimstar}
        We have an isomorphism of sheaves $\oo{\LL}_c \simeq \omega_{C_c/k}(\sum 2x_i) \simeq \nu^* \omega_{C'/k}$, and the pullback map $\nu^*$ induces an isomorphism $H^0(C', \omega_{C'/k}) \simeq W$.
    \end{claimstar}
    
    \begin{proof}
        We first establish the isomorphism of invertible sheaves $\oo{\LL}_c \simeq \nu^* \omega_{C'/k}$ by comparing them locally on $C_c$.
    
        Let $U = C_c \setminus \{x_1, \dots, x_r\}$. On this open set, the restriction $\nu|_U\colon U \to \nu(U)$ is an isomorphism, and the divisor $\sum 2x_i$ is trivial. Thus, we have the required isomorphism on $U$:
        \[
            \oo{\LL}_c|_U \simeq \omega_{C_c/k}|_U \simeq (\nu|_U)^*(\omega_{C'/k}|_{\nu(U)}) = (\nu^* \omega_{C'/k})|_U.
        \]
    
        Now consider a point $x=x_i$. Let $t$ be a local uniformizer at $x$. The sheaf $\oo{\LL}_c = \omega_{C_c/k}(2x)$ is locally generated by $\omega_0 := t^{-2}dt$.
        We analyze the dualizing sheaf $\omega_{C'/k}$ at the cusp $x'=\nu(x)$. Since $C'$ is Gorenstein (having only $A_n$ singularities), $\omega_{C'/k}$ is invertible. By Rosenlicht's Theorem, $\omega_{C'/k}$ is the sheaf of regular differentials. Locally at $x'$, we identify sections of $\omega_{C'/k}$ with meromorphic differentials $\eta$ on the (partial) normalization $C_c$ that satisfy the regularity conditions for the $A_2$ singularity. By Proposition \ref{prop:local_conditions}, $\eta$ is regular at $x'$ if it has a pole of order at most $2$ at $x$ and $\Res_x(\eta)=0$.
        The differential $\omega_0 = t^{-2}dt$ satisfies these conditions. Since $\omega_{C'/k}$ is invertible, $\omega_0$ must be a local generator for $\omega_{C'/k}$ at $x'$.
        The pullback $\nu^*\omega_{C'/k}$ is locally generated by $\nu^*(\omega_0)$, which is simply $\omega_0$ viewed on $C_c$. Since the local generators agree everywhere, we conclude $\oo{\LL}_c \simeq \nu^* \omega_{C'/k}$.
    
        Now we analyze the pullback map on global sections:
        \[
            \nu^*\colon H^0(C', \omega_{C'/k}) \to H^0(C_c, \nu^* \omega_{C'/k}) \simeq H^0(C_c, \oo{\LL}_c).
        \]
        The map $\nu^*$ is induced by the unit of adjunction $\eta\colon \omega_{C'/k} \to \nu_*(\nu^*\omega_{C'/k})$. Since $\nu\colon C_c \to C'$ is a partial normalization between reduced curves, the induced map on structure sheaves $\OO_{C'} \to \nu_*\OO_{C_c}$ is injective. As $\omega_{C'/k}$ is invertible (hence locally free), the map $\eta$, which is obtained locally by tensoring the injective map on structure sheaves with $\omega_{C'/k}$, is also injective. Consequently, the induced map on global sections $\nu^* = H^0(C', \eta)$ is injective.
    
        We identify the image of $\nu^*$. A section $s_c \in H^0(C_c, \oo{\LL}_c)$ is in the image of $\nu^*$ if and only if $s_c$, viewed as a global section of $\nu_*(\nu^*\omega_{C'/k})$ via the isomorphism established above, belongs to the subsheaf $\omega_{C'/k}$ (i.e., it is in the image of $\eta$). By Rosenlicht's Theorem, this is equivalent to $s_c$ being a regular differential everywhere on $C'$.
    
        We verify this condition locally. Away from the cusps $x'_i$, $\nu$ is an isomorphism and $s_c$ is regular on $C_c$ (as $\oo{\LL}_c \simeq \omega_{C_c/k}$ there), so it is regular on $C'$. At the cusps $x'_i$, the condition for $s_c$ to be regular on $C'$ is precisely $\Res_{x_i}(s_c) = 0$, by Proposition \ref{prop:local_conditions}.
        Therefore, the image of $\nu^*$ is exactly the subspace $W = \{ s_c \in H^0(C_c, \oo{\LL}_c) \ssep \Res_{x_i}(s_c) = 0 \text{ for all } i \}$, which establishes the isomorphism.
    \end{proof}

    By the preceding claim, the linear series $(\oo{\LL}_c, W)$ is base-point-free, as it is identified with the canonical system $|\omega_{C'/k}|$ of the canonically positive, inseparable Gorenstein curve $C'$, which is itself base-point-free by Catanese's Theorem \ref{thm:catanese_base_points}. It follows from Lemma \ref{lem:contraction_of_tails} that the line bundle $\LL$ is generated by its global sections, ensuring that the map $\phi\colon \X \to \PP^2_{\OO_K}$ is a morphism. We define its scheme-theoretic image as $\X_0$, which is a plane model of $X$.

    The claim also implies that the morphism on the core, $\oo{\phi}_c$, factors through the canonical map of $C'$:
    \[
        C_c \xrightarrow{\nu} C' \xrightarrow{\Phi_{C'}} \PP^2_k.
    \]
    Since the $1$-tails are contracted to points (Lemma \ref{lem:contraction_of_tails}), the special fiber of this plane model is simply the image of the core, namely $\X_{0,s} = \Phi_{C'}(C')$.

    It remains to analyze the properties of the canonical map $\Phi_{C'}$, which we do using the results of Catanese. We rely on the key equivalence from Lemma \ref{lem:hyp_vs_quasi-hyp_v2}: the stable curve $C$ is hyperelliptic if and only if $C'$ is honestly hyperelliptic. The necessary condition to apply Catanese's theorems is that $C'$ is very strongly connected, which follows from the 2-inseparability of its core, $C_c$.

    \begin{prooflist}
        \item If $C$ is non-hyperelliptic, then $C'$ is not quasi-hyperelliptic. Since $C'$ is very strongly connected and has only $A_n$ singularities (nodes $A_1$ and cusps $A_2$), Theorem \ref{thm:canonical_embedding_catanese} applies, showing that the canonical map $\Phi_{C'}$ is a closed immersion. The image $\X_{0,s}$ is therefore isomorphic to $C'$ and is embedded as a plane quartic. Its singularities are the nodes inherited from $C_c$ and the $r$ cusps $x'_i$, so $\X_{0,s}$ is GIT-stable. The morphism on the core, $\oo{\phi}_c = \Phi_{C'} \circ \nu$, realizes $C_c$ as the normalization of $\X_{0,s}$ at its cusps. The $1$-tails are contracted to these cusps, which are distinct points because $\Phi_{C'}$ is an embedding.
        \item If $C$ is hyperelliptic, then $C'$ is honestly hyperelliptic. By Proposition \ref{prop:hyp_vs_quasi-hyp}, the canonical map $\Phi_{C'}$ is a degree-two morphism onto a rational normal curve in $\PP^2_k$, i.e., a smooth conic $Q$. Since the total degree of the model $\X_0$ is $4$, the scheme-theoretic image $\X_{0,s}$ must be the double conic $2Q$, which is GIT-semistable but not stable. In this case, the morphism on the core $\oo{\phi}_c$ maps $C_c$ two-to-one onto the conic $Q$. The $1$-tails are contracted to points on this conic, which are distinct because $\Phi_{C'}$ is injective on the singular locus of $C'$ (Theorem \ref{thm:necessary_injective}).
    \end{prooflist}
\end{proof}

\subsection{From GIT-Stable to Stable Model}

We now prove the converse direction of the Main Theorem \ref{thm:main_theorem}, i.e., $\ref{thm:main_theorem_ii} \implies \ref{thm:main_theorem_i}$. We show that the existence of a GIT-stable model implies that the stable reduction is non-hyperelliptic and that the stable model is the minimal semistable model dominating the GIT-stable model. The proof requires two preliminary lemmas: the first concerns the geometry of a GIT-stable quartic, and the second addresses the existence of a minimal model dominating a given normal model.

\begin{lemma} \label{lem:structure_GIT_stable_quartic}
Let $C \subset \PP^2_k$ be a GIT-stable plane quartic over an algebraically closed field $k$. By definition, $C$ is reduced and has only ordinary nodes ($A_1$ singularities) and ordinary cusps ($A_2$ singularities). Let $r$ be the number of cusps of $C$. Let $\nu\colon \tilde{C} \to C$ be the partial normalization of $C$ at the cusps, and let $S$ be the set of preimages of the cusps.
Then the following properties hold:
\begin{enumerate}[label=\StatementListLabel]
    \item \label{lem:structure_GIT_stable_quartic_i} Either $C$ is irreducible, in which case $r \leq 3$; or $C$ is the union of a line and a cubic, in which case $r \leq 1$ and any cusp must lie on the cubic component; or $C$ is a union of two, three, or four components of degree at most $2$ (i.e., lines and conics), in which case $r=0$.
    \item \label{lem:structure_GIT_stable_quartic_ii} $C$ and $\tilde{C}$ are 2-inseparable.
    \item \label{lem:structure_GIT_stable_quartic_iii} $C$ is canonically positive and not quasi-hyperelliptic.
    \item \label{lem:structure_GIT_stable_quartic_iv} $(\tilde{C},S)$ is a stable pointed curve of genus $3-r$, which is not hyperelliptic in the sense of Definition \ref{def:hyperelliptic}.
\end{enumerate}
\end{lemma}

\begin{proof}
\begin{prooflist}
    \item The arithmetic genus of any plane quartic is $g(C) = 3$. Note that a cusp is a unibranch singularity with a $\delta$-invariant of 1, and thus it must lie on a single irreducible component.

    The partial normalization $\tilde{C}$ of $C$ at its $r$ cusps is a connected curve, so its arithmetic genus $g(\tilde{C})$ is non-negative. The genus formula gives $g(C) = g(\tilde{C}) + r$, which becomes $3 = g(\tilde{C}) + r$. This immediately implies the general bound $r \leq 3$.

    For the reducible case, sharper bounds apply. Since a cusp must be a self-singularity of an irreducible component, we analyze the components by degree. A component of degree $1$ (a line) or $2$ (a conic) has an arithmetic genus of $0$ and thus cannot have a cusp, as this would imply a negative geometric genus. This observation covers all reducible configurations that do not involve a cubic component: a union of two conics (degrees $2,2$), a conic and two lines (degrees $2,1,1$), or four lines (degrees $1,1,1,1$). In these configurations, no component can have a cusp, so we must have $r=0$. The only remaining reducible possibility is the union of a line and a cubic. Here, only the cubic component can have a cusp; since a plane cubic has arithmetic genus $1$ and therefore at most one cusp, it follows that for this configuration, $r \leq 1$.
    \item We prove that $C$ is 2-inseparable by showing that any decomposition of $C$ leads to a contradiction with Bézout's theorem.

    Suppose for contradiction that $C$ is not 2-inseparable. Then it must have a separating set of nodes $S$ (where $|S|=1$ for a separating node or $|S|=2$ for a separating pair) that decomposes $C$ into two subcurves, $C_1$ and $C_2$. By Bézout's theorem, the product of the degrees of the components equals the total intersection multiplicity:
    \[
        \deg(C_1)\deg(C_2) = \sum_{p \in S} I_p(C_1, C_2) = |S|.
    \]
    since each intersection is a node with multiplicity 1.
    If $|S|=1$, then $\deg(C_1)\deg(C_2)=1$, forcing the component degrees to be $1$ and $1$, and the total degree of $C$ to be $1+1=2$. If $|S|=2$, then $\deg(C_1)\deg(C_2)=2$, forcing the component degrees to be $1$ and $2$, and the total degree of $C$ to be $1+2=3$.
    In both scenarios, the resulting total degree is less than $4$, which contradicts the fact that $C$ is a plane quartic. Therefore, $C$ can have neither separating nodes nor separating pairs and must be 2-inseparable.

    Finally, the partial normalization $\nu\colon \tilde{C} \to C$ resolves the cusps. Since cusps are unibranch singularities, this process does not alter the number of irreducible components or the nodal structure connecting them. The property of being 2-inseparable depends only on this connecting structure and is therefore preserved. Consequently, $\tilde{C}$ is also 2-inseparable.
    \item As a plane quartic, $C$ is a Gorenstein curve. By the adjunction formula for a plane curve of degree $d=4$, its dualizing sheaf is isomorphic to the hyperplane bundle: $\omega_{C/k} \simeq \OO_C(1)$. The line bundle $\OO_C(1)$ is very ample by definition, as it provides the embedding of $C$ into $\PP^2_k$.

    The very ampleness of $\omega_{C/k}$ has two direct consequences. First, a very ample sheaf is also ample, so $C$ is canonically positive. Second, since $\omega_{C/k}$ is very ample, the canonical map $\Phi_C$ is a closed immersion and, in particular, birational onto its image. By Proposition \ref{prop:quasi-hyperelliptic}, it therefore follows that $C$ is not quasi-hyperelliptic.
    \item As the partial normalization of $C$ at its cusps, $\tilde{C}$ has only nodes as singularities and is thus semistable. As established in the proof of \ref{lem:structure_GIT_stable_quartic_i}, its arithmetic genus is $g(\tilde{C}) = 3-r$, and the set of preimages $S$ has cardinality $|S|=r$.

    To verify stability, we use Proposition \ref{prop:stable_via_canonical} and show that the degree of the log-canonical sheaf, $\deg(\omega_{\tilde{C}/k}(S)|_{\tilde{Z}})$, is positive for every irreducible component $\tilde{Z}$ of $\tilde{C}$ (where $S$ is identified with the divisor $\sum_{p \in S} p$).
    \begin{claimstar}
        Let $Z$ be the irreducible component of $C$ corresponding to $\tilde{Z}$. Then
        \[
            \deg(\omega_{\tilde{C}/k}(S)|_{\tilde{Z}}) = \deg(Z) - r_Z,
        \]
        where $r_Z$ is the number of cusps on $Z$.
    \end{claimstar}
    \begin{proof}
    The normalization $\nu|_{\tilde{Z}}\colon \tilde{Z} \to Z$ resolves the $r_Z$ cusps on $Z$, so the genus formula gives $g(Z) = g(\tilde{Z}) + r_Z$. The degree of the dualizing sheaf on a curve is determined by its arithmetic genus and its connecting nodes to other components. Since the connections between components are unchanged by the partial normalization at cusps, the difference in the degrees of the restricted dualizing sheaves is simply:
    \[ \deg(\omega_{\tilde{C}/k}|_{\tilde{Z}}) - \deg(\omega_{C/k}|_Z) = (2g(\tilde{Z})-2) - (2g(Z)-2) = 2g(\tilde{Z}) - 2g(Z) = -2r_Z. \]
    The marked points on $\tilde{Z}$ are precisely the $r_Z$ preimages of the cusps on $Z$. Thus,
    \[ \deg(\omega_{\tilde{C}/k}(S)|_{\tilde{Z}}) = \deg(\omega_{\tilde{C}/k}|_{\tilde{Z}}) + |S \cap \tilde{Z}| = \deg(\omega_{C/k}|_Z) - 2r_Z + r_Z = \deg(\omega_{C/k}|_Z) - r_Z. \]
    As $C$ is a plane quartic, $\omega_{C/k} \simeq \OO_C(1)$, so $\deg(\omega_{C/k}|_Z) = \deg(Z)$. This proves the claim.
    \end{proof}
    For any irreducible component $Z$ of $C$, a direct check of the cases from \ref{lem:structure_GIT_stable_quartic_i} shows that $\deg(Z) - r_Z > 0$. Thus, $(\tilde{C},S)$ is stable.

    Finally, by \ref{lem:structure_GIT_stable_quartic_iii}, $C$ is not quasi-hyperelliptic. If $r = 0$, then $C$ and $\tilde{C}$ coincide, and it follows from the equivalence in Corollary~\ref{cor:all_the_same} that $\tilde{C}$ is not hyperelliptic. Otherwise, the curve $C$ and the pointed curve $(\tilde{C}, S)$ fit the setup of Lemma \ref{lem:hyp_vs_quasi-hyp_v2}, which states that $C$ is quasi-hyperelliptic if and only if $(\tilde{C}, S)$ is hyperelliptic. Therefore, we conclude in either case that $(\tilde{C}, S)$ is not hyperelliptic.
\end{prooflist}
\end{proof}

\begin{remark}
    The combinatorial types of the stable pointed curves $(\tilde{C}, S)$ occurring in Lemma \ref{lem:structure_GIT_stable_quartic} correspond precisely to the combinatorial types of the pointed, 2-inseparable cores of stable curves of genus 3. There are 16 such types of pointed cores.

    This number is smaller than the 27 types of genus $3$ stable curves with 2-inseparable cores listed in Proposition \ref{prop:classification}. The reason for this discrepancy is that the process of replacing a $1$-tail with a marked point on the core does not distinguish between an elliptic tail and a pig tail; both result in the same combinatorial type for the pointed core. In view of our Main Theorem, this correspondence is to be expected.
\end{remark}

\begin{lemma}[Stable Hull] \label{lem:stable_hull}
    Let $X$ be a smooth, projective, geometrically connected curve over $K$, and let $\Y$ be a normal model of $X$ over $\OO_K$. Assume that $X$ has semistable reduction over $K$ and that the special fiber $\Y_s$ is geometrically reduced\footnote{Since the residue field $k$ is assumed to be algebraically closed, the notions of ``reduced'' and ``geometrically reduced'' coincide.}.
    
    Then the stable hull of $\Y$ exists and is unique. This stable hull, denoted $\X$, is the unique minimal semistable model dominating $\Y$ via a birational morphism $\pi\colon \X \to \Y$ (i.e., any other semistable model dominating $\Y$ also dominates $\X$).
    
    Let $E_\pi$ be the exceptional locus of $\pi$ and $C_\pi$ be the strict transform of $\Y_s$ in $\X$. We have:
    
    \begin{enumerate}[label=\StatementListLabel]
        \item \label{lem:stable_hull_i} (Relative Minimality) For every irreducible component $Z$ of the exceptional locus $E_\pi$, we have 
        \[\deg \omega_{\X/\OO_K}|_Z > 0.\]
        
        \item \label{lem:stable_hull_ii} (Construction via Desingularization) Let $\rho\colon \Y_\text{reg} \to \Y$ be the minimal desingularization of $\Y$. Then $\Y_\text{reg}$ is a semistable model and, by the minimality of $\X$, it dominates $\X$ via a morphism $\Y_\text{reg} \to \X$. This morphism is precisely the contraction of all $(-2)$-curves\footnote{A $(-2)$-curve on a model $\mathcal{W}$ is an irreducible component $Z$ of the special fiber $\mathcal{W}_s$ such that $\deg \omega_{\mathcal{W}/\OO_K}|_Z = 0$.} contained in the exceptional locus of $\rho$.
        
        \item \label{lem:stable_hull_iii} (Criterion for Stability) Assuming $g(X) \ge 2$, the stable hull $\X$ is the stable model of $X$ if and only if for every irreducible component $Z'$ of the strict transform $C_\pi$, we have 
        \[\deg \omega_{\X/\OO_K}|_{Z'} > 0.\]
    \end{enumerate}
\end{lemma}

\begin{proof}
    This is a consequence of the results in \cite{liu2006stable}. We explain how each part follows.

    The existence (after a possible extension of $K$) and uniqueness of the stable hull is the statement of \cite[Theorem 2.3]{liu2006stable}. Furthermore, by \cite[Proposition 2.7]{liu2006stable} the base field $K$ itself is sufficient if $X$ has semistable reduction and the special fiber $\Y_s$ is geometrically reduced, as we have assumed.

    Assertion \ref{lem:stable_hull_i} is \cite[Proposition 2.14 (a)]{liu2006stable}, and assertion \ref{lem:stable_hull_ii} follows from \cite[Proposition 2.14 (b)]{liu2006stable}.

    The final assertion \ref{lem:stable_hull_iii} is also a direct consequence. The semistable model $\X$ is the stable model if and only if its special fiber $\X_s$ is a stable curve. This is equivalent to $\X_s$ being canonically positive, meaning $\deg \omega_{\X/\OO_K}|_Z > 0$ for every component $Z$ of $\X_s$. By \ref{lem:stable_hull_i}, this condition is already satisfied for all components of the exceptional locus $E_\pi$. Since $\X_s$ is the union of $E_\pi$ and the strict transform $C_\pi$, the stability of $\X$ is therefore equivalent to the condition holding for the remaining components, namely those of $C_\pi$.
\end{proof}

With these two lemmas, we can now finally prove the converse of our Main Theorem \ref{thm:main_theorem}.

\begin{theorem}
    Let $X$ be a smooth plane quartic over $K$. Assume that a GIT-stable plane model $\X_0$ of $X$ and the stable model $\X$ of $X$ both exist over $K$. Then the stable reduction $\oo{X}$ is non-hyperelliptic, and the stable model $\X$ is the unique minimal semistable model dominating the GIT-stable model $\X_0$. The corresponding domination morphism $\phi\colon \X \to \X_0$ induces a morphism $\oo{\phi}$ on the special fiber that contracts the $1$-tails of $\oo{X}$ to the cusps of $\X_{0,s}$ and is an immersion elsewhere.
\end{theorem}

\begin{proof}
    We first verify that the hypotheses of Lemma \ref{lem:stable_hull} are met. The GIT-stable model $\X_0$ is normal because its special fiber is reduced (see \cite[Lemma 1.1]{liu1999models}). Since we assume the stable model of $X$ exists over $K$, we know that $X$ has semistable reduction. The hypotheses of Lemma \ref{lem:stable_hull} are therefore met, which guarantees the existence of a unique stable hull of $\X_0$. We denote this stable hull by $\X$ and the corresponding domination morphism by $\phi\colon \X \to \X_0$. A priori, it is not clear that this stable hull is the stable model of $X$; we now verify this.

    Let $C := \X_s$ be the special fiber of the stable hull, let $C_0 := \X_{0,s}$, and let $C_\phi$ and $E_\phi$ be the strict transform and exceptional locus of $\phi$, respectively, so that set-theoretically $C = C_\phi \cup E_\phi$. The morphism $\phi$ modifies $\X_0$ to achieve semistability by resolving only the non-semistable singularities, which are the cusps. The nodal singularities of $C_0$ are preserved, meaning $\phi$ is an isomorphism in a neighborhood of each node. To see this, recall from Lemma \ref{lem:stable_hull} \ref{lem:stable_hull_ii} how the stable hull is constructed: one first takes the minimal desingularization $\rho\colon \X_{0, \text{reg}} \to \X_0$, and then contracts all resulting $(-2)$-curves to obtain $\X$. The minimal desingularization of an arithmetic $A_{d-1}$ singularity (a node of thickness $d$) consists of a chain of $d-1$ $(-2)$-curves (see \cite[Corollary 10.3.25]{liu}). The stable hull construction contracts these exact same curves, so the net result is that the original nodal singularity is unchanged. At the cusps of $C_0$, $C$ must only have nodes. Since cusps are unibranch singularities, the strict transform $C_\phi$ must be smooth at the unique point lying over each cusp. Therefore, the induced map $\phi|_{C_\phi} \colon C_\phi \to C_0$ is the partial normalization of $C_0$ at its $r$ cusps. Let $S = \{x_1, \dots, x_r\} \subset C_\phi$ be the set of preimages of the cusps.

    Since $\X_0$ is normal, the fibers of the birational morphism $\phi$ are connected (by Zariski's Connectedness Principle, see \cite[Corollary 5.3.16]{liu}). The exceptional locus $E_\phi$ thus consists of a disjoint union of $r$ connected components, $E_\phi = \bigsqcup_{i=1}^r E_i$, where each $E_i$ is the fiber over a cusp and intersects $C_\phi$ at the single point $x_i$.

    To show that $\X$ is the stable model, we must show its special fiber $C$ is a stable curve. Since $C$ is formed by attaching the components $E_i$ to the curve $C_\phi$ at the nodes in $S$, by the decomposition lemma (Lemma \ref{lem:stable_decomposition}), $C$ is stable if and only if its pointed components, $(C_\phi, S)$ and $(E_i, x_i)$ for all $i$, are stable.
    \begin{itemize}
        \item The stability of the pointed curve, $(C_\phi, S)$, is given directly by Lemma \ref{lem:structure_GIT_stable_quartic} \ref{lem:structure_GIT_stable_quartic_iv}.
        \item To prove the stability of the pointed exceptional components, $(E_i, x_i)$, we apply Proposition \ref{prop:stable_via_canonical}. We must show that $\deg(\omega_{E_i/k}(x_i)|_Z) > 0$ for every irreducible component $Z$ of $E_i$. By the adjunction formula, the dualizing sheaf of $C$ restricts as $\omega_{C/k}|_{E_i} \simeq \omega_{E_i/k}(x_i)$. Restricting this isomorphism further to a component $Z \subset E_i$ gives $\deg(\omega_{E_i/k}(x_i)|_Z) = \deg(\omega_{C/k}|_Z)$. The latter term is positive by Lemma \ref{lem:stable_hull} \ref{lem:stable_hull_i}.
    \end{itemize}
    Since all components of its decomposition are stable, $C$ is a stable curve. By uniqueness, $\X$ is the stable model of $X$, and $C = \oo{X}$ is its stable reduction.

    The components $E_i$ are attached to $C_\phi$ at the separating nodes $x_i$. The genus formula for this decomposition is $g(\oo{X}) = g(C_\phi) + \sum_{i=1}^r g(E_i)$. By Lemma \ref{lem:structure_GIT_stable_quartic}, $g(C_\phi)=3-r$. Since $g(\oo{X})=3$, it follows that $\sum g(E_i) = r$. As we have already established that each $(E_i, x_i)$ is a stable pointed curve, its genus must be at least $1$ (see Remark \ref{rem:stable_curves_low_genus}). We must therefore have $g(E_i)=1$ for all $i$. This shows that the components $E_i$ are precisely the $1$-tails of $\oo{X}$. The remaining curve, $C_\phi$, is 2-inseparable by Lemma \ref{lem:structure_GIT_stable_quartic} \ref{lem:structure_GIT_stable_quartic_ii} and thus has no separating nodes itself. By definition, it is the core of $\oo{X}$.

    Finally, we prove that $\oo{X}$ is non-hyperelliptic. By the decomposition lemma (Lemma \ref{lem:decomp_sep_node}), if $\oo{X}$ were hyperelliptic, then its pointed core $(C_\phi, S)$ would also have to be hyperelliptic. This contradicts Lemma \ref{lem:structure_GIT_stable_quartic}  \ref{lem:structure_GIT_stable_quartic_iv}. Therefore, $\oo{X}$ is non-hyperelliptic. This completes the proof, as the construction of $\X$ as the stable hull of $\X_0$ and the analysis of the morphism $\phi$ have established all claims made in the theorem.
\end{proof}

\begin{remark}
    The contraction morphism $\phi\colon \X \to \X_0$ derived from the stable hull construction is precisely the morphism described in Theorem \ref{thm:main_thm_2-insep_tails}.
\end{remark}

\subsection{Final Remarks}

To conclude this thesis, we offer some final remarks on how our results provide insight into the computation of the stable model of a smooth plane quartic. From a practical viewpoint, one starts with a smooth plane quartic $X$ over $K$ and must first find an extension $L/K$ over which the stable model $\X$ of $X_L := X \otimes_K L$ exists. Using the methods in \cite{stern2025models}, one can compute an extension $K'/K$ such that the base-changed curve $X_{K'} := X \otimes_K K'$ admits a GIT-semistable model $\X_0$. If this model is GIT-stable, it is the \emph{unique} GIT-semistable model; otherwise, no GIT-stable model exists, and there may be multiple non-isomorphic GIT-semistable models (cf.~\cite[\S1.3, \S6.3]{stern2025models}). Our Main Theorem \ref{thm:main_theorem} provides the geometric interpretation of this dichotomy: a GIT-stable model fails to exist if and only if the stable reduction $\oo{X}$ is hyperelliptic. Thus, by testing the computed model $\X_0$ for GIT-stability, one can algorithmically determine whether the stable reduction $\oo{X}$ is hyperelliptic. Characterizing when $\oo{X}$ is a \emph{smooth} hyperelliptic curve (so-called ``good hyperelliptic reduction'') was already investigated in \cite{lercier2019stable}.

In the hyperelliptic case, it is not clear how to proceed to determine the reduction type or to compute the stable model. Proposition \ref{prop:classification_hyp_3} shows that some combinatorial types are always non-hyperelliptic and thus excluded in this scenario, but many possibilities remain. If $\oo{X}$ is assumed to be additionally smooth, the theory of \emph{toggle models} developed in \cite{lercier2019stable, lercier2021reduction} can be applied; assuming the residue characteristic is not $2$, these models can be computed using a specific set of bitangents. In \cite[Remark 2.10]{lercier2021reduction}, the authors mention that this notion can be easily extended to the case where $\oo{X}$ is irreducible and hyperelliptic, and they indicate that it may be extended to further types. In view of Theorem \ref{thm:main_thm_2-insep_tails}, it seems plausible that the theory of toggle models could be extended to all hyperelliptic reductions $\oo{X}$ whose core is 2-inseparable. It may be possible to adapt the methods in \cite{lercier2019stable} to compute the stable model in these cases. The case when the reduction $\oo{X}$ has a 2-separable core seems to be of a different nature. For example, if the reduction has type \texttt{1=1} and the nodes have thicknesses $1$ and $2$, the only possible contraction morphism $\X \to \PP^2_{\OO_K}$ is given by the dualizing sheaf $\omega_{\X/\OO_K}$, as any non-trivial twists would result in a line bundle not generated by its global sections. The resulting plane model will have as its special fiber two double lines crossing at one point. This is not a GIT-semistable model, which suggests that the theory of toggle models does not apply in this case. The table in \cite[Table 5.1]{van2025reduction} suggests a connection between a hyperelliptic stable reduction with a 2-separable core and the appearance of $A_3$-singularities on the special fiber of a GIT-semistable model. This connection is not yet entirely clear.

In the non-hyperelliptic case, our results provide a clearer path forward. Here, any GIT-semi\-stable model $\X_0$ must be GIT-stable. If its special fiber, $\X_{0,s}$, has no cusps, then by Proposition \ref{prop:main_thm_2insep}, this model is in fact the stable model. This occurs if and only if the stable reduction $\oo{X}$ has no $1$-tails; in such cases, our approach yields a direct method to compute the stable model. In the general case, the quartic $\X_{0,s}$ still largely determines the reduction type of $X$. Each cusp on $\X_{0,s}$ corresponds to a $1$-tail on the stable reduction $\oo{X}$, allowing for an almost complete determination of the combinatorial type; the only remaining ambiguity is whether a $1$-tail is an elliptic tail or a pig tail. Our Main Theorem also describes how, in theory, to obtain the stable model $\X$ from the GIT-stable model $\X_0$: it is the minimal semistable model dominating $\X_0$. Crucially, this relationship only holds over a field extension $L/K$ where both models exist. Therefore, a crucial step toward a computational algorithm is the explicit determination of this field extension $L/K$. The subsequent construction of the stable model $\X$ from the GIT-stable model $\X_0$ is then achieved via a direct weighted blow-up at each cusp. The development of a complete procedure that details how to find the required extension and perform this resolution is the subject of the companion articles \cite{quarticpaper, wewers2026explicit}.

\appendix
\section{Riemann-Roch Proofs} \label{app}

The proof of Theorem \ref{thm:main_thm_2-insep_tails} hinges on the analysis of the linear system $(\oo{\LL}_c, W)$ on the 2-in\-sep\-a\-ra\-ble core $C_c$; by Lemma \ref{lem:contraction_of_tails}, the properties of this system determine both the existence of the contraction morphism and the geometry of its image. While the proof in Chapter \ref{chap:3} identifies this system with the canonical system of an associated curve $C'$ with cusps, this appendix provides a more direct approach. We analyze $(\oo{\LL}_c, W)$ using standard Riemann-Roch calculations, which highlights precisely how the geometric conditions for hyperellipticity from Proposition \ref{prop:2-insep_geom_condition_hyperelliptic} arise. This method is particularly illustrative for the case $r=3$, where an explicit computation reveals the crucial dependence on the field's characteristic.

The proof proceeds by a case-by-case analysis for $r \geq 1$, mirroring the structure of Proposition \ref{prop:2-insep_geom_condition_hyperelliptic}. For notational simplicity, we replace the symbols $C_c$ with $C$ and $\oo{\LL}_c$ with $\LL$. Using this new notation, the crucial step in proving Theorem \ref{thm:main_thm_2-insep_tails} is the following proposition.

\begin{proposition}
    Let $r \in \{1, 2, 3\}$, and let $(C, \{x_1, \dots, x_r\})$ be a stable, 2-inseparable pointed curve of arithmetic genus $g = 3-r$ over an algebraically closed field $k$. Let $\LL := \omega_{C/k}(\sum_{i=1}^r 2x_i)$, and let
    \[
        W := \{ s \in H^0(C, \LL) \ssep \Res_{x_i}(s) = 0 \text{ for all $i = 1, \dots, r$} \}.
    \]
    By Lemma \ref{lem:subspace_W}, the vector space $W$ is 3-dimensional. Then the associated linear series $(\LL, W)$ is base-point free and induces a morphism $\phi \colon C \to \PP^2_k$.
    \begin{enumerate}[label=\StatementListLabel]
        \item If $(C, \{x_1, \dots, x_r\})$ is not hyperelliptic, then $\phi$ is a birational morphism onto a plane quartic $C'$. The map $\phi$ is an immersion away from the marked points $\{x_i\}$; it maps the nodes of $C$ to distinct nodes on $C'$ and each marked point $x_i$ to a distinct ordinary cusp on $C'$.
        \item If $(C, \{x_1, \dots, x_r\})$ is hyperelliptic, then $\phi$ is a degree-two morphism onto a smooth plane conic $Q$, and the points $\phi(x_i)$ are distinct points on $Q$.
    \end{enumerate}
\end{proposition}

\begin{proof}
    To determine the geometric properties of the rational map $\phi$ induced by the 3-di\-men\-sion\-al linear series $(\LL, W)$, we analyze the restriction map for zero-dimensional closed subschemes $Z \subset C$:
    \[
        \rho_Z\colon W \to H^0(C, \LL|_Z).
    \]
    The behavior of $\phi$ is determined by the surjectivity of $\rho_Z$ (see, e.g., the local criterion in \cite[\href{https://stacks.math.columbia.edu/tag/0E8R}{Section 0E8R}]{stacks-project}):
    \begin{itemize}
        \item The linear series $(\LL, W)$ is base-point-free, and thus $\phi$ is a morphism, if and only if $\rho_Z$ is surjective for all $Z$ of length 1, i.e., for all closed points.
        \item The morphism $\phi$ is a closed immersion if and only if $\rho_Z$ is surjective for all $Z$ of length $2$. This criterion ensures both injectivity and the immersion property.
    \end{itemize}
    We focus on the closed immersion criterion. Since $\dim(W)=3$ and $\dim(H^0(C, \LL|_Z))=2$ for a length-2 subscheme $Z$, surjectivity of $\rho_Z$ is equivalent to its kernel having dimension one. The kernel consists of sections in $W$ vanishing on $Z$:
    \[
        \ker(\rho_Z) = W \cap H^0(C, \mathcal{I}_Z \otimes \LL),
    \]
    where $\mathcal{I}_Z$ is the ideal sheaf of $Z$. We define a \emph{problematic subscheme} as a length-2 subscheme $Z$ for which this condition fails, i.e., for which $\dim(\ker(\rho_Z)) \ge 2$.
    
    The proof proceeds by identifying all problematic subschemes. This also establishes base-point-freeness, since the surjectivity of $\rho_Z$ for a length-2 subscheme $Z$ whose support includes a point $p$ implies that the restriction map to $p$ is also surjective, meaning $p$ is not a base point. The subsequent analysis will show that such a subscheme $Z$ exists for any point $p$. The geometric nature of $\phi$ is then determined by the configuration of these problematic subschemes, which we will show corresponds to the hyperellipticity of the pointed curve $(C, \{x_i\})$.

    \begin{prooflist}
        \item Assume $(C, \{x_1, \dots, x_r\})$ is not hyperelliptic. We will show that the only problematic subschemes are $Z=2x_i$ for $i=1,\dots,r$. Consequently, the rational map $\phi$ is a morphism that is both injective --- since $\rho_{p+q}$ is surjective for any distinct points $p, q \in C$ --- and an immersion at any point not in $\{x_i\}$. The map $\phi$ is therefore a birational morphism onto its image $C' = \phi(C)$. Since $\deg(\LL)=4$, $C'$ is a plane quartic of arithmetic genus $g(C')=3$.
    
        The difference in arithmetic genera, $g(C') - g(C) = 3 - (3-r) = r$, is accounted for by the singularities created by $\phi$. Since $\phi$ is an immersion away from the marked points $\{x_i\}$, this difference must be concentrated at the images of these points, $\phi(x_i)$. As the immersion fails at each $x_i$ (since $Z=2x_i$ is a problematic subscheme), each image point $\phi(x_i)$ is a unibranch singularity with $\delta$-invariant $\delta_{\phi(x_i)} \ge 1$. As $\phi$ is injective, these $r$ points are distinct, and the sum of their $\delta$-invariants must be $r$. Thus, $\delta_{\phi(x_i)} = 1$ for each $i$, which means each $\phi(x_i)$ is an ordinary cusp.
    
        \item Assume $(C, \{x_1, \dots, x_r\})$ is hyperelliptic, and let $\pi\colon C \to \PP^1_k$ be the associated degree-two morphism as in Remark \ref{rem:degree2_map} (note that the target is $\PP^1_k$ since $C$ is 2-inseparable). We will show that the problematic subschemes are precisely the fibers of $\pi$. If this holds, the linear system $W$ fails to separate points and tangent vectors in the fibers of $\pi$, so the morphism $\phi$ must factor through $\pi$ as $\phi = \psi \circ \pi$, for some morphism $\psi\colon \PP^1_k \to \PP^2_k$. This factorization implies that the linear series $(\LL, W)$ is the pullback of a linear series on $\PP^1_k$ defined by a vector space $U$ that induces $\psi$, i.e., $W = \pi^*(U)$.

        A degree calculation shows that $\psi$ is induced by a line bundle of degree $\deg(\LL)/\deg(\pi) \allowbreak = 4/2=2$. Since $\pi$ is a surjective morphism, the pullback map $\pi^*$ on global sections is injective. This provides an isomorphism between $U$ and its image $W$, so we have $\dim(U)=\dim(W)=3$. This forces $(\mathcal{O}(2), U)$ to be the complete linear system $|\mathcal{O}(2)|$ on $\PP^1_k$. Therefore, $\psi$ is the 2-uple Veronese embedding. The image of $\phi$ is a smooth plane conic, and $\phi$ is a degree-two morphism. Our case-by-case analysis will also verify that the points $\phi(x_i)$ are distinct.
    \end{prooflist}

    We now identify the problematic subschemes in a case-by-case analysis, applying the Riemann-Roch and Serre Duality theorems for coherent sheaves on the Gorenstein curve $C$ to verify the preceding claims.

    \begin{case}[$r=1$, irreducible]
        The curve $C$ is an irreducible curve of genus $g=2$ with a single marked point $x_1$, and $\LL = \omega_{C/k}(2x_1)$. By \eqref{eq:dimension_h0}, $h^0(C, \LL)=3$. Since $W$ is a 3-dimensional subspace of $H^0(C, \LL)$, we have $W=H^0(C, \LL)$.
    
        A length-2 subscheme $Z$ is problematic if and only if $h^0(C, \mathcal{I}_Z \otimes \LL) \ge 2$. As $\deg(\mathcal{I}_Z \otimes \LL) = 2$ and $g(C)=2$, the Riemann-Roch theorem implies this is equivalent to $h^1(C, \mathcal{I}_Z \otimes \LL) \ge 1$. By Serre Duality, this holds if and only if $\mathrm{Hom}(\mathcal{I}_Z \otimes \LL, \omega_{C/k}) \neq 0$. Substituting $\LL = \omega_{C/k}(2x_1)$, this condition becomes $\mathrm{Hom}(\mathcal{I}_Z \otimes \omega_{C/k}(2x_1), \omega_{C/k}) \neq 0$, which simplifies to $\mathrm{Hom}(\mathcal{I}_Z, \mathcal{I}_{2x_1}) \neq 0$. Since $C$ is irreducible, any non-zero homomorphism between these torsion-free, rank-1 sheaves is injective. As both sheaves have the same degree ($-2$), such a homomorphism must be an isomorphism. This isomorphism directly implies that $Z$ is an effective Cartier divisor linearly equivalent to $2x_1$. The analysis now depends on the dimension of $H^0(C, \mathcal{O}_C(2x_1))$.
    
        \begin{subcase}[Non-hyperelliptic]
            By Proposition \ref{prop:2-insep_geom_condition_hyperelliptic}, $(C, \{x_1\})$ is not hyperelliptic if and only if $h^0(C, \mathcal{O}_C(2x_1))=1$. This implies that the only effective Cartier divisor linearly equivalent to $2x_1$ is $2x_1$ itself. Thus, $Z=2x_1$ is the unique problematic subscheme.
        \end{subcase}
    
        \begin{subcase}[Hyperelliptic]
            If the pointed curve $(C, \{x_1\})$ is hyperelliptic, then by Proposition \ref{prop:2-insep_geom_condition_hyperelliptic}, we have $\omega_{C/k} \simeq \mathcal{O}_C(2x_1)$. This implies $h^0(C, \mathcal{O}_C(2x_1)) = h^0(C, \omega_{C/k}) = g(C)=2$. The linear system $|\mathcal{O}_C(2x_1)|$ therefore induces the degree-two canonical map $\pi\colon C \to \PP^1_k$. The problematic subschemes are precisely the fibers of this map.
        \end{subcase}
    \end{case}

    \begin{case}[$r=1$, reducible]
        The curve $C$ is a binary curve of genus $g=2$, composed of two smooth rational components $C_1$ and $C_2$ meeting at three nodes $N = C_1 \cap C_2$ (cf. Proposition \ref{prop:classification}). By Proposition \ref{prop:2-insep_geom_condition_hyperelliptic}, this configuration is always non-hyperelliptic. We may assume the marked point $x_1$ is a smooth point on $C_1$. As in the irreducible case, $W=H^0(C, \LL)$.
        
        The analysis proceeds exactly as in the irreducible case, where a subscheme $Z$ is problematic if and only if there exists a non-zero homomorphism $\psi\colon \mathcal{I}_Z \to \mathcal{I}_{2x_1}$. We claim that any such $\psi$ is necessarily an isomorphism. Let $\mathcal{K} = \ker \psi$ and $\mathcal{F} = \im \psi$ be its kernel and image, which fit into the exact sequence
        \[
            0 \to \mathcal{K} \to \mathcal{I}_Z \to \mathcal{F} \to 0.
        \]
        Since $g(C)=2$ and $Z$ is a subscheme of length 2, we have $\chi(\mathcal{I}_Z) = (1-g(C)) - 2 = -3$.
    
        Assume by contradiction that $\mathcal{K} \neq 0$. Since $\mathcal{I}_Z$ is torsion-free of rank $1$ (i.e., is generically locally free with multirank $(1,1)$ on $(C_1, C_2)$), if $\mathcal{K}$ had multirank $(1,1)$, then $\mathcal{F}$ would have rank 0. As $\mathcal{F}$ is torsion-free (being a subsheaf of $\mathcal{I}_{2x_1}$), this would imply $\mathcal{F}=0$, contradicting $\psi \neq 0$. Thus, $\mathcal{K}$ must have multirank $(1,0)$ or $(0,1)$.
    
        Suppose $\mathcal{K}$ has multirank $(1,0)$. Torsion-freeness implies $\mathcal{K}$ vanishes identically on $C_2$. Thus, $\mathcal{K}$ is a subsheaf of the ideal sheaf $\mathcal{J}_2$ of $C_2$ in $C$. Correspondingly, $\mathcal{F}$ has multirank $(0,1)$ and is a subsheaf of the ideal sheaf $\mathcal{J}_1$ of $C_1$ in $C$.
    
        We analyze the Euler characteristics. The ideal sheaf $\mathcal{J}_1$ is isomorphic to $\mathcal{O}_{C_2}(-N)$, viewed as a sheaf on $C$ supported on $C_2$, and similarly $\mathcal{J}_2$ is isomorphic to $\mathcal{O}_{C_1}(-N)$ supported on $C_1$. The cohomology of these sheaves on $C$ agrees with the cohomology of the corresponding sheaves on the components. Thus,
        \[
            \chi(\mathcal{J}_1) = \chi(\mathcal{J}_2) = \chi(\OO_{\PP^1_k}(-3)) = -2.
        \]
        Consider the inclusion $\mathcal{K} \hookrightarrow \mathcal{J}_2$. The cokernel $\mathcal{Q}$ is a rank 0 (skyscraper) sheaf, hence $\chi(\mathcal{Q}) \ge 0$. This implies $\chi(\mathcal{K}) = \chi(\mathcal{J}_2) - \chi(\mathcal{Q}) \leq -2$. Similarly, the inclusion $\mathcal{F} \hookrightarrow \mathcal{J}_1$ implies $\chi(\mathcal{F}) \leq -2$. Combining these inequalities yields
        \[
            \chi(\mathcal{K}) + \chi(\mathcal{F}) \leq -4,
        \]
        which contradicts $\chi(\mathcal{K}) + \chi(\mathcal{F}) = \chi(\mathcal{I}_Z)=-3$. The case where $\mathcal{K}$ has multirank $(0,1)$ is symmetric. Thus, we must have $\mathcal{K}=0$, so $\psi$ is injective.
   
        Since $\psi$ is an injective morphism between torsion-free, rank-1 sheaves with the same Euler characteristic, its cokernel is a torsion sheaf of length zero. The cokernel is therefore trivial, which means $\psi$ is an isomorphism, so $\mathcal{I}_Z \simeq \mathcal{I}_{2x_1}$. Consequently, $\mathcal{I}_Z$ is an invertible sheaf isomorphic to $\mathcal{I}_{2x_1} = \mathcal{O}_C(-2x_1)$ (as $x_1$ is smooth), which implies $Z$ is an effective Cartier divisor linearly equivalent to $2x_1$.
    
        To conclude that $Z=2x_1$, we show that $h^0(C, \mathcal{O}_C(2x_1))=1$. Let $\mathcal{M} = \mathcal{O}_C(2x_1)$. On the components $C_1$ and $C_2$, the restrictions of this sheaf are $\mathcal{M}|_{C_1} \simeq \mathcal{O}_{\mathbb{P}^1_k}(2)$ and $\mathcal{M}|_{C_2} \simeq \mathcal{O}_{\mathbb{P}^1_k}(0)$, respectively. A global section $s=(s_1,s_2)$ must therefore have a constant component $s_2=c$. The gluing condition then requires that $s_1 \in H^0(C_1, \mathcal{M}|_{C_1})$ take the value $c$ at the three distinct nodes. A rational function on $\mathbb{P}^1_k$ of degree at most $2$ that takes the same value at three distinct points must be constant. Thus $s_1=c$, and it follows that the space of global sections is 1-dimensional, i.e., $h^0(C, \mathcal{M})=1$.
    
        Consequently, $2x_1$ is the unique effective Cartier divisor in its linear equivalence class, establishing that $Z=2x_1$ is the unique problematic subscheme.
    \end{case}

    \begin{case}[$r=2$]
        The curve $C$ has genus $g=1$ and marked points $x_1, x_2$; the 2-inseparability implies that $C$ is irreducible (cf. Proposition~\ref{prop:classification}). The line bundle is $\LL = \omega_{C/k}(2x_1+2x_2)$, and we have $h^0(C, \LL) = 4$. For any section $s \in H^0(C, \LL)$, the sum of its residues must be zero. As poles can only occur at $x_1$ and $x_2$, we have $\Res_{x_1}(s) + \Res_{x_2}(s) = 0$. The subspace $W$ is therefore the kernel of either residue map, as the condition $\Res_{x_1}(s)=0$ implies $\Res_{x_2}(s)=0$.
    
        To identify problematic subschemes, let $Z$ be a length-2 subscheme and let $V_Z := H^0(C, \mathcal{I}_Z \otimes \LL)$. By the Riemann-Roch theorem, $\dim(V_Z) = 2$. A subscheme $Z$ is problematic if $\dim(W \cap V_Z) \ge 2$. Since $\dim(V_Z)=2$, this is equivalent to the condition $V_Z \subset W$. This, in turn, occurs if and only if every section in $V_Z$ has zero residue at the marked points, i.e., $\Res_{x_1}|_{V_Z}=0$ and $\Res_{x_2}|_{V_Z}=0$. We now analyze the possibilities for $Z$. 
        
        First, if $Z=2x_1$, then $V_{2x_1} = H^0(C, \LL(-2x_1)) = H^0(C, \omega_{C/k}(2x_2))$. Any section $s \in V_{2x_1}$ is a differential with at most a double pole at $x_2$, so it is regular at $x_1$, which implies $\Res_{x_1}(s)=0$. Thus $V_{2x_1} \subset W$, and $Z=2x_1$ is a problematic subscheme. By symmetry, $Z=2x_2$ is also problematic.
    
        Second, for a subscheme $Z=x_1+p$ with $p \neq x_1$, the kernel of the residue map $\Res_{x_1}$ on $V_Z = H^0(C, \mathcal{I}_p \otimes \omega_{C/k}(x_1+2x_2))$ is the subspace $H^0(C, \mathcal{I}_p \otimes \omega_{C/k}(2x_2))$, which has dimension $1$ by the Riemann-Roch theorem. Since the kernel is a proper subspace of $V_Z$, it follows that $\Res_{x_1}|_{V_Z} \neq 0$ and thus $V_Z \not\subset W$. This means $Z=x_1+p$ is not a problematic subscheme, which confirms that $\phi$ separates $x_1$ from any other point $p$. In particular, for $p=x_2$, we have $\phi(x_1) \neq \phi(x_2)$.
    
        Finally, a subscheme $Z$ disjoint from $\{x_1, x_2\}$ is problematic if and only if it is a Cartier divisor $D$ linearly equivalent to both $2x_1$ and $2x_2$. To prove this, first assume $Z$ is problematic. Then $V_Z \subset W$, so any section in $V_Z$ has a vanishing residue at $x_1$. A section in the one-dimensional subspace $H^0(C, \mathcal{I}_Z \otimes \omega_{C/k}(x_1+2x_2)) \subset V_Z$ must therefore be regular at $x_1$, implying it is a global section of $\mathcal{F} := \mathcal{I}_Z \otimes \omega_{C/k}(2x_2)$. Since $\mathcal{F}$ is a torsion-free, rank-1 coherent sheaf of degree $0$ on a genus $1$ curve with a non-zero section, it must be isomorphic to the trivial sheaf, $\mathcal{O}_C$. This implies $\mathcal{I}_Z \simeq \omega_{C/k}(-2x_2) \simeq \mathcal{O}_C(-2x_2)$, and hence $Z$ is a Cartier divisor $D$ with $D \sim 2x_2$. A symmetric argument using the vanishing residue at $x_2$ shows that $D \sim 2x_1$. Conversely, assume $D$ is a Cartier divisor such that $D \sim 2x_1$ and $D \sim 2x_2$. We show that $V_D \subset W$ by decomposing $V_D$ as a direct sum. Define the subspaces $V_1 := H^0(C, \mathcal{I}_D \otimes \LL(-2x_1))$ and $V_2 := H^0(C, \mathcal{I}_D \otimes \LL(-2x_2))$. The linear equivalences imply that $V_1 \simeq H^0(C, \omega_{C/k})$ and $V_2 \simeq H^0(C, \omega_{C/k})$, so both are one-dimensional. Their intersection $V_1 \cap V_2 = H^0(C, \omega_{C/k}(-D))$ is trivial, as the sheaf has negative degree. Since $\dim(V_D)=2$, we have the decomposition $V_D = V_1 \oplus V_2$. Any section in $V_1$ is regular at $x_1$, so its residue at $x_1$ vanishes. Any section in $V_2$ is regular at $x_2$; by the Residue Theorem, its residue at $x_1$ must also vanish. Thus, every section in $V_D$ has a vanishing residue at $x_1$, which confirms that $V_D \subset W$.
    
        The analysis of these cases now depends on whether the pointed curve $(C, \{x_1, x_2\})$ is hyperelliptic, which is by Proposition~\ref{prop:2-insep_geom_condition_hyperelliptic} equivalent to the condition $\mathcal{O}_C(2x_1) \simeq \mathcal{O}_C(2x_2)$.
    
        \begin{subcase}[Non-hyperelliptic]
            In this case, the Cartier divisors $2x_1$ and $2x_2$ are not linearly equivalent. The preceding analysis implies that no problematic subscheme disjoint from $\{x_1, x_2\}$ can exist, as its existence would require $2x_1 \sim D \sim 2x_2$. Recalling that subschemes of the form $x_i+p$ (for $p \neq x_i$) are also not problematic, we conclude that the only problematic subschemes are $2x_1$ and $2x_2$.
        \end{subcase}
        
        \begin{subcase}[Hyperelliptic]
            In this case, we have the linear equivalence $2x_1 \sim 2x_2$. The preceding analysis then shows that the problematic subschemes are precisely the effective Cartier divisors linearly equivalent to $2x_1$. The linear system $|2x_1|$ induces the degree-two map $\pi\colon C \to \PP^1_k$ associated with the hyperelliptic involution (cf. proof of Lemma~\ref{lem:always_hyperelliptic}~\ref{lem:always_hyperelliptic_i}), and its fibers are therefore exactly the problematic subschemes.
        \end{subcase}
    \end{case}

    \begin{case}[$r=3$]
    The curve $(C, \{x_i\})$ is a stable $3$-pointed curve of genus $g=0$, thus $C \simeq \PP^1_k$. The line bundle $\LL = \omega_{C/k}(2x_1+2x_2+2x_3)$ has degree 4. $W$ is the 3-dimensional subspace of the 5-dimensional space $H^0(C, \LL)$ consisting of sections with vanishing residues at the marked points.
    
    We identify problematic subschemes $Z$ of length $2$. Let $V_Z := H^0(C, \mathcal{I}_Z \otimes \LL)$. By the Riemann-Roch theorem, $\dim(V_Z)=3$. $Z$ is problematic if $\dim(W \cap V_Z) \ge 2$.

    First, consider $Z=2x_1$. We have $V_{2x_1} = H^0(C, \omega_{C/k}(2x_2+2x_3))$. Sections $s \in V_{2x_1}$ are regular at $x_1$, so $\Res_{x_1}(s)=0$. By the Residue Theorem, $\Res_{x_2}(s) + \Res_{x_3}(s) = 0$. Thus, the condition for $s \in W$ reduces to the single linear constraint $\Res_{x_2}(s)=0$. This constraint is non-trivial as the residue map $\Res_{x_2}|_{V_{2x_1}}\colon V_{2x_1} \to k$ is surjective (cf. the argument in the proof of Lemma \ref{lem:subspace_W}). Therefore, $\dim(W \cap V_{2x_1}) = \dim(V_{2x_1}) - 1 = 2$. By symmetry, $2x_1, 2x_2, 2x_3$ are always problematic subschemes.

    To analyze other subschemes, we fix an isomorphism $C \simeq \PP^1_k$ such that $x_1=0, x_2=1, x_3=\infty$. A basis for $W$ is:
    \[
        s_1 = \frac{1}{t^2} dt, \quad s_2 = \frac{1}{(t-1)^2} dt, \quad s_3 = dt.
    \]
    (These have zero residues at $0$ and $1$. At infinity, using $u=1/t$, $dt=-u^{-2}du$, the basis transforms to $s_1 = -du$, $s_2 = -(1-u)^{-2}du$, and $s_3 = -u^{-2}du$. Since none have a $u^{-1}$ term, the residues at infinity are zero.)
    
    Let $s = \sum a_i s_i$ be a generic section. For $p \notin \{0, 1, \infty\}$, $dt$ generates $\LL$ locally, and $s$ vanishes at $p$ if its coefficient function $f(t) = \frac{a_1}{t^2} + \frac{a_2}{(t-1)^2} + a_3$ vanishes at $p$. At the marked points, the basis elements serve as local generators for $\LL$. Specifically, $s_1$ generates $\LL$ at $x_1=0$. Since $s_2$ and $s_3$ vanish at $t=0$ when expressed in terms of the generator $s_1$ (i.e., the ratios $s_2/s_1$ and $s_3/s_1$ vanish at $x_1$), the section $s$ vanishes at $x_1$ if and only if $a_1=0$. Similarly, the vanishing condition at $x_2=1$ is $a_2=0$, and at $x_3=\infty$ it is $a_3=0$.

    Consider $Z=p+q$ with distinct $p, q$. $Z$ is problematic if the two vanishing conditions on $(a_i)$ are dependent. If $p,q \notin \{0,1,\infty\}$, dependence requires the vectors $(\frac{1}{p^2}, \frac{1}{(p-1)^2}, 1)$ and $(\frac{1}{q^2}, \frac{1}{(q-1)^2}, 1)$ to be proportional. This forces $p^2=q^2$ and $(p-1)^2=(q-1)^2$, implying $p=q$ regardless of the characteristic, a contradiction. If $p$ or $q$ is a marked point, the conditions (e.g., $a_1=0$ and $f(q)=0$) are also readily verified to be independent. Thus, $\dim(W \cap V_{p+q})=1$. These subschemes are never problematic, confirming that $\phi$ is injective.

    Finally, consider $Z=2p$, where $p \notin \{x_i\}$. The conditions for a section $s \in W$ to be in $V_{2p}$ are $f(p)=0$ ($\ast$) and $f'(p)=0$ ($\ast\ast$), where
    \[
        f'(t) = -2 \frac{a_1}{t^3} - 2 \frac{a_2}{(t-1)^3}.
    \]
    The analysis depends on the characteristic, which by Proposition~\ref{prop:2-insep_geom_condition_hyperelliptic} determines hyperellipticity.

    \begin{subcase}[Non-hyperelliptic, $\mathrm{char}(k) \neq 2$]
        The factor $-2$ is invertible. We check the independence of ($\ast$) and ($\ast\ast$). The matrix of coefficients (normalizing ($\ast\ast$) by dividing by $-2$) is:
        \[
            M = \begin{pmatrix} 1/p^2 & 1/(p-1)^2 & 1 \\ 1/p^3 & 1/(p-1)^3 & 0 \end{pmatrix}.
        \]
        Since $p \notin \{0, 1\}$, the rank is $2$ (e.g., the minor of the last two columns is $-1/(p-1)^3 \neq 0$). The conditions are independent, so $\dim(W \cap V_{2p})=1$. Thus, the only problematic subschemes are $2x_1, 2x_2, 2x_3$.
    \end{subcase}

    \begin{subcase}[Hyperelliptic, $\mathrm{char}(k) = 2$]
        The derivative $f'(t)$ is identically zero, so condition ($\ast\ast$) is trivial. $W \cap V_{2p}$ is defined solely by $f(p)=0$, implying $\dim(W \cap V_{2p})=2$. Thus, every subscheme of the form $Z=2p$ is problematic.

        In characteristic 2, the basis functions for the coefficients are $1/t^2$, $1/(t^2+1)$, and $1$. The linear system $W$ depends only on $t^2$ and is the pullback of the complete linear system $H^0(\PP^1_k, \mathcal{O}(2))$ under the Frobenius morphism $\pi=F\colon C \to \PP^1_k$, $t \mapsto t^2$. The fibers of $\pi$ are the subschemes $Z=2p$, coinciding with the set of problematic subschemes. \hfill\proofSymbol
    \end{subcase}
\end{case}
\end{proof}

\begin{remark}[The Image Curve for $r=3$]
    The geometry of the $r=3$ case is confirmed by the implicit equation of the image curve $\phi(C)\subset\PP^2_k = \Proj k[x,y,z]$. In the coordinates established in the proof, the morphism is explicitly given by
    \[
    \phi(t)=[(t-1)^2:t^2:t^2(t-1)^2].
    \]
    The $S_3$ action permuting the marked points $\{0, 1, \infty\}$ on $\PP^1_k$, combined with the intrinsic and symmetric definition of the linear system $W$, ensures the equation for the image curve is symmetric in $x, y, z$. Explicit computation yields:
    \[ (xy+yz+zx)^2 - 4xyz(x+y+z) = 0. \]
    
    If $\mathrm{char}(k) \neq 2$, this is the standard equation for a tricuspidal quartic (see, for instance,~\cite[p.~14]{hui} or~\cite[p.~50]{moe2015rational}) with cusps at $\phi(0) = [1:0:0]$, $\phi(1) = [0:1:0]$, and $\phi(\infty) = [0:0:1]$, consistent with our previous analysis.
    
    If $\mathrm{char}(k) = 2$, the equation reduces to $(xy+yz+zx)^2 = 0$, a double conic. This aligns with the proof's finding that $\phi$ is inseparable and factors through Frobenius. This is necessary, as a reduced, irreducible tricuspidal quartic $C'$ cannot exist in characteristic $2$. Such a curve (a plane quartic of arithmetic genus $3$) is canonically embedded by the Adjunction Formula. Proposition~\ref{prop:quasi-hyperelliptic} implies $C'$ is not quasi-hyperelliptic, but the proof of Lemma \ref{lem:hyp_vs_quasi-hyp_v2} shows it must be hyperelliptic in characteristic 2, a contradiction.
\end{remark}


{
\small
\newcommand{\etalchar}[1]{$^{#1}$}
\providecommand{\bysame}{\leavevmode\hbox to3em{\hrulefill}\thinspace}
\providecommand{\MR}{\relax\ifhmode\unskip\space\fi MR }
\providecommand{\MRhref}[2]{%
  \href{http://www.ams.org/mathscinet-getitem?mr=#1}{#2}
}
\providecommand{\href}[2]{#2}

}

\end{document}